\documentclass[leqno]{amsart}

\usepackage{stmaryrd,graphicx}
\usepackage{amssymb,mathrsfs,amsmath,amscd,amsthm,color}
\usepackage{float}
\usepackage[all,cmtip]{xy}
\DeclareMathAlphabet{\mathpzc}{OT1}{pzc}{m}{it}
\usepackage{amsfonts,latexsym,wasysym}

\newcommand{\ds}{\displaystyle}
\newcommand{\lb}{\langle}
\newcommand{\rb}{\rangle}
\newcommand{\wt}{\widetilde}
\newcommand{\scrd}{\mathscr{D}}

\newcommand{\bba}{\mathbb{A}}
\newcommand{\bbb}{\mathbb{W}}

\newcommand{\bbt}{\mathbb{T}}
\newcommand{\bbr}{\mathbb{R}}
\newcommand{\bbs}{S}
\newcommand{\bbd}{\mathbb{D}}
\newcommand{\bbf}{D}

\newcommand{\bbh}{\mathbb{H}}

\newcommand{\bbk}{W}

\newcommand{\bbz}{\mathbb{Z}}
\newcommand{\bbn}{\mathbb{N}}

\newcommand{\bbq}{\mathbb{Q}}
\newcommand{\mct}{\mathcal{T}}

\newcommand{\hX}{\widehat{X}}
\newcommand{\hx}{\widehat{x}}

\newcommand{\hf}{\widehat{f}}

\newcommand{\tX}{\widetilde{X}}
\newcommand{\cltg}{cl_{T,g}}
\newcommand{\cltgp}{cl_{T',g'}}

\newcommand{\ui}{[0,1]}

\newcommand{\tx}{\widetilde{x}}

\newcommand{\tXh}{\widetilde{X}_{H}}
\newcommand{\txh}{\widetilde{x}_{H}}

\newcommand{\pxxo}{P(X,x_0)}
\newcommand{\pionex}{\pi_{1}(X,x_0)}
\newcommand{\pioney}{\pi_{1}(Y,y_0)}
\newcommand{\pioned}{\pi_{1}(\bbd,d_0)}
\newcommand{\pioneh}{\pi_{1}(\bbh,b_0)}
\newcommand{\pionehp}{\pi_{1}(\bbhp,\bpp)}
\newcommand{\pionet}{\pi_{1}(\bbt,t_0)}
\newcommand{\scru}{\mathscr{U}}

\newcommand{\bbhp}{\mathbb{H}^+}
\newcommand{\finfty}{F}
\newcommand{\cinfty}{C}

\newcommand{\pinfty}{P}
\newcommand{\dinf}{d_{\infty}}

\newcommand{\mcc}{\mathcal{C}}

\newcommand{\bpp}{b_{0}^{+}}
\newcommand{\wild}{\mathbf{w}}
\newcommand{\linf}{\lambda_{\infty}}
\newcommand{\uinf}{\upsilon_{\infty}}

\newtheorem{theorem}{Theorem}
\newtheorem{lemma}[theorem]{Lemma}
\newtheorem{proposition}[theorem]{Proposition}
\newtheorem{corollary}[theorem]{Corollary}
\theoremstyle{definition}\newtheorem{definition}[theorem]{Definition}

\theoremstyle{definition}\newtheorem{example}[theorem]{Example}
\theoremstyle{definition}
\theoremstyle{definition}
\theoremstyle{definition}\newtheorem{remark}[theorem]{Remark}
\theoremstyle{definition}
\numberwithin{theorem}{section}
\begin{document}
\title[characterizations of local properties of fundamental groups]{Test map characterizations of local properties of fundamental groups}
\author{Jeremy Brazas and Hanspeter Fischer}
\date{\today}
\begin{abstract}
Local properties of the fundamental group of a path-connected topological space can pose obstructions to the applicability of covering space theory. A generalized covering map is a generalization of the classical notion of covering map defined in terms of unique lifting properties. The existence of generalized covering maps depends entirely on the verification of the unique path lifting property for a standard covering construction. Given any path-connected metric space $X$, and a subgroup $H\leq\pi_1(X,x_0)$, we characterize the unique path lifting property relative to $H$ in terms of a new closure operator on the $\pi_1$-subgroup lattice that is induced by maps from a fixed ``test" domain into $X$. Using this test map framework, we develop a unified approach to comparing the existence of generalized coverings with a number of related properties.
\end{abstract}
\maketitle
\section{Introduction and Preliminaries}\label{section1}
\subsection{Introduction}
According to classical covering space theory \cite{Spanier66}, if a path-connected space $X$ is locally path connected and semilocally simply connected, then for every subgroup $H\leq\pionex$ of the fundamental group at basepoint $x_0\in X$, there is a covering map $p : (Y,y_0)\to(X,x_0)$ such that $H$ is the image of the induced homomorphism $p_{\#}:\pioney\to\pionex$, and conjugates of $H$ correspond to equivalent covering maps. In particular, a universal covering over $X$ exists. Since all ``small" loops are null-homotopic in $X$, the natural topologies typically considered on $\pionex$ \cite{Braztopgrp,BFqtop,BDLM08,VZ14} are equivalent to the discrete topology, i.e. all subgroups of $\pionex$ are both open and closed. In this sense, the subgroup lattice of $\pionex$ is independent of the local structure of $X$.

When a space $X$ is not semilocally simply connected, the situation is more delicate, since the algebraic structure of $\pionex$ may depend heavily on the local topology of $X$. The variety of possible complications has given rise to the introduction of a number of important properties that a space $X$ might or might not satisfy, including:
\begin{enumerate}
\item Homotopically Hausdorff \cite{BS,CConedim} and its relative version \cite{FZ07},
\item (\textit{transfinite products}) Every homomorphism $f_{\#} : \pioneh \to\pionex$ induced by a map $f:\bbh\to X$ on the Hawaiian earring is uniquely determined by it's values $f_{\#}([\ell_n])$ on the individual loops $\ell_n$ of $\bbh$,
\item Existence of generalized universal and intermediate coverings \cite{FZ07},
\item Homotopically path Hausdorff \cite{FRVZ11} and it's relative version \cite{BFqtop},
\item ($1$-$UV_0$) For every $x\in X$ and every neighborhood $U$ of $x$ there is an open set $V$ in $X$ with $x\in V\subseteq U$ and such that for every map $f : D^2 \to X$ from the unit disk with $f (\partial D^2) \subseteq V$, there is a map $g : D^2\to U$ with $f|_{\partial D^2}=g|_{\partial D^2}$,
\item ($\pi_1$\textit{-shape injectivity}) The canonical homomorphism $\pionex\to\check{\pi}_1(X,x_0)$ to the first shape homotopy group is injective \cite{EK98,FZ07}.
\end{enumerate}
The above properties are not listed in any particular order. Some of them originated from the unpublished notes \cite{Zastrow}. See \cite{FRVZ11} for a diagram comparing properties (1),(3),(4), and (6). Property (2) plays a key role in Eda's remarkable classification of homotopy types of one-dimensional Peano continua according to the isomorphism types of their fundamental groups \cite{Edaonedim}. Property (5) first appeared in \cite{CMRZZ08}.

The primary purpose of this paper is to provide a unified approach to comparing local properties of fundamental groups such as those above. We are particularly motivated by the fact that even when $X$ fails to admit a traditional universal covering, it is often the case that $X$ admits a generalized universal covering in the sense of \cite{FZ07}, which acts in many ways as a suitable replacement. A generalized covering map is characterized only by its lifting properties and need not be a local homeomorphism. For instance, in the case that $X$ is a one-dimensional Peano continuum (e.g. the Hawaiian earring, Sierpi{\'n}ski Carpet, or Menger curve), a generalized universal covering exists, inherits the structure of an $\bbr$-tree, and functions as a generalized Caley graph for the fundamental group $\pionex$ \cite{FZ013caley}. Other spaces which admit generalized universal coverings include subsets of closed surfaces (including all planar sets) \cite{FZ05} and certain trees of manifolds \cite{FischerGuilbault} such as the Pontryagin surface $\prod_{2}$.

For a given space $X$, there may be many intermediate subgroups $H\leq \pionex$ which do not correspond to a covering map but which correspond to a generalized covering as defined under the name $\mathbf{lpc_0}$-covering in \cite{Brazcat}. Some examples of generalized \textit{regular} coverings corresponding to normal subgroups $N\trianglelefteq\pionex$, namely intersections of Spanier groups, appear in \cite{FZ07}. The same normal subgroups also appear in \cite{BDLM08} with an equivalent construction. In the current paper, we consider the existence of generalized coverings relative to an arbitrary subgroup $H$ of $\pionex$.

Roughly speaking, in order to verify any one of the properties (1)-(6), it is necessary to detect the existence of a specific homotopy given a certain, possibly infinite, arrangement of paths. In this paper, we formalize this viewpoint by characterizing the subgroup-relative versions of many of these properties. Our characterizations are stated in terms of set-theoretic closure operations on the subgroup lattice of $\pionex$. For a given property, we identify a based \textit{test space} $(\bbt,t_0)$, a subgroup $T\leq \pionet$, and an element $g\in \pionet$. We call $(T,g)$ a \textit{closure pair} for $(\bbt,t_0)$ and declare a subgroup $H\leq \pionex$ to be $(T,g)$\textit{-closed} if it satisfies the following criterion: for every map $f : (\bbt,t_0)\to(X,x_0)$ such that $f_{\#}(T)\leq H$, we also have $f_{\#}(g)\in H$. Using this test map criterion, we characterize properties (1)-(4) and compare them to a variety of other properties, including some new intermediate properties.

The structure of this article is as follows: The remainder of Section~\ref{section1} contains a summary of our main results, notational conventions, and a brief review of the theory of generalized covering spaces. In Section~\ref{Testmapsection}, we introduce our closure operators for subgroups of fundamental groups and discuss their general theory. In Section~\ref{section3}, we use the Hawaiian earring as a test space to study point-local properties, such as the homotopically Hausdorff property and the transfinite products property. In Section~\ref{section4}, we introduce a dyadic arc space to characterize the homotopically path-Hausdorff property and the existence of generalized covering spaces. In Section~\ref{section5}, we apply results from Sections~\ref{section3} and \ref{section4} to identify new types of non-trivial intermediate generalized coverings. In Section~\ref{section6}, we apply results from Sections~\ref{section3} and \ref{section4} to identify new conditions that are sufficient for the existence of a generalized universal covering. Finally, in Section \ref{section7}, we introduce a new property, called ``transfinite path products'', which serves as a practical intermediate for proving partial converses of well-known implications.

\subsection{Results}\label{resultssection}

The following diagram may serve as a reference for many of the results and definitions in this paper. It connects the relevant properties of a path-connected metric space $X$ and closure properties of a subgroup $H\leq \pionex$. If an extra assumption is required, it appears next to the corresponding arrow. For example, ``$\wild(X)$ t.p.d." denotes the property that the subset of points at which $X$ fails to be semilocally simply connected is a totally path-disconnected subspace of $X$, whereas $[\pi_1,\pi_1]\leq H$ indicates that $H$ contains the commutator subgroup of $\pionex$. For the non-reversibility of some of the implications see Corollary \ref{differencecellandctau}: $(C,c_{\infty})$-closed $\nRightarrow$ $(C,c_{\tau})$-closed, Theorem \ref{differencetheorem}: $(C,c_{\tau})$-closed $\nRightarrow$ $(P,p_{\tau})$-closed, Example \ref{normalsubgroupexample}: $(C,c_{\tau})$-closed $\nRightarrow$ $(D,d_{\infty})$-closed, and Example \ref{subgroupsexample}: $(D,d_{\infty})$-closed $\nRightarrow$ $(S,d_{\infty})$-closed.
\[\xymatrix{
\txt{homotopically\\path Hausdorff rel. $H$\\ \scriptsize{(Def. \ref{hompathhausdef})}} \ar@{<=>}[rr]^-{\scriptsize{\txt{(Thm. \ref{hompathhausdchar})}}} && (\bbs,\dinf)\txt{-closed \scriptsize{(Def. \ref{defsubgroups})}}
\ar@{=>}[d]^-{\text{(Rmk. \ref{sclosedisdclosed})}}\\
 \txt{$p_H:\tXh\to X$ has UPL\\ \scriptsize{(Section \ref{gencovsection})}} \ar@{<=>}[rr]^-{\scriptsize{\txt{(Thm. \ref{uplchar})}}} && (\bbf,\dinf)\text{-closed \scriptsize{(Def. \ref{defsubgroupD})}}
\ar@{=>}@<1ex>[d]^-{\text{(Prop. \ref{transfinitepathproductcomparisonprop})}}\\
\txt{transfinite\\path products rel. $H$\\ \scriptsize{(Def. \ref{tppdef})}} \ar@{<=>}[rr]^-{\text{(Prop. \ref{transfinitepathprodchar})}} && (W,w_{\infty})\text{-closed \scriptsize{(Def. \ref{defW})}} \ar@{=>}@<1ex>[dl(0.76)]^-{\scriptsize{\txt{ $H$ normal\\ (Prop. \ref{transfinitepathproductcomparisonprop})}}} \ar@{=>}@<1ex>[u]^{\scriptsize{\txt{$H$ normal and $\textbf{w}(X)$ t.p.d (Thm. \ref{tpdtheorem})}}} \ar@{=>}[dd]^-{\text{(Prop. \ref{transfinitepathproductcomparisonprop})}} \\
\txt{transfinite\\products rel. $H$\\ \scriptsize{(Def. \ref{tproddef})}} \ar@{<=>}[r]^-{\text{(Prop. \ref{transfiniteproductprop})}} & \txt{$(P,p_{\tau})$-closed\\ \scriptsize{(Def. \ref{defptau})}}
\ar@{=>}@<1ex>[ur(0.8)]^(0.3){\scriptsize{\txt{$H$ normal and $\wild(X)$ discrete \\ (Thm. \ref{discretewildsettheorem})}}}
\ar@{=>}@<0.9ex>[dr]^-{\quad\text{(Prop. \ref{ptautoctau})\quad}} \\
 &&  (\cinfty,c_{\tau})\text{-closed \scriptsize{(Def. \ref{defhplus})}} \ar@{=>}@<0.9ex>[ul]^(0.55){\scriptsize{\txt{$[\pi_1,\pi_1]\leq H$ (Prop. \ref{abelianfactorprop})}\quad\quad\quad}} \ar@{=>}@<1ex>[d]^-{\text{(Prop. \ref{normalsubgroup})}}\\
\txt{homotopically\\Hausdorff rel. $H$\\ \scriptsize{(Def. \ref{homhausdorffdef})}} \ar@{<=>}[rr]^-{\text{(Thm. \ref{homhausrelchar})}} && (\cinfty,c_{\infty})\text{-closed \scriptsize{(Def. \ref{defhplus})}} \ar@{=>}@<1ex>[u]^{H\text{ normal (Prop. \ref{normalsubgroup})}}
}\]
Equipped with this chart, we identify new types of subgroups that correspond to intermediate generalized coverings (Theorem \ref{largeextensionthm} and Corollaries \ref{heuplcorollary}, \ref{heuplcorollary2}) and shed more light on the relative position of the commutator subgroup of $\pioneh$ (Example \ref{commutatorexample}). We also extend the existence of generalized universal coverings for Peano continua with residually n-slender fundamental group to all metric spaces (Corollary \ref{nslendercor}).

Property (5) is not an invariant of homotopy type, but is an important property held by one-dimensional \cite{CF} and planar spaces \cite{FZ05} and is known to imply the homotopically Hausdorff property for metric spaces \cite{CMRZZ08}. We improve the latter result by showing that every metric space with the 1-$UV_0$ property admits a generalized universal covering space (Theorem \ref{uvzeroimpliesupl}).

\subsection{Notational considerations}

Throughout this paper, $X$ will denote a path-connected topological space with basepoint $x_0\in X$ and $H$ will denote a subgroup of the fundamental group $\pionex$. A \textit{map} $f:X\to Y$ means a continuous function and $f_{\#}:\pi_1(X,x)\to\pi_1(Y,y)$ will denote the homomorphism induced by $f$ on fundamental groups when $f(x)=y$.

If $\alpha:\ui\to X$ is a path, then $\alpha^{-}(t)=\alpha(1-t)$ is the reverse path. If $\alpha,\beta:\ui\to X$ are paths such that $\alpha(1)=\beta(0)$, then $\alpha\cdot\beta$ denotes the usual concatenation of paths. More generally, if $\alpha_1,\alpha_2,\dots,\alpha_n$ is a sequence of paths such that $\alpha_{j}(1)=\alpha_{j+1}(0)$ for each $j$, then $\prod_{j=1}^{n}\alpha_j=\alpha_1\cdot \alpha_2\cdot\;\cdots\;\cdot \alpha_n$ is the path defined as $\alpha_j$ on $\left[\frac{j-1}{n},\frac{j}{n}\right]$. The constant path at $x\in X$ is denoted by $c_x$. If $[a,b],[c,d]\subseteq \ui$ and $\gamma:[a,b]\to X$, $\delta:[c,d]\to X$ are maps, we write $\gamma\equiv\delta$ if $\gamma=\delta\circ \phi$ for some increasing homeomorphism $\phi: [a,b]\to [c,d]$; if $\phi$ is linear and if it does not create confusion, we will identify $\gamma$ and $\delta$.

A path $\alpha:[a,b]\to X$ is \textit{reduced} if whenever $a\leq s<t\leq b$ with $\alpha(s)=\alpha(t)$, the loop $\alpha|_{[s,t]}$ is not null-homotopic. Note that a constant path $\alpha:[a,b]\to X$ is reduced if and only if $\alpha$ is degenerate, that is, if $a=b$. If $X$ is a one-dimensional metric space, then every path $\alpha:\ui\to X$ is homotopic (rel. endpoints) within $\alpha(\ui)$ to either a constant path or a reduced path, which is unique up to reparameterization \cite{EdaSpatial}.

For a given space $X$, let $P(X)$ denote the space of paths in $X$ with the compact-open topology generated by the subbasic sets $\langle K,U\rangle =\{\alpha\mid\alpha(K)\subseteq U\}$ where $K\subseteq \ui$ is compact and $U\subseteq X $ is open. A convenient basis for the compact-open topology is given by neighborhoods of the form $\bigcap_{j=1}^{2^n}\left\langle \left[\frac{j-1}{2^n},\frac{j}{2^n}\right],U_j\right\rangle$ where $n\in\bbn$. It is well-known that if $(X,d)$ is a metric space, then the compact-open topology on $P(X)$ agrees with the topology of uniform convergence. For given $x\in X$, let $P(X,x)\subseteq P(X)$ denote the subspace of paths which start at $x$ and let $\Omega(X,x)$ denote the subspace of loops based at $x$.

If $H\leq \pionex$ is a subgroup and $\alpha:\ui\to X$ is a path from $\alpha(0)=x_0$ to $\alpha(1)=x$, let $H^{\alpha}=[\alpha^{-}]H[\alpha]\leq \pi_1(X,x)$ denote the path-conjugate subgroup under basepoint change.

\subsection{Generalized covering maps and the unique path lifting property}\label{gencovsection}

In \cite{FZ07}, Fischer and Zastrow initially define the notion of \textit{generalized universal covering}  and \textit{generalized regular covering} relative to a normal subgroup $N\trianglelefteq \pionex$. The following definition extends these definitions to general subgroups of $\pionex$; it appears in \cite{Brazcat} under the name ``$\mathbf{lpc_{0}}$-covering."

\begin{definition} \label{gencovdef}
A map $p:\hX\to X$ is a \textit{generalized covering map} if
\begin{enumerate}
\item $\hX$ is nonempty, path connected, and locally path connected,
\item for every path-connected, locally path-connected space $Y$, point $\hx\in\hX$, and based map $f:(Y,y)\to (X,p(\hx))$ such that $f_{\#}(\pi_1(Y,y))\leq p_{\#}(\pi_1(\hX,\hx))$, there is a unique map $\hf:(Y,y)\to (\hX,\hx)$ such that $p\circ \hf=f$.
\end{enumerate}
If $p_{\#}(\pi_1(\hX,\hx))$ is a normal subgroup of $\pionex$, we call $p$ a \textit{generalized regular covering map}. If $\hX$ is simply-connected, we call $p$ a \textit{generalized universal covering map}.
\end{definition}
Whenever a generalized covering $p:\hX\to X$ such that $p(\hx)=x_0$ exists, it is characterized up to equivalence by the conjugacy class of the subgroup $H=p_{\#}(\pi_1(\hX,\hx))\leq \pionex$. Since $\ui$ is simply-connected, it is clear that for each point $\hx\in\hX$, every path $\alpha\in P(X,p(\hx))$ has a unique lift $\widehat{\alpha}\in P(\hX,\hx)$ such that $p\circ\widehat{\alpha}=\alpha$. In particular $p$ has the unique path lifting property:

\begin{definition}
A map $f:X\to Y$ has the \textit{unique path lifting property} if whenever $\alpha,\beta:\ui\to X$ are paths with $\alpha(0)=\beta(0)$ and $f\circ \alpha=f\circ \beta$, then $\alpha=\beta$.
\end{definition}

In the attempt to construct generalized covering spaces, one is led to the following standard construction \cite{Spanier66}. We refer to \cite{FZ07} for proofs of the basic properties. Given a subgroup $H\leq \pionex$, let $\tXh=\pxxo/\sim$ where $\alpha\sim \beta$ iff $\alpha(1)=\beta(1)$ and $[\alpha\cdot\beta^{-}]\in H$. The equivalence class of $\alpha$ is denoted $H[\alpha]$. We give $\tXh$ the so-called \textit{standard topology} generated by the sets $$B(H[\alpha],U)=\left\{H[\alpha\cdot\epsilon]\mid\epsilon([0,1])\subseteq U\right\}$$ where $U$ is an open neighborhood of $\alpha(1)$ in $X$. This topology is called the \textit{whisker topology} by some authors \cite{BDLM08,VZ14}. The space $\tXh$ is path connected and locally path connected by construction and if $H[\beta]\in B(H[\alpha],U)$, then $B(H[\alpha],U)=B(H[\beta],U)$. We take the class $\txh=H[c_{x_0}]$ of the constant path to be the basepoint of $\tXh$. In the case that $H=1$ is the trivial subgroup, we simply write $\tX$ and $\tx$ for this space and its basepoint. Let $p_H:\tXh\to X$ denote the endpoint projection map defined as $p_H(H[\alpha])=\alpha(1)$. Since $p_H$ maps $B(H[\alpha],U)$ onto the path-component of $U$ which contains $\alpha(1)$, $p_H$ is an open map if and only if $X$ is locally path connected.

For a point $x\in X$, let $\mct_x$ be the set of all open neighborhoods of $x$ in $X$. For $U\in \mct_x$ and a path $\alpha:[0,1]\to X$ from $x_0$ to $x$, consider the subgroup $\pi(\alpha,U)=\{[\alpha\cdot\delta\cdot\alpha^{-}]\mid\delta\in \Omega(U,\alpha(1))\}\leq \pionex$. Let $\pi(x,U)=\langle \pi(\alpha,U)\mid\alpha(1)=x\rangle $ be the subgroup generated by all subgroups $\pi(\alpha,U)$ in $\pionex$ with $\alpha(1)=x$ and note that $\pi(x,U)$ is a normal subgroup of $\pionex$.

\begin{theorem}\label{coveringtheorem} \cite{Spanier66} Suppose $X$ is locally path connected. Then $p_H:\tXh\to X$ is a covering map in the classical sense if and only if for every $x\in X$, there is a $U\in \mathcal{T}_x$ such that $\pi(x,U)\leq H$,
\end{theorem}
If $X$ is semilocally simply connected, then for every $x$, there is a $U\in\mathcal{T}_x$ such that $\pi(x,U)=1$, so $p_H$ is a covering map for every subgroup $H\leq \pionex$. The standard lifting properties of covering maps illustrate that every covering map is a generalized covering map.

Even when $p_H$ is not a covering map, it may still be a generalized covering. The authors of \cite{FZ07} show that $p_H$ is a generalized covering map whenever it has the unique path lifting property. Indeed, every path $\alpha:(\ui,0)\to (X,x_0)$ has a continuous \textit{standard lift} $\wt{\alpha}_{\mathscr{S}}:\ui\to\tXh$ starting at $\txh$ defined as $\wt{\alpha}_{\mathscr{S}}(t)=H[\alpha_t]$ where $\alpha_t(s)=\alpha(st)$. Thus, to verify whether or not $p_H$ is a generalized covering map, it is necessary and sufficient to verify that for each path $\alpha\in P(X,x_0)$, the standard lift $\wt{\alpha}_{\mathscr{S}}$ is the \textit{only} lift of $\alpha$ starting at $\txh$.

On the other hand, if $p:(\hX,\hx)\to (X,x_0)$ is a generalized covering map such that $p_{\#}(\pi_1(\hX,\hx))=H$, then there is a homeomorphism $q:(\hX,\hx)\to (\tXh,\txh)$ such that $p_H\circ q=p$ \cite{Brazcat}. This means that the topology of any generalized covering space must be equivalent to the standard topology. These observations are summarized in the following lemma.

\begin{lemma}\label{gencovtopologylemma}\cite[Theorem 5.11]{Brazcat}
For any subgroup $H\leq \pionex$, the following are equivalent:
\begin{enumerate}
\item $p_H$ has the unique path lifting property,
\item $p_H$ is a generalized covering map,
\item $(p_H)_{\#}(\pi_1(\tXh,\txh))=H$,
\item $X$ admits a generalized covering $p:(\hX,\hx)\to (X,x_0)$ such that $p_{\#}(\pi_1(\hX,\hx))$ $=H$.
\end{enumerate}
\end{lemma}

\section{Test maps, closure pairs, and closure operators}\label{Testmapsection}
\begin{definition}
Suppose $(\bbt,t_0)$ is a based space, $T\leq \pionet$ a subgroup, and $g\in \pionet$. A subgroup $H\leq \pionex$ is $(T,g)${\it-closed} if for every map $f:(\bbt,t_0)\to (X,x_0)$ such that $f_{\#}(T)\leq H$, we also have $f_{\#}(g)\in H$. We often refer to $\bbt$ as a {\it test space} and $(T,g)$ as a {\it closure pair} for $(\bbt,t_0)$.
\end{definition}
Observe that the set of $(T,g)$-closed subgroups of $\pionex$ is closed under intersection and therefore forms a complete lattice. For any $H\leq \pionex$, we may define the $(T,g)${\it-closure} of $H$ as \[ \cltg(H)=\bigcap\{K\leq \pionex\mid K\text{ is }(T,g)\text{-closed and }H\leq K\}.\]Note that $\cltg(H)=H$ if and only if $H$ is $(T,g)$-closed. Moreover, $\cltg$ is a set-theoretic closure operator on the lattice of subgroups of $\pionex$ in the sense that $H\leq \cltg(H)$, $H\leq K$ implies $\cltg(H)\leq \cltg(K)$, and $\cltg(\cltg(H))=\cltg(H)$. 

The $(T,g)$-closure $\cltg(H)$ contains the subgroup $H'$ generated by $H$ and elements $f_{\#}(g)$ for all maps $f:(\bbt,t_0)\to (X,x_0)$ with $f_{\#}(T)\leq H$. However, $H'$ may be a proper subgroup of $\cltg(H)$ since there could be maps $(\bbt,t_0)\to (X,x_0)$ whose induced homomorphisms map $T$ into $H'$ but not into $H$. See Remark \ref{firststepremark} for an example of such an occurrence.

\begin{proposition}\label{continuity}
If $f:(X,x_0)\to (Y,y_0)$ is a map and $H\leq \pionex$, then $f_{\#}(\cltg(H))\leq \cltg(f_{\#}(H))$.
\end{proposition}
\begin{proof}
First, observe that if $K\leq \pioney$ is $(T,g)$-closed, then so is $f_{\#}^{-1}(K)\leq \pionex$. Let $k\in \cltg(H)$ and $K\leq \pioney$ be any $(T,g)$-closed subgroup such that $f_{\#}(H)\leq K$. It suffices to show $f_{\#}(k)\in K$. Since $f_{\#}^{-1}(K)$ is $(T,g)$-closed and $H\leq f_{\#}^{-1}(K)$, we have $k\in \cltg(H)\leq f_{\#}^{-1}(K)$. Therefore, $f_{\#}(k)\in K$.
\end{proof}

\begin{proposition}\label{comparisonprop}
Suppose $(T,g)$ and $(T',g')$ are closure pairs for $(\bbt,t_0)$ and $(\bbt ',t_{0}')$ respectively. Then the following are equivalent:
\begin{enumerate}
\item $g'\in \cltg(T ')$,
\item for any space $(X,x_0)$ and subgroup $H\leq \pionex$, $H$ is $(T',g')$-closed whenever $H$ is $(T,g)$-closed,
\item for any space $(X,x_0)$ and subgroup $H\leq \pionex$, $\cltgp(H)\leq \cltg(H)$.
\end{enumerate}
\end{proposition}
\begin{proof}
(1) $\Rightarrow$ (2) Suppose $g'\in \cltg(T')$ and that $H\leq \pionex$ is $(T,g)$-closed. Let $k:(\bbt ',t_{0}')\to (X,x_0)$ be a map such that $k_{\#}(T')\leq H$. By Proposition \ref{continuity} and monotonicity, we have $k_{\#}(\cltg(T'))\leq \cltg(k_{\#}(T'))\leq \cltg(H)$. Since $g'\in \cltg(T')$, it follows that $k_{\#}(g')\in  cl_{T,g}(H)=H$. This proves $H$ is $(\bbt ',g')$-closed. (2) $\Rightarrow$ (3) This follows directly from the definition of the closure operator. (3) $\Rightarrow$ (1) First, note that $g'\in \cltgp(T')$. Applying the inequality $\cltgp(H)\leq \cltg(H)$ in the case where $(X,x_0)=(\bbt ',t_{0}')$ and $H=T'$ completes the proof.
\end{proof}

\begin{remark}\label{comparisonpropremark}
If there exists a map $f:(\bbt,t_0)\to (\bbt ',t_{0}')$ such that $f_{\#}(T)\leq T'$ and $f_{\#}(g)=g'$, then we have $g'=f_{\#}(g)\in f_{\#}(\cltg(T))\leq \cltg(f_{\#}(T))\leq \cltg(T')$. Consequently, whenever such a map $f$ exists, we may conclude that all three of the equivalent conditions in Proposition \ref{comparisonprop} hold.
\end{remark}

The closure pairs of primary interest in this paper satisfy the following definition, which implies that the induced closure operator preserves conjugation.

\begin{definition}\label{normaltestpairdef}
A closure pair $(T,g)$ for the test space $(\bbt,t_0)$ is called \textit{normal} if given any space $(X,x_0)$ and subgroup $H\leq \pionex$, $H$ is $(T,g)$-closed if and only if for every path $\alpha\in P(X,x_0)$, $H^{\alpha}$ is a $(T,g)$-closed subgroup of $\pi_1(X,\alpha(1))$.
\end{definition}

\begin{proposition}\label{preserveconjugationprop}
If $(T,g)$ is a normal closure pair, then the closure operator $cl_{T,g}$ commutes with path conjugation, i.e. for all $(X,x_0)$, $H\leq \pionex$, and $\alpha\in P(X,x_0)$, we have $cl_{T,g}(H^{\alpha})=cl_{T,g}(H)^{\alpha}$.
\end{proposition}

\begin{proof}
Suppose $(T,g)$ is a normal closure pair, $H\leq \pionex$, and $\alpha\in P(X,x_0)$. Since $H\leq cl_{T,g}(H)$ and $cl_{T,g}(H)$ is $(T,g)$-closed, $H^{\alpha}\leq cl_{T,g}(H)^{\alpha}$ and $cl_{T,g}(H)^{\alpha}$ is $(T,g)$-closed by normality of $(T,g)$. Hence $cl_{T,g}(H^{\alpha})\leq cl_{T,g}(H)^{\alpha}$. Replacing $H\leq \pionex$ with $H^{\alpha}\leq \pi_1(X,\alpha(1))$ and $\alpha$ with $\alpha^{-}$, we obtain $cl_{T,g}(H)\leq cl_{T,g}(H^{\alpha})^{\alpha^{-}}$. Conjugation with $\alpha$ gives $cl_{T,g}(H)^{\alpha}\leq cl_{T,g}(H^{\alpha})$.
\end{proof}

\begin{corollary}\label{normalclosureprop}
If $(T,g)$ is a normal closure pair and $N\trianglelefteq\pionex$ is a normal subgroup, then $\cltg(N)$ is a normal subgroup of $\pionex$.
\end{corollary}

\begin{corollary}\label{agreeonnormalsubgroupscorollary}
Let $(T,g)$ and $(T',g')$ be normal closure pairs for $(\bbt,t_0)$ and $(\bbt ',t_{0}')$ respectively and let $X$ be a space. Suppose that for every normal subgroup $N\trianglelefteq\pionex$, $N$ is $(T,g)$-closed if and only if $N$ is $(T',g')$-closed. Then the closure operators $cl_{T,g}$ and $cl_{T',g'}$ agree on the normal subgroups of $\pionex$.
\end{corollary}

\begin{proof}
If $N$ is a normal subgroup, then $cl_{T,g}(N)$ contains $N$ and is normal by Corollary \ref{normalclosureprop}. By assumption, $cl_{T,g}(N)$ is $(T',g')$-closed and therefore $cl_{T',g'}(N)\leq cl_{T,g}(N)$. By switching $(T,g)$ and $(T',g')$ and applying the same argument, we see that $cl_{T,g}(N)\leq cl_{T',g'}(N)$.
\end{proof}

\begin{lemma}\label{wellpointedlemma}
If $(\bbt,t_0)$ is a well-pointed space, i.e. if $\{t_0\}\to \bbt$ is a cofibration, then every closure pair $(T,g)$ for $(\bbt,t_0)$ is normal.
\end{lemma}

\begin{proof}
Fix a space $(X,x_0)$ and subgroup $H\leq \pionex$. Suppose $H$ is $(T,g)$-closed and $\alpha\in P(X,x_0)$ is any path. It suffices to show that $H^{\alpha}$ is $(T,g)$-closed. Let $f:(\bbt,t_0)\to (X,\alpha(1))$ be a map such that $f_{\#}(T)\leq H^{\alpha}$. Since $(\bbt,\{t_0\})$ has the homotopy extension property, there is a homotopy $H: \bbt\times [0,1]\to X$ such that $H(d,0)=f(d)$ and $H(t_0,s)=\alpha^{-}(s)$. Let $f_1: (\bbt,t_0)\to (X,x_0)$ be the map $f_1(d)=H(d,1)$. Then $(f_1)_{\#}([\gamma])=[\alpha]f_{\#}([\gamma])[\alpha^{-}]$. It follows that $(f_1)_{\#}(T)=[\alpha]f_{\#}(T)[\alpha^{-}]\leq [\alpha]H^{\alpha}[\alpha^{-}]=H$. By assumption, $(f_1)_{\#}(g)\in H$, which implies $f_{\#}(g)\in H^{\alpha}$.
\end{proof}
\begin{remark}\label{addawhisker}
If $(\bbt,t_0)$ is any space, we may add a ``whisker" by forming the one-point union $\bbt^+=(\bbt,t_0)\vee (\ui,1)$. This new test space will be well-pointed if we take the basepoint $t_{0}^{+}$ to be the image of $0$ in the union. If $j:\bbt\to\bbt^+$ and $\iota:\ui\to \bbt^+$ are the inclusion maps and $(T,g)$ is a closure pair for $(\bbt,t_0)$, then $([\iota]j_{\#}(T)[\iota^{-}],[\iota]j_{\#}(g)[\iota^{-}])$ is a normal closure pair for $(\bbt^+,t_{0}^{+})$
\end{remark}
\begin{remark}\label{homotopyinvarianceremark}
For a given space $X$ and point $x_0\in X$, we may assign to $(X,x_0)$ the lattice of $(T,g)$-closed subgroups of $\pionex$. If $(T,g)$ is a normal closure pair for $(\bbt,t_0)$, then this lattice is an invariant of the homotopy type of $X$. Indeed, if $h:X\to Y$ is a homotopy equivalence and $x_0\in X$ is any point, then the induced isomorphism $h_{\#}:\pionex\to\pi_1(Y,h(x_0))$ satisfies the property that $H\leq \pionex$ is $(T,g)$-closed if and only if $h_{\#}(H)$ is $(T,g)$-closed.
\end{remark}
\begin{definition}
A subgroup $H\leq \pi_1(X,x_0)$ is $(T,g)$\it{-dense} in $\pionex$ if $\cltg(H) $ $=\pionex$.
\end{definition}

\begin{remark}\label{densitysufficientremark}
It is apparently most practical to verify the following condition sufficient for density: if for every $k\in \pionex$, there is a map $f:(\bbt,t_0)\to (X,x_0)$ such that $f_{\#}(T)\leq H$ and $f_{\#}(g)=k$, then $H$ is $(T,g)$-dense in $\pionex$.
\end{remark}
By applying Proposition \ref{continuity}, we obtain the following.
\begin{corollary}\label{denseapplicationcor}
If $T\leq \pionet$ is $(T,g)$-dense in $\pionet$, then $H\leq \pionex$ is $(T,g)$-closed if and only if for every map $f:(\bbt,t_0)\to (X,x_0)$ such that $f_{\#}(T)\leq H$, we have $f_{\#}(\pionet)\leq H$.
\end{corollary}
%
%
\section{The Hawaiian earring as a test space}\label{section3}
Let $C_n\subseteq \bbr^2$ be the circle of radius $\frac{1}{n}$ centered at $\left(\frac{1}{n},0\right)$ and $\bbh=\bigcup_{n\in \bbn}C_n$ be the usual Hawaiian earring space with basepoint $b_0=(0,0)$. For $m\geq 1$, let $\bbh_{\geq n}=\bigcup_{m\geq n}C_m$ denote the smaller copies of the Hawaiian earring, all of which are homeomorphic to $\bbh$. We define some important loops in $\bbh$ as follows:
\begin{enumerate}
\item Let $\ell_n$ define the canonical counterclockwise loop traversing the circle $C_n$. These loops generate the free subgroup $\finfty=\lb [\ell_n]\mid n\in\bbn\rb\leq \pi_1(\bbh,b_0)$.
\item Let $\ell_{\geq m}$ denote the ``infinite concatenation" which is defined as $\ell_{n+m-1}$ on on the interval $\left[\frac{n-1}{n},\frac{n}{n+1}\right]$ and $\ell_{\geq m}(1)=b_0$. As a special case, we denote $\ell_{\infty}=\ell_{\geq 1}$.
\item Let $\mcc\subseteq \ui$ be the standard middle third Cantor set. Write $\ui\backslash \mcc=\bigcup_{n\geq 1}\bigcup_{k=1}^{2^{n-1}} I_{n}^{k}$ where $I_{n}^{k}$ is an open interval of length $\frac{1}{3^n}$ and, for fixed $n$, the sets $I_{n}^{k}$ are indexed by their natural ordering in $\ui$. Let $\ell_{\tau}:\ui\to \bbh$ be the ``transfinite concatenation" defined so that $\ell_{\tau}(\mcc)=b_0$ and $\ell_{\tau}:=\ell_{2^{n-1}+k-1}$ on $\overline{I_{n}^{k}}$ (see Figure \ref{elltaufigure}).
\end{enumerate}

\begin{figure}[H]
\centering \includegraphics[height=0.6in]{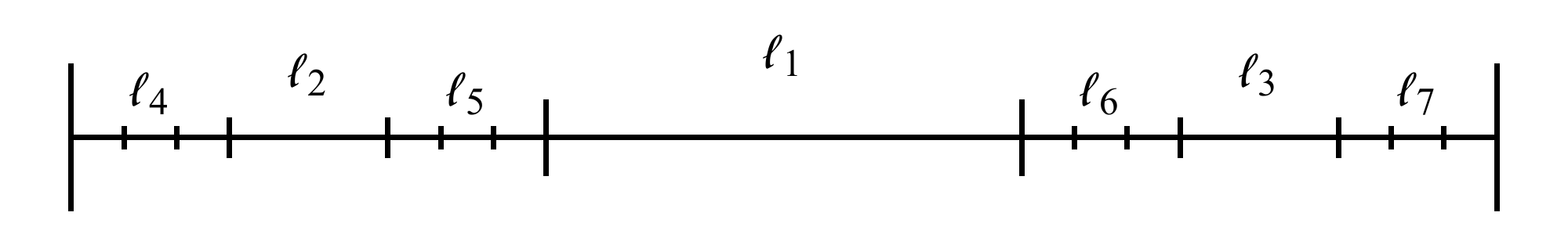}
\caption{\label{elltaufigure}The loop $\ell_{\tau}$}
\end{figure}

The fundamental group $\pioneh$ is uncountable and not free. However, it is naturally isomorphic to a subgroup of an inverse limit of free groups. Let $\bbh_{\leq n}=\bigcup_{m=1}^{n}C_n$ so that $\pi_1(\bbh_{\leq n},b_0)=F_n$ is the group freely generated by the elements $[\ell_1],[\ell_2],\dots,[\ell_n]$. For $n'>n$, the retractions $r_{n',n}:\bbh_{\leq n'}\to \bbh_{\leq n}$ collapsing $\bigcup_{n< m\leq n'}C_m$ to $b_0$ induce an inverse sequence
\[\dots \to F_{n+1}\to F_n\to \dots \to F_2\to F_1\]on fundamental groups, in which $F_{n+1}\to F_n$ deletes the letter $[\ell_{n+1}]$ from a given word. The inverse limit $\check{\pi}_{1}(\bbh,b_0)=\varprojlim_{n}F_n$ is the first shape homotopy group. The retractions $r_n:\bbh\to \bbh_{\leq n}$, which collapse $\bbh_{\geq n+1}$ to $b_0$, induce a canonical homomorphism \[\psi:\pioneh\to \check{\pi}_{1}(\bbh,b_0)\text{ where }\psi([\alpha])=((r_1)_{\#}([\alpha]),(r_2)_{\#}([\alpha]),\dots).\]
Since $\bbh$ a one-dimensional planar Peano continuum, $\psi$ is injective \cite{EK98,FZ05}. Thus a homotopy class $[\alpha]$ is trivial if and only if for every $n\in \bbn$, the projection $(r_n)_{\#}([\alpha])$ as a word in the letters $[\ell_1],[\ell_2],\dots,[\ell_n]$ reduces to the trivial word in $F_n$. Based on the injectivity of $\psi$, we also note that for every $n\in \bbn$, $\pioneh$ may be written as the free product $\pi_1(\bbh_{\leq n},b_0)\ast \pi_1(\bbh_{\geq n+1},b_0)$.

\begin{example}\label{nonnormaltestpair}
Consider the closure pair $(\finfty,[\ell_{\infty}])$ for the Hawaiian earring $\bbh$. Since $[\ell_{\infty}]\notin \finfty$, the subgroup $\finfty$ is not $(\finfty,[\ell_{\infty}])$-closed. On the other hand, if a space $X$ is semilocally simply connected at $x_0$, then for every map $f:(\bbh,b_0)\to (X,x_0)$, there is an $m\geq 1$ such that $f_{\#}(\pi_1(\bbh_{\geq m},b_0))=1$. Since every non-trivial element of $\pioneh$ may be written as a finite product of elements of $\pi_1(\bbh_{\geq m},b_0)$ and the free group $\langle [\ell_1],[\ell_2],\dots,[\ell_{m-1}]\rangle$, it is clear that every subgroup of $\pi_1(X,x_0)$ is $(\finfty,[\ell_{\infty}])$-closed. For example, if $(\bbhp,\bpp)$ is the space obtained by attaching a whisker as in Remark \ref{addawhisker} and $j:\bbh\to \bbhp$ and $\iota:\ui\to \bbhp$ are the inclusions, then $[\iota]j_{\#}(F)[\iota^{-}]$ is $(\finfty,[\ell_{\infty}])$-closed. However, the map $h:\bbhp\to \bbh$ collapsing the attached whisker to $b_0$ is a homotopy equivalence satisfying $h_{\#}([\iota]j_{\#}(F)[\iota^{-}])=F$. According to Remark \ref{homotopyinvarianceremark}, $(\finfty,[\ell_{\infty}])$ cannot be a normal closure pair.
\end{example}

\begin{definition}
An infinite sequence of paths $\alpha_1,\alpha_2,\dots$ such that $\alpha_{n}(1)=\alpha_{n+1}(0)$ for each $n\geq 1$ is \textit{null at} $x\in X$ if for every neighborhood $U$ of $x$, there is an $N$ such that $\alpha_n([0,1])\subseteq U$ for all $n\geq N$. The \textit{infinite concatenation} of such a null-sequence is the path $\alpha=\prod_{n=1}^{\infty}\alpha_n$ whose restriction to $\left[\frac{n-1}{n},\frac{n}{n+1}\right]$ is $\alpha_n$ and $\alpha(1)=x$.
\end{definition}
Note that a sequence $\alpha_n:(\ui,\{0,1\})\to (X,x)$ of loops is null at $x$ if and only if there is a map $f:(\bbh,b_0)\to(X,x)$ such that $f\circ \ell_n=\alpha_n$, in which case $f\circ \ell_{\geq m}=\prod_{n= m}^{\infty}\alpha_n$.
\subsection{The closure pairs $(C,c_{\infty})$ and $(C,c_{\tau})$ and the homotopically Hausdorff property}
\begin{definition}\label{homhausdorffdef} (Homotopically Hausdorff relative to a subgroup \cite{FZ07})
We call $X$ \textit{homotopically Hausdorff relative to a subgroup} $H\leq \pionex$ if for every $x\in X$ and every path $\alpha:\ui\to X$ from $\alpha(0)=x_0$ to $\alpha(1)=x$, only the trivial right coset of $H^{\alpha}=[\alpha^{-}]H[\alpha]$ in $\pi_1(X,x)$ has arbitrarily small representatives, that is, if for every $g\in \pi_1(X,x)\backslash H^{\alpha}$, there is an open set $U\in\mct_x$ such that there is no loop $\delta:(\ui,\{0,1\})\to (U,x)$ with $H^{\alpha}g=H^{\alpha}[\delta]$.

The space $X$ is \textit{homotopically Hausdorff} if it is homotopically Hausdorff relative to the trivial subgroup $H=1$.
\end{definition}
\begin{remark}\label{putdifferently}
Put differently, $X$ is homotopically Hausdorff relative to $H$ if and only if for every $x\in X$ and every path $\alpha$ from $x_0$ to $x$, we have \[\bigcap_{U\in\mct_x}H\pi(\alpha,U)=H.\]
\end{remark}
To characterize the homotopically Hausdorff property, we apply the construction in Remark \ref{addawhisker}.

\begin{definition}\label{defhplus}
Let $\bbh^+=\bbh\cup ([-1,0]\times \{0\})$. We take $\bpp=(-1,0)$ to be the basepoint of $\bbh^+$ so that $(\bbhp,\bpp)$ is well-pointed (see Figure \ref{hplusimage}). If $\iota:\ui\to \bbhp$ is the path $\iota(t)=(t-1,0)$, let
\begin{enumerate}
\item $c_n=[\iota\cdot \ell_n\cdot \iota^{-}]$,
\item $\cinfty=\lb c_n \mid n\in\bbn\rb\leq \pi_1(\bbhp,\bpp)$,
\item $c_{\infty}=[\iota\cdot\ell_{\infty}\cdot\iota^{-}]$,
\item and $c_{\tau}=[\iota\cdot\ell_{\tau}\cdot\iota^{-}]$.
\end{enumerate}
\end{definition}
\begin{figure}[H]
\centering \includegraphics[height=1.2in]{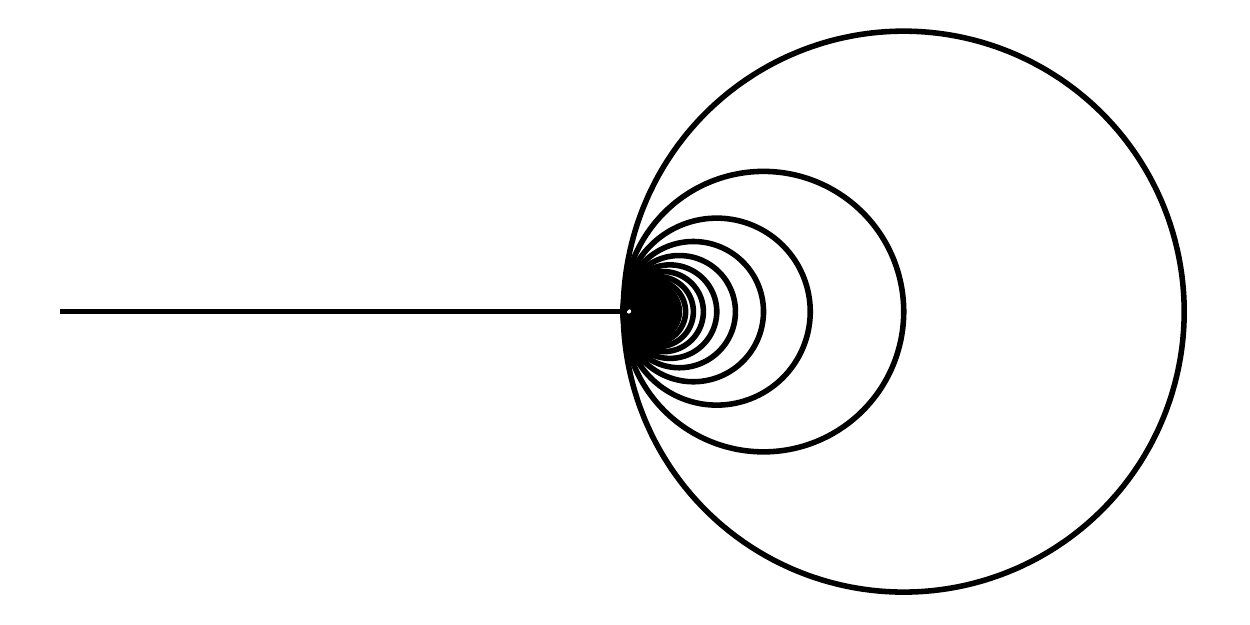}
\caption{\label{hplusimage}The space $\bbhp$}
\end{figure}
\begin{theorem}\label{homhausrelchar}
If $X$ is homotopically Hausdorff relative to $H\leq \pioneh$, then $H$ is $(\cinfty,c_{\infty})$-closed. The converse holds if $X$ is first countable; for instance, if $X$ is metrizable.
\end{theorem}
\begin{proof}
Suppose $H\leq \pionex$ is not $(\cinfty,c_{\infty})$-closed and recall the characterization of the homotopically Hausdorff property from Remark \ref{putdifferently}. Then there exists a map $f:(\bbhp,\bpp)\to (X,x_0)$ such that $f_{\#}(\cinfty)\leq H$ and $f_{\#}(c_{\infty})\notin H$. Set $x=f(b_0)$ and $\alpha=f\circ \iota$. Choose any $U\in \mathcal{T}_{x}$. By the continuity of $f$, there exists $m\geq 2$ such that the image of $f\circ \ell_{\geq m}$ lies in $U$. Since $f_{\#}(c_n)\in H$ for all $n\geq 1$, we have $f_{\#}(c_{\infty})=f_{\#}(c_1c_2\cdots c_{m-1})[\alpha\cdot(f\circ \ell_{\geq m})\cdot\alpha^{-}]\in H\pi(\alpha,U)$. Therefore, $X$ cannot be homotopically Hausdorff relative to $H$.

For the converse, suppose $X$ is first countable and is not homotopically Hausdorff relative to $H$. Then there exists a path $\alpha$ from $x_0$ to $x$ and element $g\in \left(\bigcap_{U\in \mathcal{T}_x}H\pi(\alpha,U)\right)\backslash H$. Let $U_1\supseteq U_2\supseteq U_3\supseteq \cdots$ be a countable neighborhood base at $\alpha(1)$. For each $n\geq 1$, find a loop $\delta_{n}\in  \Omega(U_n,\alpha(1))$ such that $g\in H[\alpha\cdot \delta_{n}\cdot\alpha^{-}]$. Since $g\notin H$, observe that $[\alpha\cdot \delta_{n}\cdot\alpha^{-}]\notin H$ for each $n\geq 1$. Define a map $f:(\bbhp,\bpp)\to (X,x_0)$ so that $f\circ \iota=\alpha$ and $f\circ \ell_{n}= \delta_n\cdot\delta_{n+1}^{-}$. Notice $f_{\#}(c_n)=[\alpha\cdot \delta_{n}\cdot\alpha^{-}][\alpha\cdot \delta_{n+1}\cdot\alpha^{-}]^{-1}\in Hgg^{-1}H=H$ for all $n\geq 1$. Therefore $f_{\#}(C)\leq H$. On the other hand, since the infinite concatenation $\prod_{n\geq 2}(\delta_{n}^{-}\cdot\delta_{n})$ is null-homotopic,
\[f_{\#}(c_{\infty})=[\alpha\cdot \delta_{1}\cdot\alpha^{-}][\alpha]\left[\prod_{n\geq 2}(\delta_{n}^{-}\cdot\delta_{n})\right][\alpha^{-}]=[\alpha\cdot \delta_{1}\cdot\alpha^{-}]\notin H.\]
We conclude that $H$ is not $(\cinfty,c_{\infty})$-closed.
\end{proof}
\begin{proposition}\label{ctauisdense}
$\cinfty$ is $(\cinfty,c_{\tau})$-dense in $\pionehp$.
\end{proposition}
\begin{proof}
We verify the property stated in Remark \ref{densitysufficientremark}, which is sufficient for density. If $\mcc$ is the middle third cantor set, let $I_m$ be the unique component of $\ui\backslash \mcc$ on which $\ell_{\tau}$ is defined as the path $\ell_m$. An arbitrary element $h\in \pi_1(\bbhp,\bpp)$ may be represented by a loop of the form $\iota\cdot \alpha\cdot \iota^{-}$ where $\alpha:\ui\to \bbh$ is a loop based at $b_0$. Let $J_k$, $k\in\bbn$ be the components of $[0,1]\backslash\alpha^{-1}(b_0)$, of which we may assume there are infinitely many, and choose a function $\phi :\bbn \to\bbn$ such that the collection $\{J_k \mid k \in\bbn\}$ has the same order type as $\{I_{\phi(k)} \mid k \in\bbn\}$. For each $k$, there is an $n_k$, such that $\alpha$ restricted to $\overline{J_k}$ is homotopic within $C_{n_k}$ (rel. endpoints) to either $\ell_{n_k}$, $\ell_{n_k}^{-}$, or a constant. Accordingly, define $\beta:\ui\to\bbh$ on $\overline{ I_{\phi(k)} }$ as $\ell_{n_k}$, $\ell_{n_k}^{-}$, or constant, respectively, and equal to $b_0$ elsewhere. Then $[\beta]=[\alpha]\in \pi_1(\bbh,b_0)$. Define a map $f:\bbhp\to\bbhp$ so that $f\circ\iota=\iota$ and $f\circ \ell_n \equiv \beta|_{\overline{ I_n}}$ for all $n$. Then $f_{\#}(C)\leq C$ and $f_{\#}(c_{\tau})=[\iota\cdot\beta\cdot\iota^-]=h$.
\end{proof}
\begin{proposition}\label{normalsubgroup}
If a subgroup $H\leq \pionex$ is $(\cinfty,c_{\tau})$-closed, then $H$ is $(\cinfty,c_{\infty})$-closed. The converse holds if $H$ is a normal subgroup of $\pionex$. In particular, the closure operators $cl_{C,c_{\infty}}$ and $cl_{C,c_{\tau}}$ agree on normal subgroups.
\end{proposition}
\begin{proof}
Since $\cinfty$ is $(\cinfty,c_{\tau})$-dense, we have $c_{\infty}\in \pi_1(\bbhp,\bpp)=cl_{\cinfty,c_{\tau}}(\cinfty)$ and we may apply Proposition \ref{comparisonprop}. For the partial converse, suppose $N$ is a normal, $(\cinfty,c_{\infty})$-closed subgroup of $\pionex$ and $f:\bbhp\to X$ is a map such that $f_{\#}(\cinfty)\leq N$. Recall Corollary \ref{denseapplicationcor}. To obtain a contradiction, suppose there is a loop $\gamma$ in $\bbhp$ such that $f_{\#}([\gamma])\notin N$. Set $\alpha=f\circ \iota$. We may assume $\gamma=\iota\cdot \beta\cdot \iota^{-}$ where $\beta$ is a reduced loop in $\bbh$. Notice that for each $n\geq 1$, $\beta$ is homotopic to a loop of the form $a_1\cdot b_1\cdot a_2\cdot b_2\cdot \;\cdots\;\cdot a_m\cdot b_m$ where $a_j$ is constant or is in $ \{\ell_{1}^{\pm},\ell_{2}^{\pm},\dots,\ell_{n-1}^{\pm}\}$ and $b_j$ is a (possibly constant) loop with image in $\bbh_{\geq n}$. Therefore
$f_{\#}([\gamma]) = f_{\#}\left(\prod_{j=1}^{m}[\iota\cdot a_j\cdot \iota^{-}][\iota\cdot b_j\cdot \iota^{-}]\right)$ is an element of \[\prod_{j=1}^{m}f_{\#}(\cinfty)f_{\#}([\iota\cdot b_j\cdot \iota^{-}])\leq  \prod_{j=1}^{m}Nf_{\#}([\iota\cdot b_j\cdot \iota^{-}])\]
and since $N$ is normal,
\[f_{\#}([\gamma])\in\prod_{j=1}^{m}Nf_{\#}([\iota\cdot b_j\cdot \iota^{-}])= N[\alpha\cdot \delta_n\cdot\alpha^{-}]\]
where $\delta_{n}=f\circ \prod_{j=1}^{m} b_j$ has image in $\bbh_{\geq n}$. Let $h:\bbhp\to X$ be the map defined by $h\circ \iota=\alpha$ and $h\circ \ell_{n}=\delta_{n}\cdot\delta_{n+1}^{-}$. Notice that \[h_{\#}(c_n)=[\alpha\cdot \delta_{n}\cdot \alpha^{-}][\alpha\cdot \delta_{n+1}^{-}\cdot \alpha^{-}]\in Nf_{\#}([\gamma])f_{\#}([\gamma])^{-1}N=N.\] Since $h_{\#}(\cinfty)\leq N$ and $N$ is $(\cinfty,c_{\infty}$)-closed, we have $h_{\#}(c_{\infty})\in N$. But
\[
h_{\#}(c_{\infty}) = [\alpha]\left[\prod_{n=1}^{\infty}\delta_{n}\cdot\delta_{n+1}^{-}\right][\alpha^{-}]= [\alpha]\left[\delta_{1}\right][\alpha^{-}]\in N,\]
which contradicts the fact that $[\alpha]\left[\delta_{1}\right][\alpha^{-}]\in Nf_{\#}([\gamma])$ and $f_{\#}([\gamma])\notin N$.

The final statement of the proposition follows from Corollary \ref{agreeonnormalsubgroupscorollary}.
\end{proof}
Let $\bbh\bba$ be the Harmonic Archipelago constructed by attaching 2-cells to $\bbh$ along the loops $\ell_{n}\cdot\ell_{n+1}^{-}$ \cite{BS}.
\begin{corollary}\label{homhauschar}
If $X$ is first countable, then the following are equivalent:
\begin{enumerate}
\item $X$ is homotopically Hausdorff,
\item every map $f:\bbh\to X$ such that $f_{\#}(\finfty)=1\leq \pi_1(X,f(b_0))$ induces the trivial homomorphism on $\pi_1$,
\item every map $f:\bbh\bba\to X$ induces the trivial homomorphism on $\pi_1$.
\end{enumerate}
\end{corollary}
\begin{proof}
(1) $\Leftrightarrow$ (2) follows from Corollary \ref{denseapplicationcor} and Propositions \ref{ctauisdense} and \ref{normalsubgroup}. For (2) $\Rightarrow$ (3), observe that the inclusion $\bbh\to\bbh\bba$ induces a surjection on $\pi_1$. For (3) $\Rightarrow$ (2), notice that $f:\bbh\to X$ extends to $f:\bbh\bba\to X$ if $f_{\#}(F)=1$.
\end{proof}
\begin{example}\label{commutatorexample} (A closure of a commutator subgroup) Set $G=\pioneh$ and let $[G,G]$ denote the commutator subgroup of $G$. Recall that the abelianization $G/[G,G]$ is isomorphic to the first singular homology group $H_1(\bbh)$. Observe that $[G,G]$ is not $(C,c_{\infty})$-closed since infinite products of commutators such as $\left[\prod_{n=1}^{\infty}\left(\ell_{2n-1}\cdot\ell_{2n}\cdot\ell_{2n-1}^{-}\cdot\ell_{2n}^{-}\right)\right]$ are not elements of $[G,G]$; see Lemma 3.6 of \cite{Edasingularonedim16}.

Consider the infinite torus $T=\prod_{n\geq 1}C_n$ and the canonical embedding $m:\bbh\to T$ induced by the retractions $R_n:\bbh\to C_n$. Since each $C_n$ is homotopically Hausdorff, $K=\ker(m_{\#})=\bigcap_{n\geq 1}\ker((R_n)_{\#})$ is $(C,c_{\infty})$-closed. Since $\pi_1(T,x_0)\cong \prod_{n\geq 1}\bbz$ is abelian, $[G,G]\leq K$ and thus $cl_{C,c_{\infty}}([G,G])\leq K$. Suppose $\alpha$ is a loop in $\bbh_{\geq n}$ such that $[\alpha]\in K$. Then we may write $[\alpha]=\prod_{i=1}^{m}\left([\delta_i][\ell_{n}]^{\epsilon_i}\right)$ where the loop $\delta_i$ has image in $\bbh_{\geq n+1}$ and $\sum_{i=1}^{m}\epsilon_i=0$. Since $\alpha$ and $\beta=\prod_{i=1}^{m}\delta_i$ are homologous in $\bbh_{\geq n}$, there is a loop $\gamma$ in $\bbh_{\geq n}$ such that $[\gamma]\in [G,G]$ and $[\alpha]=[\gamma][\beta]$. Note $[\beta]\in K$. Thus given any element $[\alpha]\in K$, we may inductively construct loops $\gamma_n$ in $\bbh_{\geq n}$ and $\beta_n$ in $\bbh_{\geq n+1}$ such that $[\gamma_n]\in [G,G]$, $[\beta_n]\in K$, and $[\alpha]=\left(\prod_{i=1}^{n}[\gamma_i]\right)[\beta_n]$. By composing this sequence of shrinking homotopies, we see that $[\alpha]=\left[\prod_{n\geq 1}\gamma_n\right]$ and thus $[\alpha]$ is an infinite product of elements of $[G,G]$. We conclude that $cl_{C,c_{\infty}}([G,G])= K$. See also Corollary \ref{heuplcorollary2}.

In light of Remark \ref{quotienttopremark} below, this computation sharpens the result in \cite{Corsonhomology} that $K$ is the topological closure of $[G,G]$ in $G$ when $G$ is equipped with its natural quotient topology.
\end{example}
Before presenting our next example, we prove a technical lemma and a corollary.
\begin{lemma}\label{cancellinglemma}
Let $\alpha:(\ui,0)\to (X,x_0)$ be a reduced path in a one-dimensional metric space $X$, $\gamma:\ui\to X$ be a reduced loop based at $\alpha(1)$, and $\eta$ be a reduced representative of $[\alpha\cdot \gamma\cdot \alpha^{-}]$. Then there exist $s,t\in [0,1]$ such that $\eta|_{[0,s]}\equiv \alpha|_{[0,t]}$ and $\alpha([t,1])\subseteq \gamma([0,1])$.
\end{lemma}

\begin{proof} Replacing $X$ with the image of $\alpha\cdot \gamma\cdot \alpha^{-}$, we may assume that $X$ is a Peano continuum. We examine how to obtain $\eta$ from $\alpha\cdot\gamma_n\cdot\alpha^{-}$. This can be done, for example, using the uniquely arcwise connected generalized universal covering space of $X$ \cite{FZ07}, where a path lifts to an arc if and only if it is reduced.

A (maximal) initial segment of $\gamma$ might agree (after reparameterization) with the reverse of a terminal segment of $\alpha$, or a terminal segment of $\gamma$ might agree with a terminal segment of $\alpha$, or both. If, after canceling such segments, $\gamma$ has not been entirely canceled, we have arrived at $\eta$. Otherwise, we are in one of the following segmentation scenarios:
\begin{enumerate}
\item $\alpha=\beta_1\cdot \beta_2$, $\gamma=\beta_2^{-}$;
\item $\alpha=\beta_1\cdot \beta_2$, $\gamma=\beta_2$;
\item $\alpha=\beta_1\cdot \beta_2\cdot \beta_3$, $\gamma=\beta_3^{-}\cdot \beta_2^{-}\cdot\beta_3$;
\item $\alpha=\beta_1\cdot \beta_2\cdot \beta_3$, $\gamma=\beta_3^{-}\cdot \beta_2\cdot \beta_3$.
\end{enumerate}
Canceling with matching subsegments of $\gamma$, we arrive at $\beta_1\cdot \beta_2^\pm\cdot \beta_1^{-}$. If $\beta_1$ is empty, we have arrived at $\eta$; otherwise, we continue. Since $\beta_1\cdot \beta_2$ is a subpath of $\alpha$, it is reduced. Hence, further cancellation in $\beta_1\cdot \beta_2^\pm\cdot \beta_1^{-}$ may occur only at one location. If not all of $\beta_{2}^{\pm}$ cancels in $\beta_1\cdot\beta_{2}^{\pm}\cdot \beta_{1}^{-}$, we have arrived at $\eta$. Otherwise, we consider the possibility of applying the four segmentation scenarios above to $\beta_1\cdot\beta_{2}^{\pm}\cdot \beta_{1}^{-}$. Scenarios (3) and (4) cannot occur since they would introduce an inverse pair within the reduced subpath $\beta_1\cdot \beta_2$ or $\beta_{2}^{-}\cdot\beta_{1}^{-}$. Therefore, we may cancel only according to scenario (1) or (2) above. Doing so, we arrive at $\delta_1\cdot \beta_2^{\pm} \cdot \delta_1^{-}$ where $\beta_1=\delta_1\cdot \beta_2$ is reduced. Unless we arrive at $\eta$, we may repeatedly apply scenario (1) or (2) to reduce $\delta_n\cdot \beta_2^{\pm} \cdot \delta_{n}^{-}$ to $\delta_{n+1}\cdot \beta_2^{\pm} \cdot \delta_{n+1}^{-}$ where $\delta_{n}=\delta_{n+1}\cdot\beta_2$ are reduced subpaths of $\beta_1$. However, this cannot occur infinitely often since $\beta_{2}$ may appear as a separate subpath of $\beta_1$ only finitely many times. When the finite procedure terminates, we arrive at $\eta$.

In any of the above situations, observe that $\eta$ contains an initial segment $\alpha|_{[0,t]}$ such that $\alpha([t,1])\subseteq \gamma([0,1])$.
\end{proof}

\begin{corollary}\label{cancellationsequence}
Let $\alpha:(\ui,0)\to (X,x_0)$ be a reduced path in a one-dimensional metric space $X$, $\gamma_n:\ui\to X$ be a null-sequence of reduced loops based at $\alpha(1)$, and $\eta_n$ be a reduced representative of $[\alpha\cdot \gamma_n\cdot \alpha^{-}]$. Then for every $0<t<1$, there exists an $N$ and $0<s<1$ such that $\eta_{N}|_{[0,s]}\equiv \alpha|_{[0,t]}$.
\end{corollary}

\begin{proof}
By Lemma \ref{cancellinglemma}, for each $n\in\bbn$, we have $s_n,t_n\in [0,1]$ such that $\eta_n|_{[0,s_n]}\equiv \alpha|_{[0,t_n]}$ and $\alpha([t_n,1])\subseteq \gamma_n([0,1])$. Since the diameter of $\gamma_n$ converges to $0$ as $n\to \infty$, for given $0<t<1$, we may find $N$ such that $t<t_N\leq 1$. Since $\alpha|_{[0,t_N]}$ is an initial segment of $\eta_N$,  so is  $\alpha|_{[0,t]}$, i.e. $\eta_{N}|_{[0,s]}\equiv \alpha|_{[0,t]}$ for some $0<s<1$.
\end{proof}
\begin{example}\label{countablecutpointsexample} (Countable cut-points)
For any closed subset $A$ of $X$ and point $x_0\in X$, define $CCP(X,A,x_0)$ to be the subgroup\[\{[\alpha]\in\pionex\mid\alpha^{-1}(A)\text{ is countable or }\alpha\text{ is constant}\}\]of $\pionex$. Note that $CCP(X,\emptyset,x_0)=\pionex$ and $CCP(X,X,x_0)=1$. It is well-known that a closed subspace $B$ of $\bbr$ is countable if and only if $B$ is scattered in the sense that every nonempty subspace of $B$ contains an isolated point. Hence $CCP(X,A,x_0)$ may also be described as the subgroup $\{[\alpha]\in\pionex\mid\alpha^{-1}(A)\text{ is scattered or }\alpha\text{ is constant}\}$.

The special case of $CCP(\bbh,\{b_0\},b_0)\leq \pioneh$ has been studied in a variety of contexts. In \cite{BS}, it is described as the subgroup of ``countable order types." By applying the results in Section 5 of \cite{CC00}, we may also identify $CCP(\bbh,\{b_0\},b_0)$ with the group $Scatter(\aleph_{0})$ of \cite{CC00} and the group $Sc$ of \cite{EK99subgroups}. In both papers, $CCP(\bbh,\{b_0\},b_0)$ is shown to be isomorphic to an uncountable free group.
\end{example}
\begin{proposition}\label{countablecutexample}
If $X$ is a one-dimensional metric space and $A\subseteq X$ is closed, then $CCP(X,A,x_0)$ is $(\cinfty,c_{\infty})$-closed.
\end{proposition}
\begin{proof}
First, note that if $\alpha:\ui\to X$ is a path such that $\alpha^{-1}(A)$ is countable, and $\delta$ is the reduced representative of $\alpha$, then $\delta^{-1}(A)$ is also countable. Suppose $f:(\bbhp,\bpp)\to(X,x_0)$ is a map such that $f_{\#}(c_n)\in CCP(X,A,x_0)$ for all $n\in\bbn$. Set $\alpha=f\circ\iota$, $\gamma_n=f\circ \ell_n$, and $\gamma=\prod_{n=1}^{\infty}\gamma_n$. Since $X$ is one-dimensional, we may assume the path $\alpha$ and each loop $\gamma_n$ is either reduced or constant. If all $\gamma_n$ are constant, then $f_{\#}(c_{\infty})=[\alpha\cdot\gamma\cdot\alpha^{-}]=1$. On the other hand, if $n_1<n_2<n_3<\cdots$ is the entire sequence of indices such that $\gamma_{n_k}$ is nonconstant, then by reducing constant subpaths, we have $[\gamma]=\left[\prod_{k=1}^{\infty}\gamma_{n_k}\right]$. Since we seek to show $f_{\#}(c_{\infty})=[\alpha\cdot\gamma\cdot\alpha^{-}]\in CCP(X,A,x_0)$, we may assume that $\gamma_n$ is reduced and nonconstant for every $n$.

If $\alpha$ is constant, then $f_{\#}(c_n)=[\gamma_n]\in CCP(X,A,x_0)$, which implies that $\gamma_{n}^{-1}(A)$ is countable for each $n\in \bbn$. By construction of $\gamma$, it follows that $\gamma^{-1}(A)$ is countable and thus $f_{\#}(c_{\infty})\in CCP(X,A,x_0)$.

Finally, suppose $\alpha$ is not constant. Let $\eta_n$ be the reduced representative of $[\alpha\cdot\gamma_n\cdot\alpha^{-}]$. If $\alpha^{-1}(A)$ is uncountable, then there is a $0<t<1$ such that $\alpha|_{[0,t]}^{-1}(A)$ is uncountable. By Corollary \ref{cancellationsequence}, there is an $n$ such that $\alpha|_{[0,t]}$ is an initial segment of $\eta_n$; a contradiction of the assumption that $[\eta_n]=[\alpha\cdot\gamma_n\cdot\alpha^{-}]=f_{\#}(c_n)\in CCP(X,A,x_0)$. Therefore, $\alpha^{-1}(A)$ must be countable. Since $\eta_{n}^{-1}(A)$ is also countable, we see that the preimage of $A$ under $\alpha^{-}\cdot\eta_{n}\cdot\alpha$ is countable. Since $\gamma_n$ is the reduced representative of $\alpha^{-}\cdot\eta_{n}\cdot\alpha$, $\gamma_{n}^{-1}(A)$ is countable for each $n\in\bbn$. Therefore, the preimage of $A$ under $\alpha\cdot \gamma\cdot \alpha^{-}$ is countable, proving that $f_{\#}(c_{\infty})\in CCP(X,A,x_0)$.
\end{proof}
\begin{corollary}\label{differencecellandctau}
$c_{\tau}\notin cl_{C,c_{\infty}}(C)$.
\end{corollary}
\begin{proof}
Recall that if $\alpha:\ui\to\bbhp$ is a reduced representative of an element of $CCP(\bbhp,\{b_0\},b_{0}^{+})$, then $\alpha^{-1}(b_0)$ is countable. The loop $\beta=\iota\cdot \ell_{\tau}\cdot\iota^{-}$ is reduced and $\beta^{-1}(b_0)=\{0,1\}\cup K$ where $K$ is a Cantor set in $[1/3,2/3]$. Hence $c_{\tau}\notin CCP(\bbhp,\{b_0\},b_{0}^{+})$. Since $C\leq CCP(\bbhp,\{b_0\},b_{0}^{+})$ and $CCP(\bbhp,\{b_0\},b_{0}^{+})$ is $(C,c_{\infty})$-closed by Proposition~\ref{countablecutexample}, it follows that $c_{\tau}\notin cl_{C,c_{\infty}}(C)$.
\end{proof}
In view of Proposition \ref{comparisonprop}, Corollary \ref{differencecellandctau} illustrates a difference in closure operators: $cl_{C,c_{\infty}}\neq cl_{C,c_{\tau}}$. Specifically, $cl_{C,c_{\infty}}(C)$ is a proper subgroup of $cl_{C,c_{\tau}}(C)$.
\begin{remark}\label{firststepremark}
We caution that, in general, the closure $cl_{T,g}(H)$ cannot be obtained by ``adding limit elements'' in one single step. Rather, as with the sequential closure operator in general topology, one may have to inductively add elements to calculate a given closure.
We illustrate this phenomenon with the example of $cl_{C,c_{\infty}}(C)$.

To simplify the analysis, we replace $cl_{C,c_{\infty}}(C)$ with $cl_{C,c_{\infty}}(F)$. This is justified by Remark~\ref{homotopyinvarianceremark}, using the homotopy equivalence $\bbhp\to \bbh$ that collapses the whisker.

Consider the subgroup $F'$ of $cl_{C,c_{\infty}}(F)$ generated by $F$ and elements $f_{\#}(c_{\infty})$ for all maps $f:(\bbhp,\bpp)\to (\bbh,b_0)$ with $f_{\#}(C)\leq F$.  We will show that $F'$ is a proper subgroup of $cl_{C,c_{\infty}}(F)$.

 First, we claim that $F'= \langle F\cup \{\left[\prod_{i=1}^{\infty}a_i\right]\mid a_i\in\{\ell_{m}^{\pm}\mid m\in\bbn\}\}\rangle$. Clearly, the latter group is contained in $F'$. To verify the converse, let $f:(\bbhp,\bpp)\to (\bbh,b_0)$ be a map with $f_{\#}(C)\leq F$. We may assume that $f(b_0)=b_0$, $\alpha=f\circ \iota$ is a reduced loop in $\bbh$ based at $b_0$, and $\gamma_n=f\circ \ell_n$ is a null-sequence of reduced loops in $\bbh$ based at $b_0$.
Since $f_{\#}(c_n)=[\alpha\cdot \gamma_n\cdot \alpha^-]\in F$, it follows from Corollary~\ref{cancellationsequence} applied to $(X,x_0)=(\bbh,b_0)$, that $\alpha\equiv a_1 \cdot a_2 \cdot a_3\cdot\; \cdots $ with either finitely many or infinitely many  $a_i\in \{\ell_{m}^{\pm}\mid m\in \bbn\}$. Note that $\gamma_n$ is the reduced representative of $\alpha^-\cdot (\alpha\cdot \gamma_n \cdot \alpha^-)\cdot \alpha$. Therefore, if $\alpha$ is a finite concatenation of loops $\ell_m^\pm$, then so is  $\gamma_n$. Hence, in this case, $f_\#(c_\infty)=[\alpha][\prod_{n=1}^\infty \gamma_n][\alpha]^{-1}$ is of the required form. Otherwise, by an application of Lemma~\ref{cancellinglemma}, $\gamma_n\equiv (\cdots \;\cdot a_{i_n+1}^-\cdot a_{i_n}^-)\cdot b_n \cdot (a_{j_n}\cdot a_{j_n+1}\cdot\;\cdots )$ with $[b_n]\in F$, so that $f_\#(c_\infty)=[a_1\cdot a_2\cdot\; \cdots\;\cdot a_{i_1+1}] [\prod_{n=1}^\infty d_n][\cdots\;\cdot a_{3}^{-}\cdot a_2^-\cdot a_1^-]$ with $[d_n]\in F$; which is also of the required form.

Finally, since each $[\ell_{\geq n}]$ lies in $cl_{C,c_{\infty}}(F)$, which is $(C,c_\infty)$-closed, so does the element $\left[\prod_{n=1}^\infty \ell_{\geq n}\right]$. However, $\prod_{n=1}^\infty \ell_{\geq n}$ is a reduced product of order type $\omega\cdot \omega$ and hence does not have the correct order type to represent a finite product of elements in $\pi_1(\bbh,b_0)$ each represented by some finite order type, order type $\omega$, or the reverse order type of $\omega$. That is, $\left[\prod_{n=1}^\infty \ell_{\geq n}\right]\not\in F'$. 
\end{remark}
\subsection{The closure pair $(P,p_{\tau})$ and the transfinite products property}
\begin{definition}\label{tproddef}
We say a space $X$ has {\it transfinite products relative to a subgroup} $H\leq \pionex$ provided that for every pair of maps $a,b:(\bbhp,\bpp)\to (X,x_0)$ such that $a\circ\iota=b\circ\iota$ and $Ha_{\#}(c_n)=Hb_{\#}(c_n)$ for all $n\in\bbn$, we have $Ha_{\#}(g)=Hb_{\#}(g)$ for all $g\in\pionehp$. We say $X$ has {\it transfinite products} if $X$ has transfinite products relative to $H=1$.
\end{definition}

\begin{remark} We briefly justify the terminology. A \textit{transfinite word} over an alphabet $A$ is a function $w:I\to A\cup A^{-1}$ defined on a linearly ordered domain $I$ such that $w^{-1}(s)$ is finite for every $a\in A\cup A^{-1}$ (note that if $A$ is countable, then so is $I$). Given a transfinite word $w:I\to \{\ell_{1}^{\pm 1},\ell_{2}^{\pm 1},\ell_{3}^{\pm 1},\dots\}$, choose components $((a_i,b_i))_{i\in I}$ of the complement $\ui\backslash \mcc$ of the standard middle third cantor set $\mcc$ such that $b_i<a_j$ for all $i<j$ and define a loop $\alpha_{w}:\ui\to \bbh$ by $\alpha_{w}|_{[a_i,b_i]}\equiv w(i)$ and constant at $b_0$ elsewhere. Since $\pioneh\to \check{\pi}_{1}(\bbh,b_0)$ is injective, $[\alpha_w]\in\pioneh$ does not depend on the choice of the components.

Given a map $f:(\bbh,b_0)\to (X,x)$ and a transfinite word $w:I\to \{\ell_{1}^{\pm 1},\ell_{2}^{\pm 1},\ell_{3}^{\pm 1},\dots\}$, we define $w_f=f_{\#}([\alpha_w])\in \pi_1(X,x)$.

Let $(s_n)_{n\in \bbn}$ be a sequence in $\pi_1(X,x)$ such that $s_n=f_{\#}([\ell_n])$ for some map $f:(\bbh,b_0)\to (X,x)$ and let $w:I\to \{s_{1}^{\pm 1},s_{2}^{\pm 1},s_{3}^{\pm 1},\dots\}$ be a transfinite word where $w(i)=s_{\nu(i)}^{\epsilon(i)}$. Put $\overline{w}(i)=\ell_{\nu(i)}^{\epsilon(i)}$. The space $X$ has transfinite products if and only if for any $x\in X$ the definition
\[\prod_{i\in I}s_{\nu(i)}^{\epsilon(i)}:= \overline{w}_{f}\]
does not depend on $f$. This property allows for a well-defined notion of a subgroup transfinitely generated by a null-sequence of elements.
\end{remark}

%
\begin{definition}\label{defptau}
Let $\pinfty\leq \pi_1(\bbhp,\bpp)$ be the free subgroup generated by elements $p_n=[\iota\cdot\ell_{2n-1}\cdot\ell_{2n}^{-}\cdot\iota^{-}]$ for $n\in\bbn$. Consider the maps $f_{odd},f_{even}:\bbh\to \bbh$ satisfying $f_{odd}\circ \ell_n=\ell_{2n-1}$ and $f_{even}\circ \ell_{n}=\ell_{2n}$. Let $p_{\tau}=[\iota\cdot (f_{odd}\circ\ell_{\tau})\cdot(f_{even}\circ\ell_{\tau})^{-}\cdot\iota^{-}]$ (see Figure \ref{ptaufigure}).
\end{definition}

\begin{figure}[H]
\centering \includegraphics[height=0.55in]{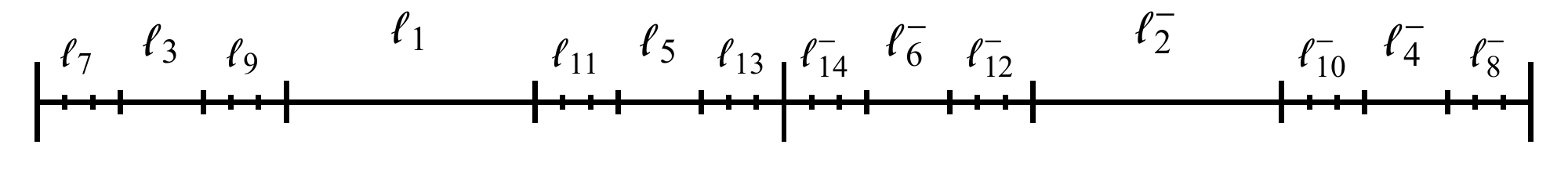}
\caption{\label{ptaufigure}The loop $(f_{odd}\circ\ell_{\tau})\cdot(f_{even}\circ\ell_{\tau})^{-}$}
\end{figure}

\begin{proposition}\label{ptautoctau}
If $H\leq \pionex$ is $(P,p_{\tau})$-closed, then $H$ is $(\cinfty,c_{\tau})$-closed.
\end{proposition}
\begin{proof}
Consider the map $f:\bbhp\to\bbhp$ defined so that $f\circ\iota=\iota$, $f\circ\ell_{2n-1}=\ell_n$, and $f\circ \ell_{2n}$ is constant. Since $f_{\#}(\pinfty)\leq \cinfty$ and $f_{\#}(p_{\tau})=c_{\tau}$, we may apply Remark \ref{comparisonpropremark}.
\end{proof}

\begin{proposition}\label{transfiniteproductprop}
$X$ has transfinite products relative to a subgroup $H\leq \pionex$ if and only if $H$ is $(\pinfty,p_{\tau})$-closed.
\end{proposition}
\begin{proof}
Suppose $X$ does not have transfinite products relative to a subgroup $H\leq \pionex$. Then there are maps $a,b:(\bbhp,\bpp)\to (X,x)$ with $a\circ \iota=b\circ \iota$ and $a_{\#}(c_n)b_{\#}(c_n)^{-1}\in H$ for all $n\in\bbn$ but with $a_{\#}(g)b_{\#}(g)^{-1}\notin H$ for some $g\in\pi_1(\bbh,b_0)$. Using the proof of Proposition \ref{ctauisdense}, there is a map $k:\bbhp\to \bbhp$ such that $k\circ \iota=\iota$, $k_{\#}(c_n)$ is either the identity or $c_{m_n}^{\epsilon_n}$ for some $m_n\geq 1$, $\epsilon_n\in\{\pm 1\}$ and $k_{\#}(c_{\tau})=g$. Now $a\circ k$ and $b\circ k$ are maps $(\bbhp,\bpp)\to (X,x)$ such that either $(a\circ k)_{\#}(c_n)(b\circ k)_{\#}(c_n)^{-1}=1$ or $(a\circ k)_{\#}(c_n)(b\circ k)_{\#}(c_n)^{-1}=a_{\#}(c_{m_n}^{\epsilon_n})b_{\#}(c_{m_n}^{\epsilon_n})^{-1}\in  H$. Additionally, $(a\circ k)_{\#}(c_{\tau})(b\circ k)_{\#}(c_{\tau})^{-1}=a_{\#}(g)b_{\#}(g)^{-1}\notin H$. Define a map $f:\bbhp\to X$ so that $f\circ \iota=a\circ k\circ \iota$, $f\circ \ell_{2n-1}=a\circ k\circ \ell_n$, and $f\circ \ell_{2n}=b\circ k\circ \ell_n$. Note $f_{\#}(p_n)=(a\circ k)_{\#}(c_n)(b\circ k)_{\#}(c_n)^{-1}\in H$ for each $n\in\bbn$ and $f_{\#}(p_{\tau})=(a\circ k)_{\#}(c_{\tau})(b\circ k)_{\#}(c_{\tau})^{-1}\notin H$. Thus $H$ is not $(\pinfty,p_{\tau})$-closed.

For the converse, suppose $H$ is not $(\pinfty,p_{\tau})$-closed. Then there is a map $f:\bbhp\to X$ such that $f_{\#}(\pinfty)\leq H$ and $f_{\#}(p_{\tau})\notin H$. Define $a,b:\bbhp\to X$ such that $a\circ \iota=b\circ \iota=f\circ \iota$ and $a\circ \ell_{n}=f\circ \ell_{2n-1}$ and $b\circ \ell_{n}=f\circ \ell_{2n}$.
Then $a_{\#}(c_n)b_{\#}(c_n)^{-1}=f_{\#}(p_n)\in H$ and $a_{\#}(c_{\tau})b_{\#}(c_{\tau})^{-1}=f_{\#}(p_{\tau})\notin H$, which shows that $X$ does not have transfinite products relative to $H$.
\end{proof}
\begin{proposition}\label{abelianfactorprop}
Suppose $N\leq \pionex$ is a subgroup containing the commutator subgroup of $\pionex$. Then the following are equivalent:
\begin{enumerate}
\item $N$ is $(C,c_{\infty})$-closed,
\item $N$ is $(C,c_{\tau})$-closed,
\item $N$ is $(P,p_{\tau})$-closed.
\end{enumerate}
\end{proposition}
\begin{proof}
(3) $\Rightarrow$ (2) and (2) $\Rightarrow$ (1) follow from Propositions \ref{ptautoctau} and \ref{normalsubgroup} respectively. To complete the proof, we show (1) $\Rightarrow$ (3). Suppose $N$ is $(C,c_{\infty})$-closed. Let $f:(\bbhp,\bpp)\to (X,x_0)$ be a map such that $f_{\#}(P)\leq N$. Recall the definition of $p_{\tau}$ and set $g=f_{\#}(p_{\tau})$. To obtain a contradiction, suppose $g\notin N$. Since $N$ contains the commutator subgroup, $N$ is a normal subgroup of $\pionex$ whose factor group $\pionex/N$ is abelian. Thus, for each $n\geq 1$ there is a loop $\delta_n$ in $\bbh_{\geq 2n+1}$ such that the coset $gN$ factors as:
\begin{eqnarray*}
gN &=& \left(\prod_{k=1}^{n}\left(f_{\#}(c_{2k-1})N\right)\left(f_{\#}(c_{2k}^{-1})N\right)\right)f_{\#}([\iota\cdot\delta_n\cdot\iota^{-}])N\\
&=&  \left(\prod_{k=1}^{n}f_{\#}(p_k)\right)Nf_{\#}([\iota\cdot\delta_n\cdot\iota^{-}])N\\
&=& f_{\#}([\iota\cdot\delta_n\cdot\iota^{-}])N
\end{eqnarray*}
Thus $f_{\#}([\iota\cdot\delta_n\cdot\iota^{-}])\notin N$ for any $n\geq 1$. We proceed as in Proposition \ref{normalsubgroup}. Define a map $h:\bbhp\to X$ so that $h\circ \iota=f\circ \iota$ and $h\circ \ell_{n}=f\circ(\delta_{n}\cdot\delta_{n+1}^{-})$. We have $h_{\#}(C)\leq N$ and $h_{\#}(c_{\infty})=f_{\#}([\iota\cdot\delta_1\cdot\iota^{-}])\notin N$; a contradiction of the assumption that $N$ is $(C,c_{\infty})$-closed.
\end{proof}
We complete this section by proving that $p_{\tau}\notin cl_{C,c_{\tau}}(P)$. Combining this fact with Proposition \ref{comparisonprop}, we confirm that the closure operators $cl_{C,c_{\tau}}$ and $cl_{P,p_{\tau}}$ are not equal.

\begin{lemma}\label{difference1}
Consider the map $f:\bbh\to \bbh$ satisfying $f\circ \ell_n=\ell_{2n-1}\cdot\ell_{2n}^{-}$. Then
\begin{enumerate}
\item $f_{\#}:\pioneh\to\pioneh$ is injective,
\item a based loop $\alpha$ in $\bbh$ is reduced if and only if $f\circ \alpha$ is reduced in $\bbh$,
\item a sequence $\alpha_k$ of based reduced loops in $\bbh$ is null if and only if $f\circ \alpha_k$ is null,
\item If $\alpha$ is a based reduced loop, and there is an increasing sequence $0<t_1<t_2<t_3<\cdots$ converging to $1$ such that $\alpha(t_k)=b_0$ and $[\alpha|_{[0,t_k]}]\in f_{\#}(\pioneh)$ for each $k\geq 1$, then $[\alpha]\in f_{\#}(\pioneh)$.
\end{enumerate}
\end{lemma}
\begin{proof}
(1) Let $\alpha$ be a loop in $\bbh$ such that $[\alpha]\neq 1$ in $\pioneh$. Then there is an $n\geq 1$, $g_1,g_2,\dots,g_k\in \pi_1(\bbh_{\geq n+1},b_0)$, and $h_1,h_2,\dots,h_k\in \pi_1(\bbh_{\leq n},b_0)=F_n$ such that $[\alpha]=g_1h_1g_2h_2\cdots g_kh_k$ and $1\neq h_1h_2\cdots h_k\in F_n$. Thus \[f_{\#}([\alpha])=f_{\#}(g_1)f_{\#}(h_1)f_{\#}(g_2)f_{\#}(h_2)\cdots f_{\#}(g_k)f_{\#}(h_k)\] where $f_{\#}(g_i)\in \pi_1(\bbh_{\geq 2n+1},b_0)$ and $f_{\#}(h_i)\in \pi_1(\bbh_{\leq 2n},b_0)=F_{2n}$. But the restriction of $f$ to $\bbh_{\leq n}$ induces an injection $F_n\to F_{2n}$ on $\pi_1$. Thus \[f_{\#}(h_1)f_{\#}(h_2)\cdots f_{\#}(h_k)=f_{\#}(h_1h_2\cdots h_k)\neq 1\] in $F_{2n}$. It follows that $f_{\#}([\alpha])\neq 1$.\\
(2) Clearly, if $\alpha$ is not reduced, then $f\circ \alpha$ is not reduced. For the converse, let $\alpha$ be reduced. Then $A=\alpha^{-1}(b_0)$ is nowhere dense and if $(a,b)$ is a component of $\ui\backslash A$, then $\alpha|_{[a,b]}$ is a loop of the form $\ell_{n}$ or $\ell_{n}^{-}$. Note that $f\circ \alpha|_{[a,b]}$ must be a loop of the form $\ell_{2n-1}\cdot\ell_{2n}^{-}$ or $\ell_{2n}\cdot\ell_{2n-1}^{-}$. To obtain a contradiction, suppose $f\circ \alpha$ is not reduced. Then there are $0\leq s<t\leq 1$ such that $f\circ \alpha|_{[s,t]}$ is a null-homotopic loop in $\bbh$. Given the definition of $f$ and the fact that $\alpha$ is reduced, we may assume $f\circ \alpha|_{[s,t]}$ is based at $b_0$.

If $\alpha|_{[s,t]}$ is a loop, then $[\alpha|_{[s,t]}]\neq 1$ since $\alpha$ is a reduced. However, $[f\circ\alpha|_{[s,t]}]=1$, which contradicts (1).

If $\alpha(s)\neq b_0$, then by definition of $f$, we have $s=\frac{a+b}{2}$ for a component $(a,b)$ of $\ui\backslash A$. The path $f\circ \alpha|_{[a,b]}$ is either of the form $\ell_{2n-1}\cdot\ell_{2n}^{-}$ or $\ell_{2n}\cdot\ell_{2n-1}^{-}$. First, suppose $f\circ \alpha|_{[a,b]}\equiv\ell_{2n-1}\cdot\ell_{2n}^{-}$.  Since $[f\circ\alpha|_{[s,t]}]=[\ell_{2n}^{-}][f\circ\alpha|_{[b,t]}]=1$, we must have a positive, equal number of appearances of $\ell_{2n}$ and $\ell_{2n}^{-}$ as subloops of $f\circ\alpha|_{[s,t]}$. Note that $\ell_{2n}$ and $\ell_{2n}^{-}$ cannot occur as consecutive subloops of $f\circ\alpha|_{[s,t]}$ since $\alpha$ is reduced. Since $\pioneh$ may be written as the free product $\pi_1(C_{2n},b_0)\ast\pi_1(\bigcup_{m\neq 2n}C_{m},b_0)$ and $f\circ\alpha|_{[s,t]}$ reduces completely in $\bbh$, it must have a subloop of the form $\ell_{2n}\cdot \beta\cdot \ell_{2n}^{-}$ or $\ell_{2n}^{-}\cdot \beta\cdot \ell_{2n}$ where $\beta$ is a nonconstant, null-homotopic loop in $\bigcup_{m\neq 2n}C_{m}$. By definition of $f$, there are $s< s'<t'< t$ such that $\alpha(s')=\alpha(t')=b_0$ and $f\circ \alpha|_{[s',t']}\equiv \beta$. But now $\alpha|_{[s',t']}$ is a reduced loop such that $f_{\#}([\alpha|_{[s',t']}])=1$; a contradiction of (1). On the other hand, if $f\circ \alpha|_{[a,b]}\equiv\ell_{2n}\cdot\ell_{2n-1}^{-}$, we may apply a similar argument using the identification $\pioneh=\pi_1(C_{2n-1},b_0)\ast\pi_1(\bigcup_{m\neq 2n-1}C_{m},b_0)$.

If $\alpha(t)\neq b_0$, we may apply the argument from the previous paragraph.\\
(3) By continuity, $f\circ \alpha_k$ is null whenever $\alpha_k$ is null. If $\alpha_k$ is a sequence of reduced loops which is not null, then there is an $n$ such that infinitely many of the loops $\alpha_k$ traverse at least one of the circles $C_1,C_2,\dots,C_n$. Consequently, infinitely many of the loops $f\circ \alpha_k$ traverse at least one of the circles $C_1,C_2,\dots,C_{2n}$. Thus $f\circ \alpha_k$ is not null.\\
(4) Since $\alpha$ is reduced, $\alpha|_{[0,t_k]}$ is reduced for each $k$. Moreover, since $[\alpha|_{[0,t_k]}]\in f_{\#}(\pioneh)$, by (2) there is a unique, reduced loop $\eta_k:[0,t_k]\to \bbh$ such that $f\circ \eta_k=\alpha|_{[0,t_k]}$. By uniqueness, $\eta_{k+1}|_{[0,t_k]}=\eta_k$. Define $\eta:\ui\to X$ such that $\eta(s)=\eta_{k}(s)$ for $0\leq s\leq t_k$ and $\eta(1)=b_0$. By (3), $\eta$ is a loop with $f\circ \eta\equiv\alpha$, so that $[\alpha]\in f_{\#}(\pioneh)$.
\end{proof}
%
%
\begin{lemma}\label{closedimagelemmma}
If $f:\bbh\to\bbh$ is the map defined in Lemma \ref{difference1}, then $H=f_{\#}(\pioneh)$ is $(C,c_{\tau})$-closed.
\end{lemma}
\begin{proof}
Let $g:(\bbhp,b_{0}^{+})\to (\bbh,b_0)$ be a map such that $g_{\#}(C)\leq H$. Since $\bbh$ is locally contractible at all points of $\bbh\backslash \{b_0\}$, we may focus on the case when $g(b_0)=b_0$. We may also assume that $\alpha=g\circ \iota$ and $\gamma_n=g\circ \ell_n$, $n\in\bbn$ are reduced (or constant) loops satisfying $[\alpha\cdot\gamma_n\cdot\alpha^{-}]\in H$ and that each $\gamma_n$ is nonconstant. We first prove that $[\alpha]\in H$. This is clear if $\alpha$ is constant. Suppose $\alpha$ is not constant.

Case I: Suppose there is a $0<t<1$ such that $\alpha|_{[t,1]}\equiv \ell_{m}^{\pm}$ for some $m\in\bbn$. Find $N$ such that $\gamma_N$ has image in $\bbh_{\geq m+2}$. Then $\alpha\cdot\gamma_N\cdot\alpha^{-}$ is already reduced. Thus, by (2) of Lemma \ref{difference1}, there is a (unique) reduced loop $\beta$ such that $f\circ \beta = \alpha \cdot \gamma_N \cdot \alpha^-$. It suffices to show $1/3\in \beta^{-1}(b_0)$ since this would imply $[\alpha]= f_{\#}([\beta|_{[0,1/3]}])\in H$. Suppose $1/3\notin \beta^{-1}(b_0)$. Since $f\circ\beta(1/3)=b_0$, it follows from the definition of $f$ that there is a component $(r,s)$ of $\ui\backslash \beta^{-1}(b_0)$ such that $1/3=\frac{r+s}{2}$ and $f\circ\beta|_{[r,s]}\equiv \ell_{2k}\cdot\ell_{2k-1}^{-}$ or $f\circ\beta|_{[r,s]}\equiv \ell_{2k-1}\cdot\ell_{2k}^{-}$. But this is impossible since $f\circ\beta|_{[r,1/3]}$ has image in $C_m$ and $f\circ\beta|_{[1/3,s]}$ has image in $\bbh_{\geq m+2}$.

Case II: Suppose there is an increasing sequence $0<s_1<s_2<s_3<\cdots$ converging to $1$ such that $\alpha|_{[s_k,1]}$ is a non-trivial, reduced loop. Let $A_k=\{t\in (s_k,s_{k+1})\mid\alpha(t)=b_0\}$; we may assume that $|A_k|\geq 1$ for each $k$. We show that there exists $t_k\in [s_k,s_{k+1}]$ such that $[\alpha|_{[0,t_k]}]\in H$. By (4) of Lemma \ref{difference1}, this is enough to show that $[\alpha]\in H$. Fix $k$ and find an $m$ such that $\ell_{m}^{\pm}$ appears in $\alpha|_{[s_{k+1},s_{k+2}]}$. Find an $N$ such that $\gamma_{N}$ has image in $\bbh_{\geq m+1}$. Let $\delta$ be the reduced representative of $\alpha\cdot\gamma_{N}\cdot\alpha^{-}$. By Lemma \ref{cancellinglemma}, there is a $0<q<1$ such that $\delta|_{[0,q]}\equiv \alpha|_{[0,s_{k+1}]}$. Additionally, there is a unique $0<p<q$ such that $\delta|_{[p,q]}\equiv \alpha|_{[s_{k},s_{k+1}]}$. Since $[\delta]\in H$, by (2) of Lemma \ref{difference1}, there is a reduced loop $\beta$ such that $f\circ\beta=\delta$. If $\beta(q)=b_0$, set $t_k=s_{k+1}$; it is clear that $[\alpha|_{[0,t_k]}]=[f\circ\beta|_{[0,q]}]\in H$. If $\beta(q)\neq b_0$, then $q=\frac{r+s}{2}$ for some component $(r,s)$ of $\ui\backslash \beta^{-1}(b_0)$. Since $|A_k|\geq 1$, we have $p<r<q$. Take $t_{k}$ to be the unique value such that $\delta|_{[0,r]}\equiv\alpha|_{[0,t_k]}$. Since $\beta(r)=b_0$, we have $[\alpha|_{[0,t_k]}]=[f\circ\beta|_{[0,r]}]\in H$.

Cases I and II together prove that $[\alpha]\in H$. Since we also have $[\alpha\cdot\gamma_n\cdot\alpha^{-}]\in H$ for each $n\in \bbn$, we have $[\gamma_n]\in H$ for each $n\in\bbn$. By (2) of Lemma \ref{difference1}, $\alpha=f\circ \beta$ and $\gamma_n=f\circ \zeta_n$ for reduced loops $\beta$ and $\zeta_n$. Since $\gamma_n$ is a null-sequence, $\zeta_n$ is a null-sequence by (3) of Lemma \ref{difference1}. Define $h:\bbhp\to\bbh$ by $h\circ \iota=\beta$ and $h\circ \ell_n=\zeta_n$. Since $f\circ h=g$, it follows that $g_{\#}(c_{\tau})=f_{\#}(h_{\#}(c_{\tau}))\in H$, completing the proof.
\end{proof}
\begin{theorem}\label{differencetheorem}
$p_{\tau}\notin cl_{C,c_{\tau}}(P)$.
\end{theorem}
\begin{proof}
Let $f:\bbh\to\bbh$ be the map defined in Lemma \ref{difference1} and $f^+:(\bbhp,b_{0}^{+})\to(\bbhp,b_{0}^{+})$ be the map where $f\circ\iota=\iota$ and $f^{+}|_{\bbh}=f$. Let $K=f^{+}_{\#}(\pi_1(\bbhp,b_{0}^{+}))$. We show that $cl_{C,c_{\tau}}(P)=K$ and $p_{\tau}\notin K$.

Note that $P=f^{+}_{\#}(C)$. Since $C$ is $(C,c_{\tau})$-dense, we have $K=f^{+}_{\#}(cl_{C,c_{\tau}}(C))\leq cl_{C,c_{\tau}}(f^{+}_{\#}(C))=cl_{C,c_{\tau}}(P)$. If $h:\bbhp\to \bbh$ is the homotopy equivalence collapsing the attached whisker to $b_0$, then $h_{\#}(K)=f_{\#}(\pioneh)$. By Lemma \ref{closedimagelemmma}, $f_{\#}(\pioneh)$ is $(C,c_{\tau})$-closed. Applying Remark \ref{homotopyinvarianceremark}, we see that $K$ is $(C,c_{\tau})$-closed. Finally, since $P=[\iota]f_{\#}(C)[\iota^{-}]\leq K$ and $K$ is $(C,c_{\tau})$-closed, we have $cl_{C,c_{\tau}}(P)\leq K$. This proves $cl_{C,c_{\tau}}(P)= K$.

The only non-trivial elements of $K$ that may be factored as a product $[\iota][\alpha][\beta][\iota^{-}]$ with $\alpha$ having image in $\bigcup_{n\text{ odd}}C_n$ and $\beta$ having image in $\bigcup_{n\text{ even}}C_n$ are the elements $p_{n}\in P$. Since $p_{\tau}$ has such a factorization yet $p_{\tau}\notin P$, we conclude that $p_{\tau}\notin K$.
\end{proof}
\section{A dyadic arc space as a test space}\label{section4}

A pair $(n,j)$ of integers is \textit{dyadic unital} if $n\geq 1$ and $1\leq j\leq 2^{n-1}$ or equivalently if $\frac{2j-1}{2^n}\in (0,1)$. Let $\scrd$ denote the set of dyadic unital pairs. For each dyadic unital pair $(n,j)$, let \[\bbd(n,j)=\left\{(x,y)\in \bbr^2\Big|\left(x-\frac{2j-1}{2^n}\right)^2+y^2=\left(\frac{1}{2^n}\right)^2\text{, }x\geq 0\right\}\]denote the upper semicircle of radius $\frac{1}{2^n}$ centered at $\left(\frac{2j-1}{2^n},0\right)$. The \textit{base arc} is the interval $B=[0,1]\times \{0\}$. We consider the union $\bbd=B\cup \bigcup_{(n,j)\in \scrd} \bbd(n,j)$ topologized as a subspace of $\bbr^2$ and with basepoint $d_0=(0,0)$ (see Figure \ref{dyadicspacefig}). We call $\bbd(n)=\bigcup_{j=1}^{2^{n-1}} \bbd(n,j)$ the $n$\textit{-th level} of $\bbd$.
\begin{figure}[H]
\centering \includegraphics[height=1.7in]{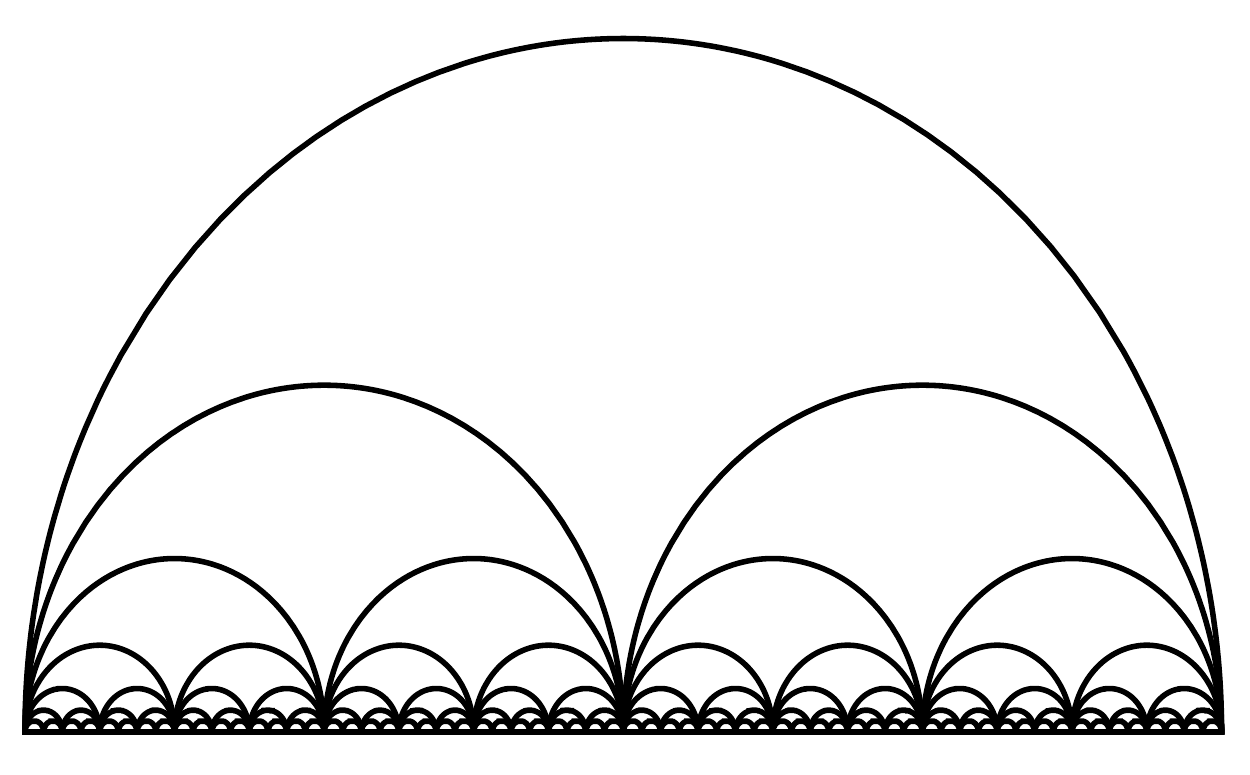}
\caption{\label{dyadicspacefig}The space $\bbd$}
\end{figure}
For each dyadic unital pair $(n,j)$, let $\ell_{n,j}:[0,1]\to \bbd(n,j)$ be the arc $\ds\ell_{n,j}(t)=\left(\frac{t+j-1}{2^{n-1}}, \frac{1}{2^{n-1}}\sqrt{t-t^2}\right)$ from $\left(\frac{j-1}{2^{n-1}},0\right)$ to $\left(\frac{j}{2^{n-1}},0\right)$ (see Figures \ref{arcfig1} and \ref{arcfig2}). We say a path in $\bbd$ is \textit{standard} if it is of the form $\ell_{n,j}$ or $(\ell_{n,j})^{-}$. To simplify notation, we define:
\begin{enumerate}
\item $\lambda_{n,m}=\prod_{j=1}^{m}\ell_{n,j}$ to be the concatenation of standard paths on the $n$-th level from $d_0$ to $\left(\frac{m}{2^{n-1}},0\right)$. We allow $\lambda_{n,0}$ to denote the constant path at $d_0$.
\item $\lambda_n=\lambda_{n,2^{n-1}}=\prod_{j=1}^{2^{n-1}}\ell_{n,j}$ to be the path from $(0,0)$ to $(1,0)$ on the $n$-th level.
\item $\lambda_{\infty}(t)=(t,0)$ to be the unit speed path on the base arc.
\end{enumerate}
\begin{figure}[H]
\centering \includegraphics[height=1.7in]{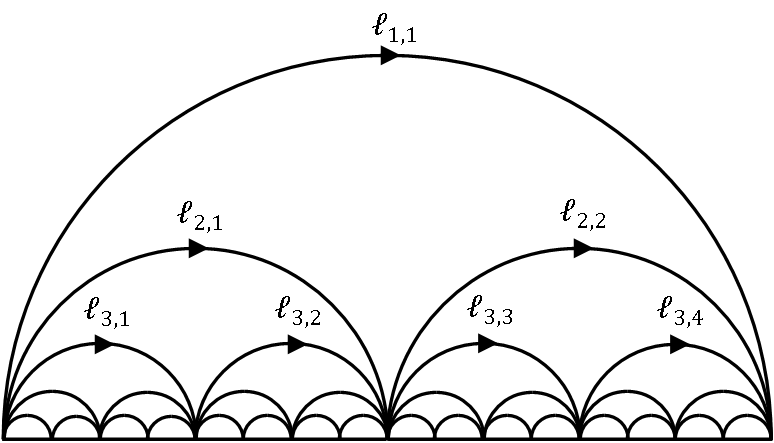}
\caption{\label{arcfig1}The canonical arcs $\ell_{n,j}:\ui\to \bbd$, $1\leq n\leq 3$}
\end{figure}
\begin{figure}[H]
\centering \includegraphics[height=1.5in]{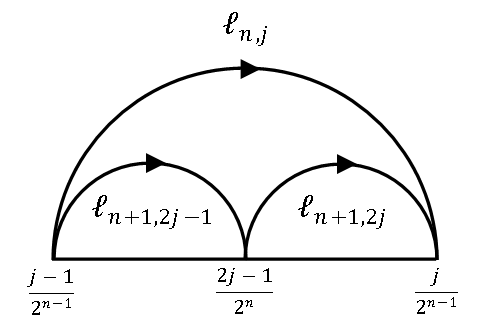}
\caption{\label{arcfig2}The canonical arc $\ell_{n,j}:\ui\to \bbd$}
\end{figure}
As with the Hawaiian earring, the fundamental group of $\bbd$ can be understood as a subgroup of an inverse limit of free groups. Consider the finite graph $E_n=B\cup \bigcup_{m=1}^{n}\bbd(m)$ whose fundamental group $\pi_1(E_n,d_0)=F_{2^{n}-1}$ is free on $2^{n}-1$ generators. For $n'>n$ the retractions $E_{n'}\to E_n$, which collapse a point $(s,t)\in\bigcup_{n<m\leq n'}\bbd(m)$ to the corresponding point $(s,0)$ on the base arc, induce an inverse sequence on $\pi_1$ whose limit $\check{\pi}_1(\bbd,d_0)=\varprojlim_{n} F_{2^{n}-1}$ is the first shape homotopy group. The retractions $r_n:\bbd\to E_n$, which collapse $\bigcup_{m>n}\bbd(m)$ onto the base arc by vertical projection, induce a canonical homomorphism $\psi:\pioned\to \check{\pi}_1(\bbd,d_0)$. Since $\bbd$ is a one-dimensional planar Peano continuum, $\psi$ is injective. Thus two loops $\alpha,\beta\in\Omega(\bbd,d_0)$ are homotopic if and only if for every $n\in\bbn$ the projections of $\alpha$ and $\beta$ are homotopic in $E_n$.
\subsection{The subgroup $\bbs\leq \pioned$ and the homotopically path Hausdorff property}
\begin{definition}\label{hompathhausdef} (Homotopically path Hausdorff relative to a subgroup) We call $X$ \textit{homotopically path Hausdorff relative to} $H$ if for every pair of paths $\alpha,\beta\in P(X,x_0)$ such that $\alpha(1)=\beta(1)$ and $[\alpha\cdot\beta^{-}]\notin H$, there is an integer $n\geq 1$ and a sequence of open sets $U_1,U_2,\dots,U_{2^{n-1}}$ with $\alpha\left(\left[\frac{j-1}{2^{n-1}},\frac{j}{2^{n-1}}\right]\right)\subseteq U_j$, such that if $\gamma:[0,1]\to X$ is another path satisfying $\gamma\left(\left[\frac{j-1}{2^{n-1}},\frac{j}{2^{n-1}}\right]\right)\subseteq U_j$ for $1\leq j\leq 2^{n-1}$ and $\gamma\left(\frac{j}{2^{n-1}}\right)=\alpha\left(\frac{j}{2^{n-1}}\right)$ for $0\leq j\leq 2^{n-1}$, then $[\gamma\cdot\beta^{-}]\notin H$.

We call $X$ \textit{homotopically path Hausdorff} if it is homotopically path Hausdorff relative to the trivial subgroup $H=1$.
\end{definition}
\begin{remark}
The original definition of the homotopically path Hausdorff property given in \cite{FRVZ11} does not use dyadic rationals, however, it is equivalent to the definition used here, which is more convenient for our purposes.
\end{remark}
\begin{definition}\label{defsubgroups}
For $n\in\bbn$, let $s_{n}=[\lambda_{n}\cdot (\lambda_{n+1})^{-}]$. Let $\bbs$ be the subgroup of $\pioned$ freely generated by $\{s_n\mid n\in\bbn\}$ and let $\dinf=[\lambda_1\cdot\lambda_{\infty}^{-}]$.
\end{definition}

Although $(\bbd,d_0)$ is not well-pointed, we use the self-similarity of $\bbd$ to show that $(\bbs,\dinf)$ is a normal closure pair.

\begin{proposition}\label{conjugatesforS}
$(\bbs,\dinf)$ is a normal closure pair for $(\bbd,d_0)$.
\end{proposition}

\begin{proof}
Let $\bbd^+$ be the well-pointed space constructed in Remark \ref{addawhisker} with accompanying normal closure pair $(S',\dinf ')=([\iota]i_{\#}(\bbs)[\iota^{-}],[\iota]i_{\#}(\dinf)[\iota^{-}])$ where $i:\bbd\to\bbd^+$ and $\iota:\ui\to \bbd^+$ are the inclusion maps. Identify $\bbd^+=B\cup \bigcup_{n=2}^{\infty}\bigcup_{j=2^{n-2}+1}^{2^{n-1}}\bbd(n,j)$ and define a canonical retraction $r:(\bbd,d_0)\to (\bbd^+,d_0)$ collapsing $\bbd\backslash \bbd^+$ vertically onto the arc $([0,1/2]\times\{0\})\cup \bbd(2,2)$. Since $r_{\#}(\bbs)=S'$ and $r_{\#}(\dinf)=\dinf '$, we have $\dinf '\in cl_{S,\dinf}(S')$ by Remark \ref{comparisonpropremark}. To prove the proposition, suppose $H\leq \pionex$ is $(\bbs,\dinf)$-closed, $\alpha\in P(X,x_0)$ is a path, and $f:(\bbd,d_0)\to (X,\alpha(1))$ is a map such that $f_{\#}(\bbs)\leq H^{\alpha}$. It suffices to show that $f_{\#}(\dinf)\in H^{\alpha}$. The map $f$ and the path $\alpha$ uniquely induce a map $k:(\bbd^+,d_0)\to (X,x_0)$ satisfying $k\circ \iota=\alpha$ and $k|_{\bbd}=f$. Since $k_{\#}(\bbs ')\leq H$, we have
\[[\alpha]f_{\#}(\dinf)[\alpha^{-}]=k_{\#}(\dinf ')\in k_{\#}(cl_{\bbs,\dinf}(\bbs '))\leq cl_{\bbs,\dinf}(k_{\#}(\bbs ')))\leq cl_{\bbs,\dinf}(H)= H.\]Thus $f_{\#}(\dinf)\in H^{\alpha}$.
\end{proof}

\begin{theorem}\label{hompathhausdchar}
If $X$ is homotopically path Hausdorff relative to $H$, then $H$ is $(\bbs,\dinf)$-closed. The converse holds if the path space $P(X)$ is first countable; for instance, if $X$ is metrizable.
\end{theorem}
\begin{proof}
If $H$ is not $(\bbs,\dinf)$-closed, then there is a map $f:(\bbd,d_0)\to (X,x_0)$ such that $f_{\#}(\bbs)\leq H$ and $f_{\#}(\dinf)\notin H$. Set $\alpha=f\circ \lambda_{\infty}$, $\beta=f\circ \lambda_1$, and $\alpha_n=f\circ \lambda_n$ for $n\geq 2$. Note that $\alpha_n\to \alpha$ in $P(X)$ and that  $[\alpha\cdot \beta^{-}]\notin H$. Pick any $n\in\bbn$ and sequence of neighborhoods $U_1,U_2,\dots,U_{2^{n-1}}$ such that $\alpha\left(\left[\frac{j-1}{2^{n-1}},\frac{j}{2^{n-1}}\right]\right)\subseteq U_{j}$ for each $j$. Choose $N\in\bbn$ so that $\alpha_N\in \bigcap_{j=1}^{2^{n-1}}\langle \left[\frac{j-1}{2^{n-1}},\frac{j}{2^{n-1}}\right],U_j\rangle $. In particular,
\begin{itemize}
\item[]$\alpha_N\left(\left[\frac{j-1}{2^{n-1}},\frac{j}{2^{n-1}}\right]\right)\subseteq U_{j}$ for $1\leq j\leq 2^{n-1}$,
\item[]$\alpha_N\left(\frac{j}{2^{n-1}}\right)=\alpha\left(\frac{j-1}{2^{n-1}}\right)$ for $0\leq j\leq 2^{n-1}$.
\end{itemize}
Put $\gamma=\alpha_N$. Since $[\alpha_n\cdot\alpha_{n+1}^{-}]=f_{\#}(s_n)\in H$ for each $n\geq 1$, we have $[\gamma\cdot \beta^{-}]=\left(\prod_{n=1}^{N-1}[\alpha_n\cdot \alpha_{n+1}^{-}]\right)^{-1}\in H$. Thus $X$ cannot be homotopically path Hausdorff relative to $H$.

For the converse, suppose $P(X)$ is first countable. If $X$ is not homotopically path Hausdorff relative to $H$, then there are paths $\alpha,\beta:\ui \to X$, $\alpha(0)=\beta(0)=x_0$ and $\alpha(1)=\beta(1)$ with the property that for any integer $n\geq 1$ and sequence of open sets $U_1,U_2,\dots,U_{2^{n-1}}$ with $\alpha\left(\left[\frac{j-1}{2^{n-1}},\frac{j}{2^{n-1}}\right]\right)\subseteq U_j$, there is a path $\gamma:\ui\to X$ satisfying
\begin{enumerate}
\item $\gamma\left(\left[\frac{j-1}{2^{n-1}},\frac{j}{2^{n-1}}\right]\right)\subseteq U_j$ for $1\leq j\leq 2^{n-1}$,
\item $\gamma\left(\frac{j}{2^{n-1}}\right)=\alpha\left(\frac{j}{2^{n-1}}\right)$ for $0\leq j\leq 2^{n-1}$,
\item and $[\gamma\cdot\beta^{-}]\in H$.
\end{enumerate}
Take a countable, nested neighborhood base $\mathcal{U}_1\supset \mathcal{U}_2\supset\mathcal{U}_3\supset\cdots$ at $\alpha$. We may assume each neighborhood is of the form \[\mathcal{U}_{p}=\bigcap_{j=1}^{2^{n(p)-1}}\left\langle \left[\frac{j-1}{2^{n(p)-1}},\frac{j}{2^{n(p)-1}}\right],U_{n(p),j}\right\rangle\]for some increasing sequence $1= n(1)<n(2)<n(3)<\cdots$ of natural numbers. By assumption, there is a path $\alpha_{n(p)}:\ui\to X$ satisfying
\begin{enumerate}
\item $\alpha_{n(p)}\left(\left[\frac{j-1}{2^{n(p)-1}},\frac{j}{2^{n(p)-1}}\right]\right)\subseteq U_{n(p),j}$ for $1\leq j\leq 2^{n(p)-1}$,
\item $\alpha_{n(p)}\left(\frac{j}{2^{n(p)-1}}\right)=\alpha\left(\frac{j}{2^{n(p)-1}}\right)$ for $0\leq j\leq 2^{n(p)-1}$,
\item $[\alpha_{n(p)}\cdot\beta^{-}]\in H$.
\end{enumerate}
If $n(p)<n<n(p+1)$, set $\alpha_n=\alpha_{n(p+1)}$. We have $\alpha_n\to\alpha$ in $P(X)$ and for every $n\in\bbn$, $\alpha_n\left(\frac{j}{2^{n-1}}\right)=\alpha\left(\frac{j}{2^{n-1}}\right)$ for $0\leq j\leq 2^{n-1}$. Thus we obtain a unique map $f:(\bbd,d_0)\to (X,x_0)$ such that $f\circ \lambda_{n}=\alpha_{n}$ and $f\circ \lambda_{\infty}=\alpha$. For any given $n\in\bbn$, we have $\alpha_n=\alpha_{n(p)}$ for some $p\in\bbn$ and thus $[\alpha_{n}\cdot\beta^{-}]=[\alpha_{n(p)}\cdot\beta^{-}]\in H$. It follows that $f_{\#}(s_n) = [\alpha_n\cdot \alpha_{n+1}^{-}] = [\alpha_{n}\cdot\beta^{-}][\alpha_{n+1}\cdot\beta^{-}]^{-1}\in H$ for each $n\in\bbn$ and therefore $f_{\#}(\bbs)\leq  H$. Moreover, $f_{\#}(\dinf)=[\alpha_{1}\cdot \alpha^{-}]=[\alpha_{1}\cdot\beta^{-}][\alpha\cdot\beta^{-}]^{-1}\notin H$ since $[\alpha\cdot\beta^{-}]\notin H$.
\end{proof}
\begin{remark}\label{quotienttopremark}
The fundamental group $\pionex$ inherits a natural topology when it is viewed as the quotient space of $\Omega(X,x_0)$. Equipped with this topology, the fundamental group may fail to be a topological group, however, it is a quasitopological group in the sense that inversion is continuous and multiplication is continuous in each variable; for more on this topology see \cite{BFqtop}. In a quasitopological group $G$, the topological closure $\overline{H}$ of a subgroup $H\leq G$ is still a subgroup of $G$ \cite{AT}. Additionally, for a locally path-connected space $X$, the property of being homotopically path Hausdorff relative to $H$ is equivalent to $H$ being closed in $\pionex$ \cite[Lemma 9]{BFqtop}. Combining these facts with Theorem \ref{hompathhausdchar}, it follows that if $X$ is a locally path-connected metric space, then the closure operator $cl_{S,d_{\infty}}$ agrees with the topological closure in $\pionex$.
\end{remark}
\subsection{The subgroup $\bbf\leq \pioned$ and the unique path lifting property}

\begin{definition}\label{defsubgroupD}
Let $\bbf\leq \pioned$ be the subgroup consisting of homotopy classes of finite concatenations $\prod_{k=1}^{m}\ell_{n_k,j_k}^{\epsilon_k}$, $\epsilon_k\in \{\pm  \}$ of standard paths. Recall that $d_{\infty}=[\lambda_{1}\cdot \lambda_{\infty}^{-}]$.
\end{definition}

\begin{lemma}\label{fchar}
$\bbf$ is generated by the homotopy classes of all well-defined loops of the form $\lambda_{n,m}\cdot \lambda_{n',m'}^{-}$.
\end{lemma}
\begin{proof}
Let $a=\prod_{i=1}^{K}[\ell_{n_i,j_i}]^{\epsilon_i}$, $\epsilon_i\in \{\pm 1\}$ be a non-trivial element of $\bbf$. Note that $K\geq 3$, $\epsilon_1=1=-\epsilon_K$, and $j_1=1=j_K$. For any dyadic unital pair $(n,j)$, we have $\ell_{n,j}\simeq \lambda_{n,j-1}^{-}\cdot\lambda_{n,j}$, $\lambda_{n,1}=\ell_{n,1}$, and $\lambda_{n,0}$ is constant. Thus $a$ may be written as the product \[[\ell_{n_1,1}]\left(\prod_{i=2}^{K-1}[\ell_{n_i,j_i}]^{\epsilon_i}\right)[\ell_{n_K,1}^{-}]=
[\lambda_{n_1,1}]\left(\prod_{i=2}^{K-1}[\lambda_{n_i,j_i-1}^{-}\cdot\lambda_{n_i,j_i}]^{\epsilon_i}\right)[\lambda_{n_K,1}^{-}]\]
the latter of which is a product of elements of the form $[\lambda_{n,m}\cdot\lambda_{n',m'}^{-}]$.
\end{proof}

\begin{remark}\label{sclosedisdclosed}
Observe that $\bbs\leq \bbf$. Therefore a subgroup $H\leq \pionex$ is $(\bbf,\dinf)$-closed whenever $H$ is $(\bbs,\dinf)$-closed.
\end{remark}

The proof of the following proposition is nearly identical to that of Proposition \ref{conjugatesforS}, so we omit it.

\begin{proposition}\label{conjugatesforF}
$(\bbf,\dinf)$ is a normal closure pair for $(\bbd,d_0)$.
\end{proposition}

Consider the subset $G=\bigcup_{n\geq 1}\bbd(n)\subseteq \bbd$ topologized as the direct limit of the inclusions $$\bbd(1)\to \bbd(1)\cup\bbd(2)\to \bbd(1)\cup\bbd(2)\cup \bbd(3)\to \cdots$$ of finite graphs. With this weak CW-topology, which is finer than the subspace topology of $\bbd$, $G$ becomes a graph whose vertex set $V$ is indexed by the set of dyadic rationals in $\ui$. The unique edge between vertices $\left(\frac{j-1}{2^{n-1}},0\right)$ and $\left(\frac{j}{2^{n-1}},0\right)$ is the semicircle $\bbd(n,j)$. Note the continuous inclusion $i:G\to \bbd$ is not an embedding but does induce a monomorphism $i_{\#}:\pi_1(G,d_0)\to \pioned$. By construction, $\bbf$ is precisely the image of $i_{\#}$. A maximal tree $T\subseteq G$ is obtained by removing all edges $\bbd(n,k)$ where $k$ is even (see Figure \ref{treefig}). According to classical graph theory, the edges of $G$ which are not also edges of $T$ are in bijective correspondence with generators of the free group $\pi_1(G,d_0)$. Thus generators of $\pi_1(G,d_0)$ are in bijective correspondence with the set of edges $\{\bbd(n+1,2j)\mid (n,j)\in \scrd\}$. It follows that $\bbf\cong\pi_1(G,d_0)$ is isomorphic to the free group $F(\scrd)$ on the countably infinite set $\scrd$ of dyadic unital pairs.

\begin{figure}[H]
\centering \includegraphics[height=1.5in]{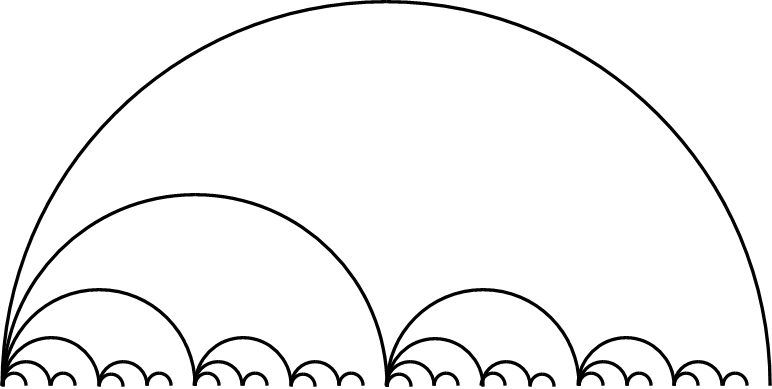}
\caption{\label{treefig}The maximal tree $T\subseteq G$}
\end{figure}

To identify explicit free generators of $\bbf$, we define, for each $t\in\ui$, a path $\delta_t:\ui\to \bbd$ from $(0,0)$ to $(t,0)$. For $t=1$, we set $\delta_1=\ell_{1,1}$. If $t<1$, recall that $t$ has a binary expansion $0.a_1a_2a_3\cdots=\sum_{n=1}^{\infty}\frac{a_n}{2^n}$, $a_n\in \{0,1\}$ such that $(a_n)$ has a cofinal subsequence of $0$s. Note that $t$ is a dyadic rational if and only if the sequence $(a_n)$ is eventually constant at $0$. We define $\delta_t$ to be the infinite concatenation $\prod_{n=1}^{\infty}\eta_n$ of a sequence of paths $\eta_n$ constructed as follows: If $a_1=0$, let $\eta_1$ be the constant path at $(0,0)$ and if $a_1=1$, let $\eta_1=\ell_{2,1}$. Inductively, suppose the paths $\eta_1,\eta_2,\dots,\eta_{n-1}$ have been defined so that the concatenation $\prod_{k=1}^{n-1}\eta_{k}$ is a path in $T\cap\bigcup_{j=1}^{n}\bbd(j)$ from $d_0$ to $\left(\sum_{k=1}^{n-1}\frac{a_k}{2^k},0\right)$. Write $\sum_{k=1}^{n-1}\frac{a_k}{2^k}=\frac{j-1}{2^{n-1}}$ for $j\in \{1,2,\dots,2^{n-1}\}$. If $a_{n}=0$, let $\eta_{n}$ be the constant path at $\left(\frac{j-1}{2^{n-1}},0\right)$ and if $a_{n}=1$, set $\eta_{n}=\ell_{n+1,2j-1}$. By construction, $\eta_n$ has image in $T\cap\bbd(n+1)$ and $\eta_n(1)=\left(\sum_{k=1}^{n}\frac{a_k}{2^k},0\right)$. It follows that the sequence of paths $\eta_n$ is null at $(t,0)$ so that the infinite concatenation $\delta_t=\prod_{n=1}^{\infty}\eta_n$ is well-defined.

Note that $\delta_0$ is the constant path at $d_0$ and if $t\in (0,1)$ is a dyadic rational, then there is an $N$ such that $\eta_n=c_{(t,0)}$ is the constant path at $(t,0)$ for all $n\geq N$. In this case, $\delta_t$ is a reparameterization of the arc in $T$ from $d_0$ to $(t,0)$. We conclude that for each dyadic unital pair $(n,j)$, there is a corresponding free generator $d_{n,j}$ of $\pi_1(G,d_0)$ (see Figure \ref{generators}) defined as the homotopy class of the loop $\left(\delta_{\frac{2j-1}{2^n}}\right)\cdot \left(\ell_{n+1,2j}\right)\cdot \left(\delta_{\frac{j}{2^{n-1}}}\right)^{-}$.
\begin{figure}[H]
\centering \includegraphics[height=1.4in]{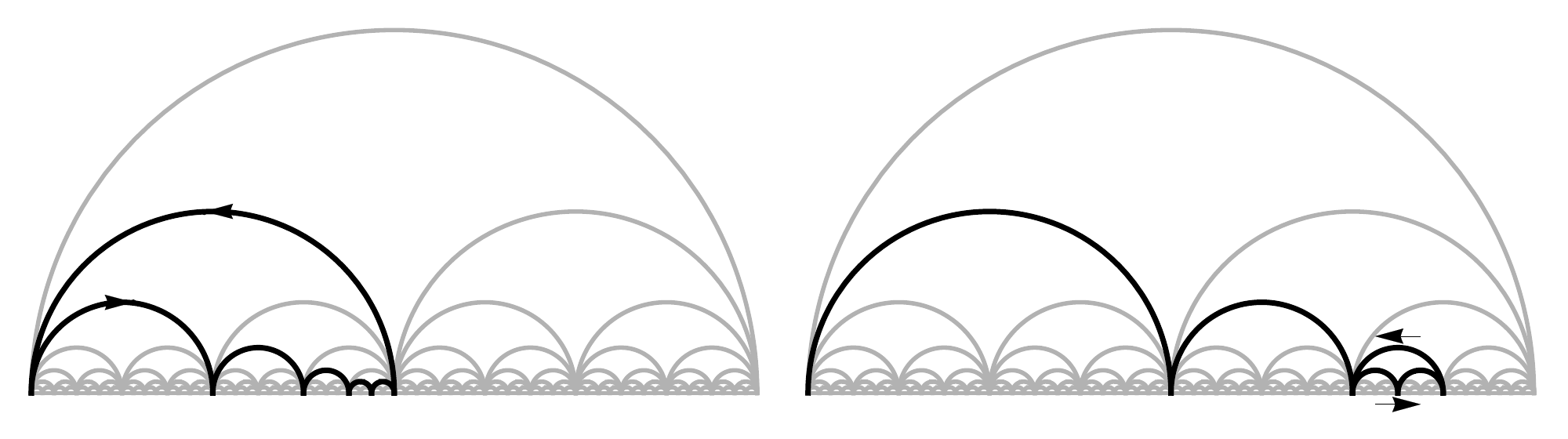}
\caption{\label{generators}Representatives of the free generators $d_{5,8}$ (left) and $d_{4,7}$ (right).}
\end{figure}
\begin{lemma}\label{fistransitive}
$\bbf$ is $(\bbf,\dinf)$-dense in $\pioned$.
\end{lemma}
\begin{proof}
To verify the sufficient condition in Remark \ref{densitysufficientremark}, we show that for every loop $\alpha:([0,1],0)\rightarrow (\bbd,d_0)$, there is a continuous map $f:(\bbd,d_0)\rightarrow (\bbd,d_0)$ such that $f_{\#}(\dinf)=[\alpha]$ and $f_{\#}(\bbf)\leq \bbf$. To begin, we identify the domain of $\alpha$ with $B=[0,1]\times\{0\}\subseteq \bbd$ and define $f(t)=\alpha(t)$ for $t\in [0,1]$ and $f(t)=d_0$ for $t\in \bbd(1,1)$.

For each $a,b\in [0,1]$ with $a\leq b$ and each $n\in \bbn$, we define a path $\beta_n^{a,b}$ in $\bigcup_{k=1}^\infty \bigcup_{j=1}^{2^{k-1}} \bbd(k,j)\subseteq \bbd$ from $a$ to $b$ as follows:
     Let $i,j\in \bbn$ with $\frac{i-1}{2^{n-1}}\leq a< \frac{i}{2^{n-1}}$ and $\frac{j-1}{2^{n-1}}\leq b< \frac{j}{2^{n-1}}$. Using the transformation $T_{n,k}(x,y)=(2^{n-1}x-k+1,2^{n-1}y)$ from the homeomorphic copy of $\bbd$ under $\bbd(n,k)$ to $\bbd$, we define \[\beta_n^{a,b}=(T_{n,i}^{-1}\circ \delta_{T_{n,i}(a,0)})^-\cdot \ell_{n,i}\cdot \ell_{n,i+1}\cdot\;\cdots\;\cdot \ell_{n,j-1}\cdot T_{n,j}^{-1}\circ \delta_{T_{n,j}(b,0)}.\] We also define  $\beta_n^{b,a}=\left(\beta_n^{a,b}\right)^-$.

Let $\mathcal I$ be the set of components of $\alpha^{-1}(\bbd\setminus [0,1])$. We now define $f$ on each $\bbd(n,j)$ with $n\geqslant 2$ and $1\leq j\leq 2^{n-1}$.
If $\frac{j-1}{2^{n-1}}\in I$ for some interval  $I\in \mathcal{I}$ with endpoints $c<d$, then put $u=\min\{d,\frac{j}{2^{n-1}}\}$; otherwise put $u=\frac{j-1}{2^{n-1}}$.
Likewise,
if $\frac{j}{2^{n-1}}\in J$ for some interval $J\in \mathcal{I}$ with endpoints $s<t$, then put $v=\max\{s,\frac{j-1}{2^{n-1}}\}$; otherwise put $v=\frac{j}{2^{n-1}}$. Now, let $(x,y)\in \bbd(n,j)$. We define $f(x,y)=\alpha(x)$ if $x\in [\frac{j-1}{2^{n-1}}, u]\cup [v, \frac{j}{2^{n-1}}]$. We define $f(x,y)=\beta_n^{\alpha(u),\alpha(v)}(x)$ if $x\in[u,v]$.

Clearly, $f:\bbd\rightarrow \bbd$ is well-defined. Continuity of $f$ follows from the uniform continuity of $\alpha$ and the fact that $diam(\beta_n^{a,b})\leq |a-b| + \frac{3}{2^n}$. Also, by construction, we have $f_\#(\dinf)=[\alpha]$.

   It remains to show that $f_\#(\bbf)\leq \bbf$. To this end, let $[p]\in \bbf$ for some finite edge path $p=\ell_{n_1,j_1}^{\epsilon_1}\cdot\ell_{n_2,j_2}^{\epsilon_2}\cdot\;\cdots\;\cdot\ell_{n_K,j_K}^{\epsilon_K}$ in $G$, with $\epsilon_k\in\{+,-\}$. In order to show that $f\circ p$ is homotopic to a finite edge path in $G$, let $u_k$ and $v_k$ be defined as above for $\bbd(n_k,j_k)$. Then each $f\circ \ell_{n_k,j_k}$ is homotopic to $\alpha|_{[\frac{j_k-1}{2^{n_k-1}}, u_k]}\cdot \beta_{n_k}^{\alpha(u_k),\alpha(v_k)}\cdot \alpha|_{[v_k, \frac{j_k}{2^{n_k-1}}]}$. Observe that if $f\circ \ell_{n_k,j_k}^{\epsilon_k}$ terminates in a nondegenerate $\alpha$-segment, then $f\circ \ell_{n_{k+1},j_{k+1}}^{\epsilon_{k+1}}$ either begins with or equals a nondegenerate $\alpha$-segment. Further, every maximal concatenation of contiguous $\alpha$-segments of $f\circ p$ forms a path within one and the same $\bbd(n,j)$, starting and ending in $[0,1]$. Hence, each such subpath of $f\circ p$ can be homotoped to either $\ell_{n,j}^{\pm}$ or a constant. As for the $\beta$-segments, they might be ``wild'' at either end.
  However, if $u_k> \frac{j_k-1}{2^{n_k-1}}$, then the corresponding endpoint of the $\beta$-segment is a dyadic rational in $[0,1]$, making it ``tame'' at that end. The same is true at the other end if $v_k< \frac{j_k}{2^{n_k-1}}$. Therefore, we only need to consider points $b$ that are the endpoints of two consecutive $\beta$-segments such that $b$ is a dyadic irrational and $b=\alpha(t)\in [0,1]$ with $t=\frac{s-1}{2^{n-1}}$ for some $s$. In this case, however, the two $\beta$-segments meeting in $b$ cancel over their ``wild ends'', as can be seen from the following formulas:
  $T^{-1}_{n,j}\circ \delta_{T_{n,j}(b,0)}=T^{-1}_{n+1,2j-1} \circ \delta_{T_{n+1,2j-1}(b,0)}$ if $\frac{j-1}{2^{n-1}}=\frac{2j-2}{2^n}< b< \frac{2j-1}{2^n}$, and $T^{-1}_{n,j}\circ \delta_{T_{n,j}(b,0)}=\ell_{n+1,2j-1}\cdot \left(T^{-1}_{n+1,2j} \circ \delta_{T_{n+1,2j}(b,0)}\right)$ if $\frac{2j-1}{2^n}< b< \frac{2j}{2^n}=\frac{j}{2^{n-1}}$.
\end{proof}
We use the closure pair $(\bbf,\dinf)$ to characterize the subgroups $H\leq \pionex$ for which $p_H:\tXh\to X$ has the unique path lifting property. Recall the construction of $\tXh$ from Section 1.
\begin{lemma}\label{deltafinite}
Let $(t,0)\in B$, $\epsilon>0$, and $V=\bbd\cap E$ where $E\subseteq \bbr^2$ is the open disk of radius $\epsilon$ centered at $(t,0)$. If $|s-t|<\epsilon$, then $\bbf[\delta_s]\in B(\bbf[\delta_t],V)$.
\end{lemma}
\begin{proof}
Set $I=\{s\in\ui \mid |s-t|<\epsilon\}$ so that $I\times \{0\}=V\cap B$. First, suppose $u,v\in I$ are dyadic rationals. In this case, $\delta_u$ and $\delta_v$ are homotopic to a finite concatenation of standard paths. Additionally, there is an arc $\gamma:\ui\to V$ which is a finite concatenation of standard paths (with image in $V$) from $(u,0)$ to $(v,0)$. Since the loop $\delta_u\cdot \gamma\cdot (\delta_{v})^{-}$ is a finite concatenation of standard paths, we have $[\delta_u\cdot \gamma\cdot \delta_{v}^{-}]\in \bbf$. Thus $\bbf[\delta_u]\in B(\bbf[\delta_v],V)$. It follows that $B(\bbf[\delta_u],V)=B(\bbf[\delta_v],V)$ for all dyadic rationals $u,v\in I$. It now suffices to show that for each dyadic irrational $s\in I$, there is a dyadic rational $u\in I$ such that $\bbf[\delta_s]\in B(\bbf[\delta_u],V)$.

If $s\in I$ is not a dyadic rational, then $\delta_s$ is an infinite concatenation $\prod_{n=1}^{\infty}\eta_{n}$ from $d_0$ to $(s,0)\in V$ such that $\eta_n$ is either a standard path or a constant path. Find $N>1$ such that $\eta_n$ has image in $V$ for each $n\geq N$ and let $\gamma$ be the path $\prod_{n=N}^{\infty}\eta_{n}$. Note that $\gamma(0)=(u,0)$ with $u\in I$ a dyadic rational and that $\delta_u$ is a reparameterization of $\prod_{n=1}^{N-1}\eta_{n}$. Since $\left[\delta_s\cdot \gamma^{-}\cdot \delta_{u}^{-}\right]=\left[\left(\prod_{n=1}^{N-1}\eta_{n}\right)\cdot \delta_{u}^{-}\right]=1\in \bbf$ where $\gamma^{-}$ has image in $V$, we have $\bbf[\delta_s]\in B(\bbf[\delta_u],V)$.
\end{proof}
%
%


%
\begin{theorem}\label{uplchar}
If $p_H:\tXh\to X$ has the unique path lifting property, then $H\leq \pionex$ is $(\bbf,\dinf)$-closed. The converse holds if the path space $P(X)$ is first countable; for instance, if $X$ is metrizable.
\end{theorem}
\begin{proof}
Suppose that $H$ is not $(\bbf,\dinf)$-closed. Then there is a map $f:(\bbd,d_0)\to (X,x_0)$ such that $f_{\#}(\bbf)\leq H$ and $f_{\#}(\dinf)\notin H$. We show the path $\gamma(t)=f(t,0)$ does not have a unique lift with respect to $p_H:\tXh\to X$. Take $\widetilde{\gamma}_{\mathscr{S}}:\ui\to \tXh$ to be the standard lift of $\gamma$. We define a second lift $\beta:\ui\to \tXh$ by setting $\beta(t)=H\left[f\circ\delta_t\right]$. Observe that $p_H\circ \beta=\gamma$, $\beta(0)=\txh=\widetilde{\gamma}_{\mathscr{S}}(0)$. Moreover, since $f_{\#}(\dinf)=\left[f\circ \left(\lambda_{1}\cdot \lambda_{\infty}^{-}\right)\right]\notin H$, we have $\beta(1)=H[f\circ \lambda_1]\neq H[f\circ \lambda_{\infty}]=H[\gamma]=\widetilde{\gamma}_{\mathscr{S}}(1)$. Therefore, it suffices to verify the continuity of $\beta$.

Suppose $B(H\left[f\circ\delta_t\right],U)$ is an open neighborhood of $\beta(t)=H\left[f\circ\delta_t\right]$ in $\tXh$ where $U$ is an open neighborhood of $f\circ\delta_t(1)=f(t,0)$ in $X$. Since $f$ is continuous, there is an $\epsilon>0$ such that if $E\subseteq \bbr^2$ is the open disk of radius $\epsilon$ centered at $(t,0)$, then $V=\bbd\cap E\subseteq f^{-1}(U)$. If $|s-t|<\epsilon$, then $\bbf[\delta_s]\in B(\bbf[\delta_t],V)$ by Lemma \ref{deltafinite}. Thus $[\delta_s\cdot \zeta^{-}\cdot \delta_{t}^{-}]\in \bbf$ for some path $\zeta$ in $V$. Applying the homomorphism $f_{\#}$, we see that
$[(f\circ\delta_s)\cdot (f\circ\zeta^{-})\cdot (f\circ\delta_{t})^{-}]=f_{\#}([\delta_s\cdot \zeta^{-}\cdot \delta_{t}^{-}])\in f_{\#}(\bbf)\leq H$ where $f\circ \zeta^{-}$ is a path in $f(V)\subseteq U$. It follows that $\beta(s)=H\left[f\circ\delta_s\right]\in B(H\left[f\circ\delta_t\right],U)$.

For the converse, we assume that $P(X)$ is first countable and that $p_H:\tXh\to X$ does not have the unique path lifting property. Then there are paths $\alpha\in P(X,x_0)$ and $\beta\in P(\tXh,\txh)$ such that $p_H\circ \beta=\alpha$ and $\beta(1)\neq \widetilde{\alpha}_{\mathscr{S}}(1)$. Note that $\beta(t)=H[\beta_t]$ for some path $\beta_t:\ui\to X$ from $x_0$ to $\alpha(t)$. Thus $H[\beta_1]\neq H[\alpha]$. Since $H[\beta_0]=\txh$, we may assume that $\beta_0=c_{x_0}$. We extend $\alpha$ to a map $f:\bbd\to X$ such that $f(u,0)=\alpha(u)$ on the base arc. Consider a countable neighborhood base $\mathcal{U}_1\supset \mathcal{U}_2\supset\mathcal{U}_3\supset\cdots$ at $\alpha$ in $P(X)$. We may assume that $\mathcal{U}_p$ is of the form \[\mathcal{U}_{p}=\bigcap_{j=1}^{2^{n(p)-1}}\left\langle \left[\frac{j-1}{2^{n(p)-1}},\frac{j}{2^{n(p)-1}}\right],U_{n(p),j}\right\rangle\] where $1= n(1)<n(2)<n(3)<\cdots $ is an increasing sequence of natural numbers.

To define $f$ on the $n(p)$-th level, we take the approach of \cite[Theorem 2.9]{FRVZ11}. Suppose $n\in \bbn$ is such that $n=n(p)$ for some $p$. Since $\beta_t(1)=\alpha(t)$ for each $t\in \ui$, we have $\beta(t)=H[\beta_t]\in B(H[\beta_t],U_{n,j})$ for each $t\in \left[\frac{j-1}{2^{n-1}},\frac{j}{2^{n-1}}\right]$. Therefore, there is a subdivision $\frac{j-1}{2^{n-1}}=s_0<s_1<\cdots <s_k=\frac{j}{2^{n-1}}$ such that $\beta([s_{i-1},s_i])\subseteq B(H[\beta_{s_{i-1}}],U_{n,j})$ for each $i=1,2,\dots,k$. In particular, there is a path $\zeta_i:\ui\to U_{n,j}$ from $\alpha(s_{i-1})$ to $\alpha(s_i)$ such that $H[\beta_{s_{i-1}}\cdot \zeta_i]=H[\beta_{s_i}]$.

Note that the concatenation $\alpha_{n,j}=\zeta_1\cdot \zeta_2\cdot\;\cdots \;\cdot\zeta_k$ is a path in $U_{n,j}$ from $\alpha\left(\frac{j-1}{2^{n-1}}\right)$ to $\alpha\left(\frac{j}{2^{n-1}}\right)$. Since $[\beta_{s_{i-1}}\cdot \zeta_i\cdot \beta_{s_i}^{-}]\in H$, we have
\begin{equation}
\left[\beta_{\frac{j-1}{2^{n-1}}}\cdot \alpha_{n,j}\cdot\beta_{\frac{j}{2^{n-1}}}^{-}\right]=\prod_{i=1}^{k}[\beta_{s_{i-1}}\cdot \zeta_i\cdot \beta_{s_i}^{-}]\in H\nonumber
\end{equation} for each $j=1,2,\dots,2^{n-1}$. Set $\alpha_n=\prod_{j=1}^{2^{n-1}}\alpha_{n,j}$.

If $n\in \bbn$ is such that $n(p)<n<n(p+1)$, put $\alpha_n=\alpha_{n(p+1)}$. By construction, the sequence $\alpha_n$ converges to $\alpha$ and satisfies $\alpha_n\left(\frac{j}{2^{n-1}}\right)=\alpha\left(\frac{j}{2^{n-1}}\right)$ for all $n\in\bbn$, $0\leq j\leq 2^{n-1}$. Thus we obtain a unique map $f:\bbd\to X$ such that $f\circ \lambda_{n}=\alpha_n$ and $f\circ \lambda_{\infty}=\alpha$. Observe that $f\circ \ell_{n,j}=\alpha_{n,j}$ whenever $n=n(p)$ for some $p\in \bbn$ and $1\leq j\leq 2^{n-1}$.

Finally, we check that $f_{\#}(\bbf)\leq H$ and $f_{\#}(\dinf)\notin H$. First, we claim that $H\left[f\circ \lambda_{n,m}\right]=H\left[\beta_{\frac{m}{2^{n-1}}}\right]$ for each dyadic unital pair $(n,m)$. Since \[f\circ \lambda_{n,m}=f\circ \left(\prod_{j=1}^{m}\ell_{n,j}\right)= f\circ \left(\prod_{j=1}^{m 2^{n(p)-n}}\ell_{n(p),j}\right)=f\circ \lambda_{n(p),m 2^{n(p)-n}}\] for some $n(p)\geq n$ and $p\geq 1$, we may assume that $n=n(p)$. In this case, $f\circ \lambda_{n,m}=\prod_{j=1}^{m}\alpha_{n,j}$. We have \begin{eqnarray*}
\left[\left(f\circ \lambda_{n,m}\right)\cdot \left(\beta_{\frac{m}{2^{n-1}}}\right)^{-}\right] &=&  \left[\beta_0\cdot (f\circ\lambda_{n,m})\cdot \left(\beta_{\frac{m}{2^{n-1}}}\right)^{-}\right]\\
&=&\left[\beta_0\right]\left(\prod_{j=1}^{m}\left[\alpha_{n,j}\right]\right) \left[\left(\beta_{\frac{m}{2^{n-1}}}\right)^{-}\right]\\
&=& \prod_{j=1}^{m}\left[\beta_{\frac{j-1}{2^{n-1}}}\cdot \alpha_{n,j}\cdot\left(\beta_{\frac{j}{2^{n-1}}}\right)^{-}\right] \in H
\end{eqnarray*}
showing that $H\left[f\circ \lambda_{n,m}\right]=H\left[\beta_{\frac{m}{2^{n-1}}}\right]$ as desired.

Whenever $\lambda_{m,n}\cdot\lambda_{m',n'}^{-}$ is a well-defined loop, we have $\frac{m}{2^{n-1}}=\frac{m'}{2^{n'-1}}$ and thus \[H\left[f\circ \lambda_{n,m}\right]=H\left[\beta_{\frac{m}{2^{n-1}}}\right]=H\left[\beta_{\frac{m'}{2^{n'-1}}}\right]=H\left[f\circ \lambda_{n',m'}\right]\]
This proves $f_{\#}([\lambda_{m,n}\cdot\lambda_{m',n'}^{-}])\in H$ whenever the loop is defined. Since the elements $[\lambda_{m,n}\cdot\lambda_{m',n'}^{-}]$ generate $\bbf$ by Lemma \ref{fchar}, we have $f_{\#}(\bbf)\leq H$. Finally, since $H[f\circ \lambda_1]=H[\beta_1]\neq H[\alpha] $, we have
$f_{\#}(\dinf)=f_{\#}([\lambda_{1}\cdot \lambda_{\infty}^{-}])=[(f\circ \lambda_{1})\cdot \alpha^{-}]\notin H$.
\end{proof}
If $X$ is a one-dimensional metric space, the following theorem of Eda implies that every homomorphism $\pioned\to\pionex$ is induced by a continuous map up to a change of basepoint. In this case, the unique path lifting property for $H\leq\pionex$ depends only on group theoretic conditions. If $\gamma:\ui\to X$ is a path from $x_0$ to $x_1$, let $\phi_{\gamma}:\pionex\to\pi_1(X,x_1)$ be the conjugation isomorphism $\phi_{\gamma}([\alpha])=[\gamma^{-}\cdot\alpha\cdot \gamma]$.
\begin{theorem}\label{edainducedthm}
\cite{Edaonedim} Let $Y$ be a one-dimensional Peano continuum, $X$ be a one-dimensional metric space, $x_0\in X$, and $y_0\in Y$. For each homomorphism $h:\pi_1(Y,y_0)\to\pionex$, there exists a continuous map $f:Y\to X$ and a path $\gamma:\ui\to X$ from $x_0$ to $f(y_0)$ such that $\phi_{\gamma}\circ h= f_{\#}$.
\end{theorem}
\begin{corollary}\label{onedlifing}
Suppose $X$ is a one-dimensional metric space and $H\leq \pionex$ is a subgroup. Then $p_H:\tXh\to X$ has the unique path lifting property if and only if every homomorphism $h:\pioned\to \pionex$ satisfying $h(\bbf)\leq H$ has image in $H$.
\end{corollary}
\begin{proof}
One direction follows immediately from Theorem \ref{uplchar}. If $p_H:\tXh\to X$ has the unique lifting property, then $H$ is $(\bbf,\dinf)$-closed. Let $h:\pioned\to \pionex$ be any homomorphism satisfying $h(\bbf)\leq H$. Since $\bbd$ is a one-dimensional Peano continuum, by Theorem \ref{edainducedthm}, there is a map $f:\bbd\to X$ and a path $\gamma:\ui\to X$ from $x_0$ to $f(d_0)$ such that $\phi_{\gamma}\circ h= f_{\#}$. By Proposition \ref{conjugatesforF}, the conjugate subgroup $H^{\gamma}=[\gamma^{-}]H[\gamma]$ is $(\bbf,\dinf)$-closed. Since $f_{\#}(\bbf)=\phi_{\gamma}( h(\bbf))\leq  \phi_{\gamma}(H)=H^{\gamma}$ and $\bbf$ is $(\bbf,\dinf)$-dense in $\pioned$, we have $f_{\#}(\pioned)\leq H^{\gamma}$. It follows that $h(\pioned)=\phi_{\gamma^{-}}(f_{\#}(\pioned))\leq \phi_{\gamma^{-}}(H^{\gamma})=H$.
\end{proof}
\section{Intermediate Generalized Coverings}\label{section5}
We begin our discussion of intermediate generalized coverings by contrasting it with the situation for traditional covering maps. Call a subgroup $H\leq \pionex$ a \textit{(generalized) covering subgroup} if there is a (generalized) covering map $p:(\widehat{X},\widehat{x})\to (X,x_0)$ such that $p_{\#}(\pi_1(\widehat{X},\widehat{x}))=H$. Recalling Theorem \ref{coveringtheorem}, it is a classical result of covering space theory that if $X$ is locally path connected then a subgroup $H\leq \pionex$ is a covering subgroup if and only if for every point $x\in X$, there is an open neighborhood $U_{x}\in\mct_{x}$ such that the normal subgroup $N=\lb \pi(x,U_x)\mid x\in X\rb\trianglelefteq\pionex$ is contained in $H$. In particular, $N$ itself is a covering subgroup and so is every subgroup $K$ with $N\leq K\leq \pionex$. Therefore the collection of covering subgroups of $\pionex$ is upward closed in the subgroup lattice of $\pionex$.

It also follows from the previous paragraph that the collection of covering subgroups is closed under finite intersection and that the core of a covering subgroup $H$ is a covering subgroup of $\pionex$. Despite these special cases, covering subgroups are not closed under arbitrary intersection. For example, $$\bigcap \{N\trianglelefteq\pioneh\mid N\text{ is a covering subgroup}\}=1$$ but $\bbh$ does not admit a universal covering space.

The situation for generalized coverings is quite different. The collection of generalized covering subgroups is closed under arbitrary intersection \cite{Brazcat} but is not upward closed in the subgroup lattice of $\pionex$. For example, $1\leq\pioneh$ is a generalized covering subgroup since $\bbh$ admits a generalized universal covering, while the free subgroup $F=\lb [\ell_n]\mid n\in \bbn\rb\leq\pioneh$ is not a generalized covering subgroup \cite{FZ07}. On the other hand, the core of a generalized covering subgroup is always a generalized covering subgroup, because it equals the intersection of conjugate generalized covering subgroups. In Theorem \ref{largeextensionthm} below, we identify a condition sufficient to conclude that if $N$ is a normal, generalized covering subgroup and $N\leq H$, then $H$ is also a generalized covering subgroup.
\begin{example}
Recall the $(\cinfty,c_{\infty})$-closed subgroup $CCP(\bbd,B,d_0)\leq\pioned$ constructed in Example \ref{countablecutpointsexample}. We claim that $CCP(\bbd,B,d_0)$ is not $(\bbf,\dinf)$-closed. Since $\bbf\leq CCP(\bbd,B,d_0)$, it suffices to check that $\dinf\notin CCP(\bbd,B,d_0)$. If $\dinf=[\lambda_1\cdot\lambda_{\infty}^{-}]\in CCP(\bbd,B,d_0)$, then there is a loop $\beta$ which is homotopic to the reduced loop $\alpha=\lambda_1\cdot\lambda_{\infty}^{-}$ such that $\beta^{-1}(B)$ is countable. However the reduced loop $\alpha$ must be obtained by contracting subpaths of $\beta$ within its own image. Thus $\alpha^{-1}(B)$ must be countable; a contradiction.
\end{example}
\begin{example}\label{normalsubgroupexample}
The construction of a normal subgroup $N\trianglelefteq\pioned$ with the same properties as the subgroup in the previous example is a bit more involved. However this is done in \cite{VZ13} using a "triangle-space" $T$, which is homotopy equivalent to $\bbd$. We refer to this paper for detailed proofs. Call a path $\alpha:\ui\to\bbd$ \textit{generic} if there exists a countable, closed set $A\subset \ui$ such that the set of components of $\ui\backslash A$ may be written as a disjoint union $C_0\cup C_{1}^{+}\cup C_{1}^{-}$ where
\begin{enumerate}
\item if $(a,b)\in C_0$, then $\alpha([a,b])\cap B$ is closed and nowhere dense in $B$,
\item there is bijection $\theta:C_{1}^{+}\to C_{1}^{-}$ such that if $\theta((a,b))=(c,d)$, then $\alpha|_{[a,b]}\equiv\left(\alpha|_{[c,d]}\right)^{-}$.
\end{enumerate}
The subgroup $N=\{[\alpha]\in\pioned\mid\alpha\text{ is generic}\}$ is a normal, $(\cinfty,c_{\infty})$-closed subgroup of $\pioned$ and thus $(\cinfty,c_{\tau})$-closed by Proposition \ref{normalsubgroup}. While $D\leq N$, the arguments in \cite{VZ13} show that $\dinf\notin N$. Therefore, $N$ is not $(\bbf,\dinf)$-closed.
\end{example}
If $H,K\leq G$ are subgroups, let $K^H= \bigcup_{h\in H}h^{-1}Kh$. For instance, recall that $\pi(x,U)=\left\langle\pi(\alpha,U)^{\pionex}\right\rangle=\langle \pi(\beta,U)\mid\beta(1)=x\rangle$ is the normal closure of $\pi(\alpha,U)$ in $\pionex$.
\begin{proposition}\label{largeextension}
Suppose $X$ is a metric space and $N\leq H\leq \pionex$ where $N$ is a normal subgroup of $\pionex$ and $p_N:\tX_N\to X$ has the unique path lifting property. If, for every path $\alpha:\ui\to X$ with $\alpha(0)=x_0$, there is a $U\in\mct_{\alpha(1)}$ such that $\pi(\alpha,U)^{H}\cap H\subseteq N$, then $p_H:\tXh\to X$ has the unique path lifting property.
\end{proposition}

\begin{proof}
Let $f:(\bbd,d_0)\to (X,x_0)$ be such that $f_{\#}(\bbf)\leq H$. By Theorem \ref{uplchar}, it suffices to show that there is an $m\geq 1$ such that $f_{\#}([\lambda_{m}\cdot\lambda_{\infty}^{-}])\in H$. Identify $\ui$ with the base arc $B$. For each $(n,j)\in \scrd$, let $\bbd_{n,j}$ be the homeomorphic copy of $\bbd$ bounded by the arc $\bbd(n,j)$ and the interval $\left[\frac{j-1}{2^{n-1}},\frac{j}{2^{n-1}}\right]\subseteq B$. Recall the canonical homeomorphism $T_{n,j}:\bbd_{n,j}\to \bbd$ from Lemma \ref{fistransitive}.

For each $t\in B$ there is an open neighborhood $U_t$ of $f(t)$ such that $\pi(f\circ\delta_t,U_t)^{H}\cap H\subseteq N$. Let $V_t$ be an open ball centered at $t$ in $\bbd$ such that $f(V_t)\subseteq U_t$. Take $0\leq t_1<t_2<\cdots<t_k\leq 1$ such that the sets $V_{t_1},V_{t_2},\dots,V_{t_k}$ cover $B$. There exists an $m\geq 1$ such that for each $j=1,2,\dots,2^{m-1}$, we have $\bbd_{m,j}\subseteq V_{t_i}$ for some $i$. For the moment, fix such $j$ and $i$. Put $s=\frac{j-1}{2^{m-1}}$, $\alpha_j=f\circ\delta_{s}$, and let $f_j:\bbd_{m,j}\to X$ denote the restriction of $f$. By Lemma \ref{deltafinite}, there is a path $\epsilon:\ui\to V_{t_i}$ from $t_i$ to $s$ and an element $L\in\bbf$ such that $[\delta_{s}]=L^{-1}[\delta_{t_i}\cdot \epsilon]$. Consider a free generator $d_{n,k}=[\delta_{q_1}\cdot\ell_{n+1,2k}\cdot\delta_{q_2}^{-}]$ of $\bbf$ where $q_1=\frac{2k-1}{2^n}$ and $q_2=\frac{k}{2^{n-1}}$. We have $[\delta_s](T_{m,j}^{-1})_{\#}(d_{n,k})[\delta_{s}^{-}]\in D$ since $\delta_s\cdot T_{m,j}^{-1}\circ(\delta_{q_1}\cdot\ell_{n+1,2k}\cdot\delta_{q_2}^{-})\cdot \delta_{s}^{-}$ is a finite concatenation of standard paths. Also,\[[\delta_s](T_{m,j}^{-1})_{\#}(d_{n,k})[\delta_{s}^{-}]=
L^{-1}[\delta_{t_i}][\epsilon](T_{m,j}^{-1})_{\#}(d_{n,k})[\epsilon^{-}][\delta_{t_i}^{-}]L\]is an element of $ L^{-1}\pi(\delta_{t_i},V_{t_i})L\subseteq \pi(\delta_{t_i},V_{t_i})^{\bbf}$. Applying the homomorphism $f_{\#}$ and recalling that $f_{\#}(D)\leq H$, we have \[[\alpha_j](f_j\circ T_{m,j}^{-1})_{\#}(d_{n,k})[\alpha_{j}^{-}]\in \pi(f\circ \delta_{t_i},U_{t_i})^{H}\cap H\subseteq N.\]
Thus $f_j\circ T_{m,j}^{-1}:\bbd\to X$ is a map satisfying $(f_j\circ T_{m,j}^{-1})_{\#}(\bbf)\leq N^{\alpha_j}$ for $1\leq j\leq 2^{m-1}$. Since $N$ is $(\bbf,\dinf)$-closed by assumption, $N^{\alpha_j}$ is $(\bbf,\dinf)$-closed by Proposition \ref{conjugatesforF}. Thus, for each $j$, we have $(f_j\circ T_{m,j}^{-1})_{\#}(\dinf)\in N^{\alpha_j}$. Let $\beta_j$ be the restriction of $\lambda_{\infty}$ to $\left[\frac{j-1}{2^{n-1}},\frac{j}{2^{n-1}}\right]\subseteq B$. Then $f_{\#}([\delta_s][\ell_{m,j}\cdot \beta_{j}^{-}][\delta_{s}^{-}])\in N$ for each $s=\frac{j-1}{2^{m-1}}$. Since $N$ is a normal subgroup and $[\lambda_{m}\cdot\lambda_{\infty}^{-}]$ factors as a product of conjugates of the elements $[\delta_s][\ell_{m,j}\cdot \beta_{j}^{-}][\delta_{s}^{-}]$, we see that $f_{\#}([\lambda_{m}\cdot\lambda_{\infty}^{-}])\in N\leq H$.
\end{proof}

Since $\pi(\alpha,U)^{H}\subseteq \pi(x,U)$ for any path $\alpha$ with $\alpha(1)=x$ and any subgroup $H\leq \pionex$, we obtain the following theorem.

\begin{theorem}\label{largeextensionthm}
Suppose $X$ is a metric space and $N\leq H\leq \pionex$ where $N$ is a normal subgroup of $\pionex$ and $p_N:\tX_N\to X$ has the unique path lifting property. If, for every point $x\in X$, there is a $U\in\mct_{x}$ such that $\pi(x,U)\cap H\leq N$, then $p_H:\tXh\to X$ has the unique path lifting property.
\end{theorem}

\begin{remark}
The condition given in Theorem \ref{largeextensionthm} may be interpreted as follows: if a normal $(\bbf,\dinf)$-closed subgroup $N\leq \pionex$ is enlarged to a subgroup $H\leq \pionex$ by adding only homotopy classes of loops which are ``large" in the sense that they do not factor as products of conjugates of arbitrarily small loops based at some fixed point, then $H$ must also be $(\bbf,\dinf)$-closed.
\end{remark}

\begin{example}\label{subgroupsexample}
Since $\bbd$ is a 1-dimensional Peano continuum it admits a generalized universal covering \cite{FZ07}, which makes the trivial subgroup $1\leq \pioned$ a $(\bbf,\dinf)$-closed subgroup. The subgroup $\bbs\leq \pioned$ is not $(\bbs,\dinf)$-closed since it does not contain $\dinf$. However, $1\leq \bbs$ and no non-trivial element of $\bbs$ has a representative of the form $\prod_{i=1}^{n}\alpha_{i}\cdot\delta_{i}\cdot\alpha_{i}^{-}$ with loops $\delta_i$ based at the same $x\in X$ and $diam(\delta_i)<1$. Thus Theorem \ref{largeextensionthm} implies that $\bbs$ is $(\bbf,\dinf)$-closed. We conclude that $p_{\bbs}:\wt{\bbd}_{\bbs}\to\bbd$ is a generalized covering map, which is not a covering map since $\bbs$ does not satisfy the necessary condition given in Theorem \ref{coveringtheorem}.
\end{example}
\section{Generalized universal coverings}\label{section6}
Consider $\bbd$ as the subspace $\bbd\times \{0\}\subseteq \bbr^3$. The dyadic unital pairs $(n,j)$ are in bijective correspondence with the open disks $e_{n,j}\subseteq \bbr^2\times\{0\}$ which are the bounded components of $(\bbr^2\backslash \bbd)\times \{0\}$ in $\bbr^2\times \{0\}$. We construct the space $\bbd\bba\subseteq \bbr^3$ by continuously raising a point in each disk $e_{n,j}$ up to unit height while leaving $\bbd$ unchanged (see Figure \ref{archipelagofig}). Note that $\bbd\bba$ is homotopy equivalent to the relative CW-complex in the weak topology obtained by attaching a 2-cell to $\bbd$ using the loop $\ell_{n,j}\cdot (\ell_{n+1,2j})^{-}\cdot (\ell_{n+1,2j-1})^{-}$ as an attaching map for each dyadic unital pair $(n,j)$. Thus $\pi_{1}(\bbd\bba,d_0)\cong \pioned/N$ where $N$ is the normal closure of $\bbf$ in $\pioned$.
\begin{figure}[H]
\centering \includegraphics[height=1.7in]{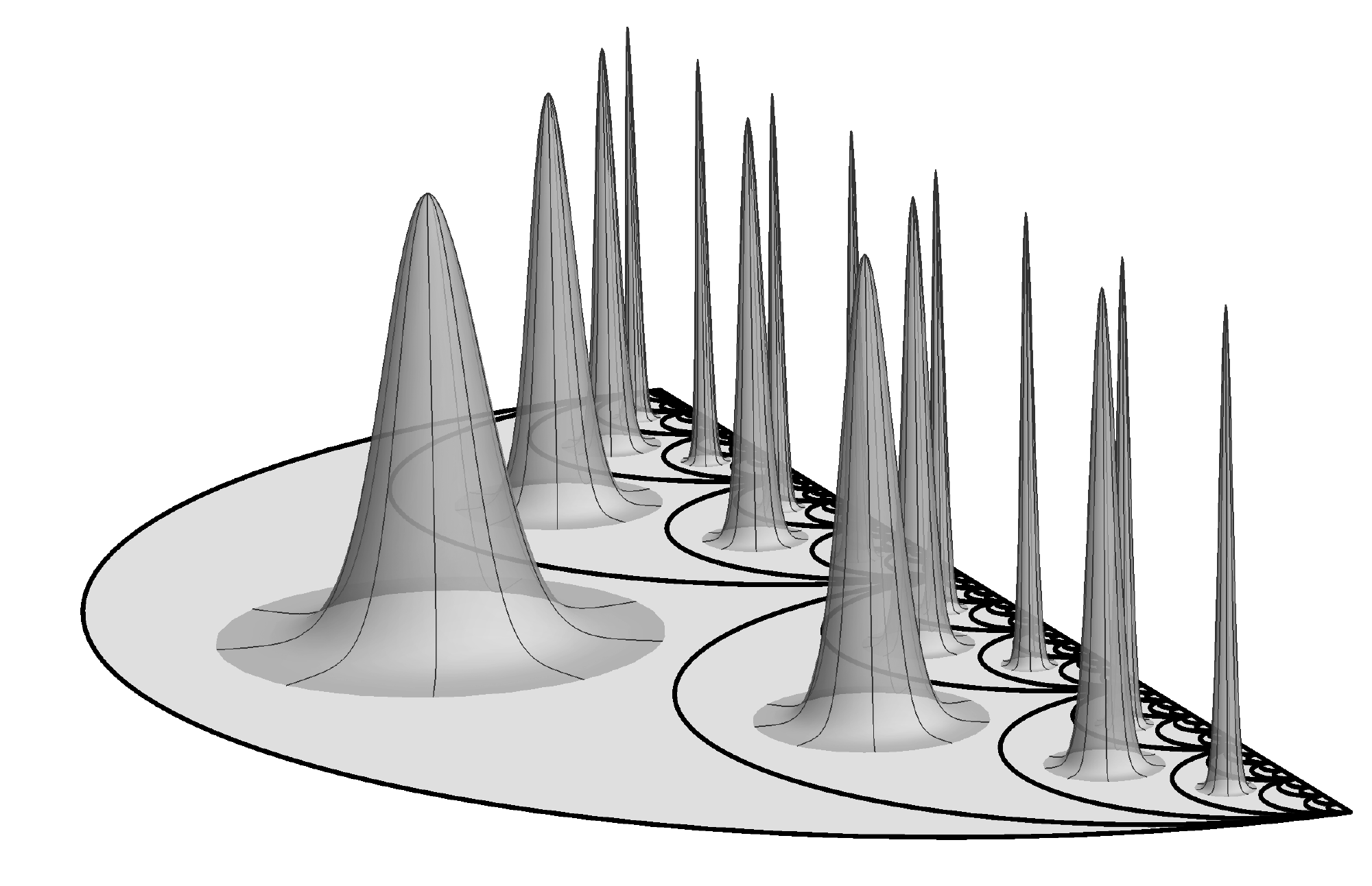}
\caption{\label{archipelagofig}The space $\bbd\bba$}
\end{figure}
\begin{theorem}\label{mainresultcomplete}
The following are equivalent for any path-connected metric space $X$.
\begin{enumerate}
\item $X$ admits a generalized universal covering,
\item every map $f:\bbd\to X$ such that $f_{\#}(\bbf)=1$ induces the trivial homomorphism on $\pi_1$,
\item every map $g:\bbd\bba\to X$ induces the trivial homomorphism on $\pi_1$.
\end{enumerate}
\end{theorem}
\begin{proof}
(1) $\Leftrightarrow$ (2) is a direct combination of Theorem \ref{uplchar}, Lemma \ref{fistransitive}, and Corollary \ref{denseapplicationcor}. (2) $\Leftrightarrow$ (3) is evident from the fact that a map $f:\bbd\to X$ extends to a map $g:\bbd\bba\to X$ if and only if $f_{\#}(\bbf)=1$.
\end{proof}
\begin{definition}
Let $Cov(X)$ denote the set of all open covers of $X$, inversely directed by refinement. For $\scru\in Cov(X)$, the \textit{Spanier group of $X$ relative to $\scru$} is the subgroup $\pi(\scru,x_0)=\lb \pi(x,U)\mid x\in U\in \scru\rb$. The \textit{Spanier group }of $X$ is the intersection $\pi^{s}(X,x_0)=\bigcap_{\scru\in Cov(X)}\pi(\scru,x_0)$.
\end{definition}
If $X$ is locally path connected, then $X$ is semilocally simply connected if and only if there is an open cover $\scru$ of $X$ such that $\pi(\scru,x_0)=1$. It is shown in \cite{FZ07} that if $\pi^{s}(X,x_0)=1$, then $X$ admits a generalized universal covering. Indeed, for every path-connected space $X$ and every open cover $\scru$ of $X$, $\pi(\scru,x_0)$ is an $(S,d_{\infty})$-closed subgroup of $\pi_1(X,x_0)$ (the straightforward proof is similar to that of Proposition \ref{homomorphicallyhausdorffprop} below). Therefore, the intersection $\pi^s(X,x_0)$ is $(S,d_{\infty})$-closed. We also note that for locally path-connected $X$, $\pi^s(X,x_0)$ equals the kernel of the canonical homomorphism $\pi(X,x_0)\rightarrow \check{\pi}_1(X,x_0)$ \cite{BF13}.
\begin{definition}\label{slenderdef}
Let $G$ be a group.
\begin{enumerate}
\item We call $G$ \textit{noncommutatively slender} (or \textit{n-slender} for short) if for every homomorphism $h:\pioneh\to G$, there is an $N$ such that $h([\alpha])=1$ for all $\alpha\in \Omega(\bbh_{\geq N},b_0)$; equivalently, $G$ is n-slender if and only if for every Peano continuum $X$ and homomorphism $h:\pi_1(X,x_0)\to G$, there exists an open cover $\scru$ of $X$ such that $h(\pi(\scru,x_0))=1$ \cite{EdaAlgTopofPeano}.
\item We call $G$ {\it residually n-slender} if for every $g\in G\backslash \{1\}$, there is an n-slender group $K$ and a homomorphism $k:G\to K$ such that $k(g)\neq 1$.
\item We call $G$ {\it homomorphically Hausdorff relative to a space $X$} if for every homomorphism $h:\pionex\to G$, we have $\bigcap_{\scru\in Cov(X)}h(\pi(\scru,x_0))=1$.
\end{enumerate}
\end{definition}
Observe that if $X$ is path connected, locally path connected and first countable, and if $\pionex$ is n-slender, then $X$ is semilocally simply connected so that $X$ admits a classical universal covering.

Every free group is n-slender \cite{Edafreesigmaproducts} and certainly every n-slender group is residually n-slender. If $G$ is residually n-slender, then $G$ is homomorphically Hausdorff relative to every Peano continuum \cite{EdaFischer}. Therefore, if $X$ is a Peano continuum and $\pionex$ is residually n-slender, then we have $\pi^{s}(X,x_0)=\bigcap_{\scru\in Cov(X)}id_{\#}(\pi(\scru,x_0))=1$, which implies that $X$ admits a generalized universal covering. Using the test space $\bbd$, we extend this result to all metric spaces.
\begin{proposition}\label{homomorphicallyhausdorffprop}
If $X$ is metrizable and $\pionex$ is homomorphically Hausdorff relative to $\bbd$, then $X$ is homotopically path-Hausdorff.
\end{proposition}
\begin{proof}
Suppose $\pionex$ is homomorphically Hausdorff relative to $\bbd$. By Theorem \ref{mainresultcomplete}, we may check that the trivial subgroup is $(S,\dinf)$-closed. Let $f:(\bbd,d_0)\to (X,x_0)$ be a based map such that $f_{\#}(S)=1$. Fix an open cover $\scru\in Cov(\bbd)$. There is an $n\geq 1$ such that $[\lambda_n\cdot\lambda_{\infty}^{-}]\in \pi(\scru,d_0)$. Since $[\lambda_1\cdot\lambda_{n}^{-}]\in S$, we have $f_{\#}(\dinf)=f_{\#}([\lambda_1\cdot\lambda_{n}^{-}])f_{\#}([\lambda_n\cdot\lambda_{\infty}^{-}])= f_{\#}([\lambda_n\cdot\lambda_{\infty}^{-}])\in f_{\#}(\pi(\scru,d_0))$. Thus $f_{\#}(\dinf)\in f_{\#}(\pi(\scru,d_0))$ for every $\scru\in Cov(\bbd)$. By assumption, $\bigcap_{\scru\in Cov(\bbd)}f_{\#}(\pi(\scru,x_0))=1$; therefore $f_{\#}(\dinf)=1$.
\end{proof}
\begin{corollary}\label{nslendercor}
If $X$ is a metric space and $\pionex$ is residually n-slender, then $X$ admits a generalized universal covering.
\end{corollary}

\begin{definition}\label{oneuvzerodef}
A space $X$ is $1$-$UV_0$ at $x\in X$ if for every neighborhood $U$ of $x$ there is an open set $V$ in $X$ with $x\in V\subseteq U$ and such that for every map $f:D^2\to X$ with $f(S^1)\subseteq V$, there is a map $g:D^2\to U$ with $f|_{S^1}=g|_{S^1}$. We say that $X$ is $1$-$UV_0$ if $X$ is $1$-$UV_0$ at every point $x\in X$.
\end{definition}
\begin{remark}
Unlike many of the properties considered in this paper, the $1$-$UV_0$ property is not an invariant of homotopy type. Indeed, the cone $C\bbh=\bbh\times[0,1]/\bbh\times\{1\}$ over the Hawaiian earring is homotopy equivalent to the one-point space but is not $1$-$UV_0$.
\end{remark}

The authors of \cite{CMRZZ08} show that $X$ is homotopically Hausdorff at $x\in X$ whenever $X$ is $1$-$UV_0$ at $x$. We improve upon this result in Theorem \ref{uvzeroimpliesupl} below. Note the resemblance of the following characterization to Corollary \ref{homhauschar}.

\begin{proposition}\label{uv0char}
Identify $\bbh$ with a subspace of the unit disk $D^2$ so that the outermost circle $C_1\subseteq\bbh$ is identified with $S^1$. For any space $X$, (1) $\Rightarrow$ (2) $\Leftrightarrow$ (3) holds for the following properties. If $X$ is first countable and locally path connected, then all three are equivalent.
\begin{enumerate}
\item $X$ is $1$-$UV_0$ at $x\in X$,
\item every map $f:(\bbh,b_0)\to (X,x)$ such that $f_{\#}(\finfty)=1$ extends to a map $g:D^2\to X$,
\item for every map $f:(\bbh\bba,b_0)\to (X,x)$, $f|_{\bbh}:\bbh\to X$ extends to a map $g:D^2\to X$.
\end{enumerate}
\end{proposition}
\begin{proof}
The proof of (1) $\Rightarrow$ (3) is identical to the proof of Lemma 4.1 in \cite{CMRZZ08}. (3) $\Rightarrow$ (2) is obvious since every map $f:\bbh\to X$ satisfying $f_{\#}(\finfty)=1$ extends to a map on the harmonic archipelago. (2) $\Rightarrow$ (3) Given a map $f:(\bbh\bba,b_0)\to (X,x)$, define $g:\bbh\to X$ by $g\circ \ell_n=f\circ (\ell_{n}\cdot \ell_{n+1}^{-})$. Then $g_{\#}(F)=1$ and so by assumption, $g$ extends to a map on $D^2$. In particular, we have null homotopies $h_n:D^2\to X$ such that $h_n|_{S^1}\equiv f\circ (\ell_{n}\cdot \ell_{n+1}^{-})$ and every neighborhood of $x$ contains all but finitely many of the images $h_n(D^2)$. We can now define a continuous extension $f':D^2\to X$ of $f|_{\bbh}$ by filling in the component of $D^2\backslash \bbh$ between $C_n$ and $C_{n+1}$ using $h_n$.

Finally, to prove (2) $\Rightarrow$ (1) we suppose $X$ is first countable and locally path connected. Suppose $X$ is not $1$-$UV_0$ at $x.$ Then there is an open neighborhood $U$ of $x$, a countable basis of path-connected neighborhoods $\cdots\subseteq U_3\subseteq U_2\subseteq U_1=U$ at $x$, and loops $\gamma_n:S^1\to U_n$ which are inessential in $X$ but which are essential in $U$. Since $U_n$ is path connected, we may assume $\gamma_n$ is based at $x$. Define a map $f:(\bbh,b_0)\to (X,x)$ by $f\circ \ell_n=\gamma_n$ and note that $f_{\#}(\finfty)=1\leq \pi_1(X,x)$. However, if $f$ extended to a map $g:D^2\to X$, there would be an $m$ such that $\gamma_m$ is inessential in $U$. Thus no such $g$ can exist.
\end{proof}
\begin{theorem}\label{uvzeroimpliesupl}
If $X$ is metrizable and $1$-$UV_0$, then $X$ admits a generalized universal covering.
\end{theorem}
\begin{proof}
Fix a metric generating the topology of $X$. By Theorem \ref{uplchar}, it suffices to show that the trivial subgroup of $\pionex$ is $(\bbf,\dinf)$-closed. Suppose $X$ is $1$-$UV_0$ and $f:(\bbd,d_0)\to (X,x)$ is a map such that $f_{\#}(\bbf)=1$. If $\beta_{n,j}:S^1\to \bbd$ is the loop defined as $\ell_{n,j}\cdot\ell_{n+1,2j}^{-}\cdot\ell_{n+1,2j-1}^{-}$, then $f\circ \beta_{n,j}$ is inessential in $X$. Let $E_{n,j}$ be the set of extensions $h:D^2\to X$ of $f\circ\beta_{n,j}$. Suppose there exists $h_{n,j}\in E_{n,j}$ such that $s_n=\max\{diam(h_{n,j}(D^2))\mid j=1,2,\dots,2^{n-1}\}\to 0$. In this case, $f$ extends to a map on $\{(x,y)\in\bbr^2\mid (x-1/2)^2+y^2\leq \frac{1}{4},y\geq 0\}$ showing that $f$ is null-homotopic; consequently $f_{\#}(\dinf)=1$. We show the other case cannot occur. Suppose there exists $\epsilon>0$ and $(n_k,j_k)\in \scrd$ where $n_1<n_2<n_3<\cdots$ such that every extension of $f\circ\beta_{n_k,j_k}$ to $D^2$ has diameter $>\epsilon$. Replacing $(n_k,j_k)$ with a subsequence if necessary, we may assume that $\beta_{n_k,j_k}$ converges uniformly to the constant loop at some point $(t,0)\in B$. Pick arcs $\alpha_{k}$ from $(t,0)$ to $\left(\frac{j_k-1}{2^{n_k-1}},0\right)$ in $B$ and observe that $diam(\alpha_k(\ui))\to 0$. Define a map $f':\bbh\to X$ by $f'\circ\ell_{k}=f\circ(\alpha_{k}\cdot  \beta_{n_k,j_k}\cdot\alpha_{k}^{-})$. It is clear that $(f')_{\#}(\finfty)=1$ yet $f'$ cannot extend to a map on $D^2$ as in Proposition \ref{uv0char}; a contradiction.
\end{proof}
\begin{example}
The Peano continua $A$ and $B$ in \cite{CMRZZ08} (also appearing in \cite{FRVZ11}) are sometimes referred to as the ``sombrero spaces." In \cite{CMRZZ08}, both spaces are shown to either have the $1$-$UV_0$ property or a stronger property. Theorem \ref{uvzeroimpliesupl} implies that both spaces admit a generalized universal covering space. Thus $B$ is an example of a Peano continuum which is not homotopically path Hausdorff \cite[Prop. 3.4]{FRVZ11} but which admits a generalized universal covering space.
\end{example}
\section{Transfinite path products}\label{section7}

The homotopically Hausdorff and transfinite product properties describe local wildness at points. Consequently, the corresponding test space $\bbhp$ has a single wild point. On the other hand, the homotopically path-Hausdorff and unique path-lifting properties describe local wildness of paths, forcing the corresponding test space $\bbd$ to have a continuum of wild points. In this section, we introduce an intermediate property lying between these point-local and path-local properties. This new property is characterized by a closure pair  $(W,w_{\infty})$ such that the wild points of the test space $\bbb$ form a Cantor set. The primary benefit of this intermediate property is that it allows us to break down technical arguments and prove the most challenging partial converses that appear in the results diagram of Section \ref{resultssection}. Specifically, \ref{transfinitepathproductcomparisonprop}, \ref{tpdtheorem}, and \ref{discretewildsettheorem} below contribute to the diagram.  
\begin{definition}\label{tppdef}
A space $X$ has {\it transfinite path products relative to }$H\leq \pionex$ if for every closed set $A\subseteq \ui$ containing $\{0,1\}$ and paths $p,q:(\ui,0)\to (X,x_0)$ such that $p|_A=q|_A$ and $[p|_{[0,b]}\cdot q|_{[a,b]}^{-}\cdot p|_{[0,a]}^{-}]\in H$ for every component $(a,b)$ of $\ui\backslash A$, we have $[p\cdot q^{-}]\in H$. A space $X$ has {\it transfinite path products} if $X$ has transfinite path products relative to the trivial subgroup $H=1$.
\end{definition}

\begin{remark}\label{nowheredensermk}
Suppose that $A\subseteq \ui$ is closed and paths $p,q$ satisfy the hypotheses of Definition \ref{tppdef}. If $B$ is the boundary of $A$ in $\ui$, then $B$ is nowhere dense and closed in $\ui$, $p|_{B}=q|_{B}$, and if $(c,d)$ is any component of $\ui\backslash B$, then we still have $[p|_{[0,c]}\cdot q|_{[c,d]}^{-}\cdot p|_{[0,c]}^{-}]\in H$. Therefore, to verify that a space $X$ has transfinite path products relative to $H$, it suffices to check the statement of the definition for nowhere dense, closed subsets $A\subseteq \ui$.
\end{remark}

Let $\mcc\subseteq \ui$ be the standard middle third cantor set and $\mathcal{I}$ be the set of components of $\ui\backslash \mcc$ with the natural ordering inherited from $\ui$ (which is order isomorphic to $\bbq$). For each interval $I=\left(\frac{j-1}{3^{n-1}},\frac{j}{3^{n-1}}\right)\in\mathcal{I}$, let $\bbb_I$ be the upper-half of the circle of radius $\frac{1}{2(3^{n-1})}$ centered at $\left(\frac{2j-1}{2(3^{n-1})},0\right)$. Let $\bbb$ be the union of these semicircles and the base arc $B=[0,1]\times\{0\}$ (see Figure \ref{spacewfig}).

\begin{figure}[H]
\centering \includegraphics[height=0.9in]{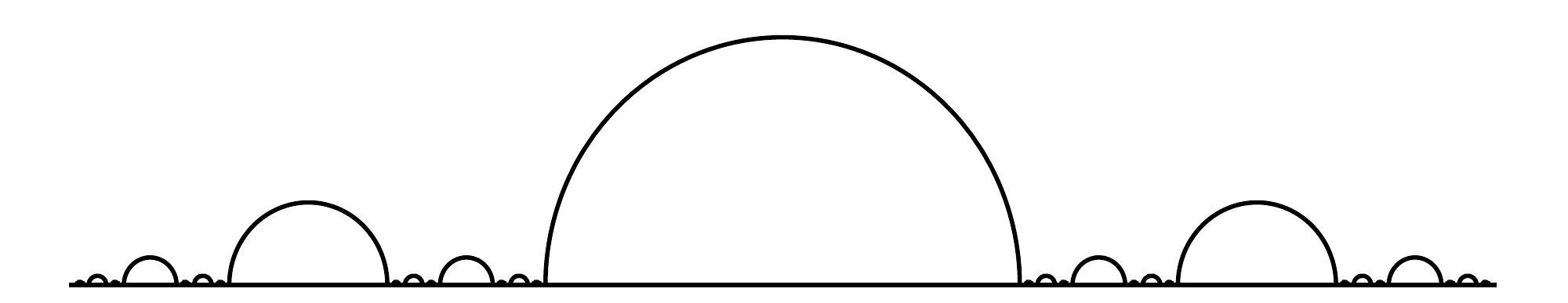}
\caption{\label{spacewfig}The space $\bbb$}
\end{figure}
Let $\linf:\ui\to\bbb$, $\linf(t)=(t,0)$ be the arc along the base arc and $\uinf:\ui\to\bbb$ be the arc along the upper portion of $\bbb$, that is $\uinf(t)=(t,0)$ for $t\in \mcc$ and if $I=(a,b)\in\mathcal{I}$, then $\uinf|_{[a,b]}$ is the arc along $\bbb_{I}$ from $(a,0)$ to $(b,0)$.

\begin{definition}\label{defW}
Let $\bbk\leq \pi_1(\bbb,d_0)$ be the subgroup generated by the set $\{[\uinf|_{[0,b]}\cdot \linf|_{[a,b]}^{-}\cdot \uinf|_{[0,a]}^{-}]\mid (a,b)\in \mathcal{I}\}$ and let $w_{\infty}=[\uinf\cdot \linf^{-}]$.
\end{definition}

We may identify $\bbb^+$ as a subspace of $\bbb$ and define a retraction $r:\bbb\to \bbb^+$ such that $r_{\#}(\bbk)\leq \bbk^+$ and $r_{\#}(w_{\infty})=w_{\infty}^{+}$. Thus the proof of Proposition \ref{conjugatesforS} may be modified for $\bbb$.
\begin{proposition}\label{conjugatesforK}
$(\bbk,w_{\infty})$ is a normal closure pair for $(\bbb,d_0)$.
\end{proposition}
\begin{proposition}\label{transfinitepathprodchar}
$X$ has transfinite path products relative to $H\leq \pionex$ if and only if $H$ is $(\bbk,w_{\infty})$-closed.
\end{proposition}
\begin{proof}
One direction is straightforward. Suppose $H$ is $(\bbk,w_{\infty})$-closed, $\{0,1\}\subseteq A\subseteq \ui$ where $A$ is closed and $p,q:(\ui,0)\to (X,x_0)$ are paths such that $p|_A=q|_A$ and $[p|_{[0,d]}\cdot q|_{[c,d]}^{-}\cdot p|_{[0,c]}^{-}\in H$ for every component $(c,d)$ of $\ui\backslash A$. By Remark \ref{nowheredensermk}, we may assume $A$ is nowhere dense in $\ui$. Find a non-decreasing, continuous, surjection $h:\ui\to \ui$ mapping the middle third cantor set $\mcc$ onto $A$ and such that every component of $\ui\backslash \mcc$ is either mapped homeomorphically onto some component of $\ui\backslash A$ or mapped to a point. Define $f:\bbb\to X$ so that $f\circ \uinf=p\circ h$ and $f\circ \linf=q\circ h$. Since $p\circ h$ and $p\circ h$ agree on $\mcc$, $f$ is well-defined. Continuity at each point of $\mcc\times\{0\}$ follows directly from the continuity of $p$ and $q$. Fix $(a,b)\in\mathcal{I}$ and let $k=[\uinf|_{[0,a]}\cdot \uinf|_{[a,b]}\cdot \linf|_{[a,b]}^{-}\cdot \uinf|_{[0,a]}^{-}]$ be the corresponding generator of $\bbk$. If $h$ maps $(a,b)$ to a point, then $f_{\#}(k)=1$. If $h$ maps $(a,b)$ onto a component $(c,d)$ of $\ui\backslash A$, then $f_{\#}(k)=[p|_{[0,c]}\cdot p|_{[c,d]}\cdot q|_{[c,d]}^{-}\cdot p|_{[0,c]}^{-}]\in H$. This proves $f_{\#}(\bbk)\leq H$ which allows us to conclude that $[p\circ q^{-}]=f_{\#}(w_{\infty})\in H$.
\end{proof}
\begin{proposition}\label{transfinitepathproductcomparisonprop}
Let $H\leq\pionex$ be a subgroup.
\begin{enumerate}
\item If $H$ is $(\bbf,\dinf)$-closed, then $H$ is $(\bbk,w_{\infty})$-closed.
\item If $H$ is $(\bbk,w_{\infty})$-closed, then $H$ is $(\cinfty,c_{\tau}$)-closed.
\item If $H$ is normal and $(\bbk,w_{\infty})$-closed, then $H$ is $(\pinfty,p_{\tau})$-closed.
\end{enumerate}
\end{proposition}
\begin{proof}
(1) We view $\bbb$ as a specific retract of $\bbd$ and apply Remark \ref{comparisonpropremark}. For a dyadic unital pair $(n,j)$, let $I_{n,j}=\left(\frac{j-1}{2^{n-1}},\frac{j}{2^{n-1}}\right)$. Consider the following recursively defined subset $A\subseteq \scrd$ of dyadic unital pairs: $A_3=\{(3,2)\}$ and \[A_{n+2}=\{(n+2,4j-2)\mid\text{for all }m<n\text{ and }(m,i)\in A_m, I_{n+2,4j-2}\cap I_{m,i}=\emptyset\}.\] Set $A=A_3\cup A_5\cup A_7\cup \cdots$. The intervals $\{I_{n,j}\mid(n,j)\in A\}$ are disjoint and have dense union in $[0,1]$. Consequently, the subspace $\bbb '=B\cup \bigcup_{(n,j)\in A}\bbd(n,j)\subseteq \bbd$ is a homeomorphic copy of $\bbb$ (see Figure \ref{embedding}). Let $\uinf '(t)=(t,v(t))$ be the arc along $\bigcup_{(n,j)\in A}\bbd(n,j)$ from $d_0$ to $(1,0)$. Let $s:\bbb\to \bbb '$ be a homeomorphism such that $s\circ \linf \equiv \linf $ and $s\circ \uinf\equiv\uinf '$. Define a retraction $r:\bbd\to \bbb '$ by downward vertical projection: Given $(t,y)\in \bbd$ if $y\geq v(t)$, let $r(t,y)=\uinf '(t)$ and if $0\leq y<v(t)$, let $r(t,y)=(t,0)$.
By construction, $r|_{\bbb '}=id_{\bbb '}$. Notice that $r_{\#}(\dinf)=[\uinf ' \cdot (\linf ')^{-}]=s_{\#}(w_{\infty})$. Therefore, it suffices to show $r_{\#}(\bbf)\leq s_{\#}(\bbk)$. Recall that the free generator $d_{n,j}$ of $\bbf$ is the homotopy class of the loop $L_{n,j}=\left(\delta_{\frac{2j-1}{2^n}}\right)\cdot \left(\ell_{n+1,2j}\right)\cdot \left(\delta_{\frac{j}{2^{n-1}}}\right)^{-}$.

For the moment, fix a dyadic rational $t\in [0,1]$. Given the construction of the path $\delta_t$, it is clear that part of the image of $\delta_t$ lies strictly below the image of $\uinf '$ if and only if $t\in U=\bigcup_{(n,j)\in A}I_{n,j}$. If $t\in I_{n,j}=(a,b)$, then $\delta_{t}\equiv \delta_{a}\cdot \zeta$ where $\zeta|_{(0,1]}$ lies strictly below the arc $\ell_{n,j}$. So, if $t\in U$, then $r\circ \delta_{t}\equiv \uinf '|_{[0,a]}\cdot \linf|_{[a,t]}$. On the other hand, if $t\notin U$, then $\delta_t$ has image either on or above the image of $\uinf '$ and $r\circ \delta_{t}\equiv \uinf '|_{[0,t]}$.

Now fix a dyadic unital pair $(n,j)$. We claim that $r_{\#}(d_{n,j})\in s_{\#}(\bbk)$. There are three cases to consider:

Case I: Suppose $\ell_{n+1,2j}$ has image on or above the image of $\uinf '$. Then both $t=\frac{2j-1}{2^n}\notin U$ and $t'=\frac{j}{2^{n-1}}\notin U$. It follows that
\begin{eqnarray*}
r\circ L_{n,j} &=& (r\circ \delta_t)\cdot (r\circ\ell_{n+1,2j})\cdot (r\circ \delta_{t'})^{-}\\
& \equiv&  \uinf '|_{[0,t]}\cdot \uinf '|_{[t,t']}\cdot \uinf '|_{[0,t']}^{-}
\end{eqnarray*}
is null-homotopic. Thus $r_{\#}(d_{n,j})=1$.

Case II: Suppose $\ell_{n+1,2j}$ has image strictly under the arc $\ell_{m,k}$ where $(m,k)\in A$ and $\ell_{n+1,2j}(1)\neq \ell_{m,k}(1)$. Then both $t=\frac{2j-1}{2^n}$ and $t'=\frac{j}{2^{n-1}}$ lie in $I_{m,k}=(a,b)$. It follows that
\begin{eqnarray*}
r\circ L_{n,j} &=& (r\circ \delta_t)\cdot (r\circ\ell_{n+1,2j})\cdot (r\circ \delta_{t'})^{-} \\
&\equiv& (\uinf '|_{[0,a]}\cdot \linf|_{[a,t]})\cdot(\linf|_{[t,t']})\cdot (\uinf '|_{[0,a]}\cdot \linf|_{[a,t']})^{-}
\end{eqnarray*}
is null-homotopic. Thus $r_{\#}(d_{n,j})=1$.

Case III: Suppose $\ell_{n+1,2j}$ has image strictly under the arc $\ell_{m,k}$ where $(m,k)\in A$ and $\ell_{n+1,2j}(1)= \ell_{m,k}(1)$. Then $t=\frac{2j-1}{2^n}\in I_{m,k}=(a,b)$ and $b=\frac{j}{2^{n-1}}\notin (a,b)$. Let $(c,d)\in \mathcal{I}$ such that $s([c,d])=[a,b]$. It follows that
\begin{eqnarray*}
r\circ L_{n,j} &=& (r\circ \delta_t)\cdot (r\circ\ell_{n+1,2j})\cdot (r\circ \delta_{b})^{-} \\
&\equiv& (\uinf '|_{[0,a]}\cdot \linf|_{[a,t]})\cdot(\linf|_{[t,b]})\cdot (\uinf '|_{[0,b]})^{-}\\
&\equiv& s\circ (\uinf|_{[0,c]}\cdot \linf|_{[c,d]}\cdot \uinf|_{[0,d]}^{-})
\end{eqnarray*}
where $[\uinf|_{[0,c]}\cdot \linf|_{[c,d]}\cdot \uinf|_{[0,d]}^{-}]$ is the inverse of a generator of $\bbk$. Thus $r_{\#}(d_{n,j})\in s_{\#}(\bbk)$.

(2) Define a map $f:(\bbb,d_0)\to (\bbhp,\bpp)$ so that $f\circ \uinf|_{[0,2/3]}\equiv f\circ \linf|_{[0,2/3]}\equiv\iota$, $f\circ \uinf|_{[2/3,1]}\equiv\ell_{\tau}$ and $f\circ \linf|_{[2/3,1]}$ is constant at $b_0$. Since $f_{\#}(\bbk)\leq \cinfty$ and $f_{\#}(w_{\infty})=c_{\tau}$, we may apply Remark \ref{comparisonpropremark}.

(3) Suppose $H$ is a $(\bbk,w_{\infty})$-closed, normal subgroup of $\pionex$ and $f:(\bbhp,b_{0}^{+})\to (X,x_0)$ is a map such that $f_{\#}(\pinfty)\leq H$. Let $\alpha=f\circ \iota$ and recall that $H^{\alpha}$ is $(\bbk,w_{\infty})$-closed. Define $g:(\bbb,d_0)\to (X,f(b_0))$ so that $g(t,0)=f(b_0)$ if $t\in\mcc$, $g\circ \linf=f\circ f_{odd}\circ \ell_{\tau}$, and $g\circ \uinf=f\circ f_{even}\circ \ell_{\tau}$. We have $g_{\#}([\uinf|_{[a,b]}\cdot \linf|_{[a,b]}^{-}])\in H^{\alpha}$ for each $(a,b)\in \mathcal{I}$. Since $H^{\alpha}$ is normal, $g_{\#}(\bbk)\leq H^{\alpha}$. By assumption, we now have $g_{\#}(w_{\infty})\in H^{\alpha}$. Thus $f_{\#}(p_{\tau})=[\alpha]g_{\#}(w_{\infty})[\alpha^{-}]\in H$.
\end{proof}
\begin{figure}[H]
\centering \includegraphics[height=1.7in]{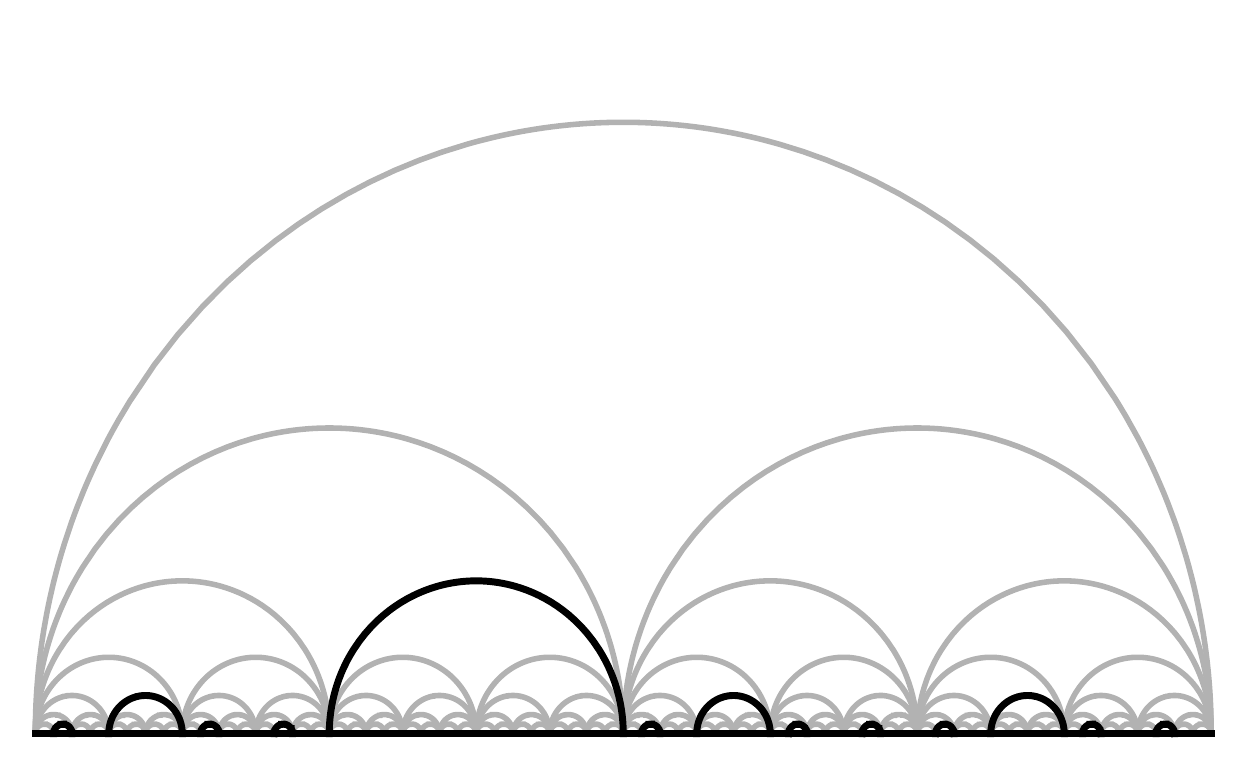}
\caption{\label{embedding}The subspace $\bbb '$ of $\bbd$.}
\end{figure}
We prove a partial converse to (1) of Proposition \ref{transfinitepathproductcomparisonprop}. To do so we require two technical lemmas. Let $\wild(X)\subseteq X$ denote the ``wild" subspace of points at which $X$ is not semilocally simply-connected.
\begin{proposition}\label{wildpullbackprop}
If $X$ is any space, $Y$ is locally path connected and $f:Y\to X$ is continuous, then $f^{-1}(\wild(X))$ is closed in $Y$. In particular, if $X$ is locally path connected, then $\wild(X)$ is closed in $X$.
\end{proposition}
\begin{proof}
Pick $y\notin f^{-1}(\wild(X))$. There is an open neighborhood $U$ of $f(y)$ such the inclusion $U\to X$ induces the trivial homomorphism $\pi_1(U,f(y))\to \pi_1(X,f(y))$. If $C$ is the path component of $f(y)$ in $U$, then the inclusion $C\to X$ also induces the trivial homomorphism $\pi_1(C,f(y))\to \pi_1(X,f(y))$. Find a path-connected neighborhood $V$ of $y$ such that $f(V)\subseteq U$. If $v\in V$, then $f(v)\in C$. Therefore the inclusion $U\to X$ induces the trivial homomorphism $\pi_1(U,f(v))\to \pi_1(X,f(v))$. Thus $V\cap f^{-1}(\wild(X))=\emptyset$.
\end{proof}
\begin{lemma}\label{transfiniteprep}
Let $H\leq\pionex$ be a subgroup.
\begin{enumerate}
\item If $H$ is $(\cinfty,c_{\infty})$-closed, $\alpha\in P(X,x_0)$, and $f:\bbd\to X$ is a map such that $f(B)=\alpha(1)$ and $f_{\#}(S)\leq H^{\alpha}$, then $f_{\#}(\dinf)\in H^{\alpha}$.
\item If $\alpha\in P(X,x_0)$ and $f:(\bbd,d_0)\to (X,\alpha(1))$ is a map such that $f(B)\subseteq X\backslash\wild(X)$ and $f_{\#}(S)\leq H^{\alpha}$, then $f_{\#}(\dinf)\in H^{\alpha}$.
\end{enumerate}
\end{lemma}
\begin{proof}
(1) Since $f$ maps $B$ to a point, the loops $f\circ \lambda_n$ are null at $\alpha(1)$. Define a map $f':(\bbhp,\bpp)\to (X,x_0)$ such that $f'\circ\iota=\alpha$ and $f'\circ\ell_n=f\circ (\lambda_n\cdot \lambda_{n+1}^{-})$ for $n\in\bbn$. Since $(f')_{\#}(C)^{\alpha}=(f)_{\#}(S)\leq H^{\alpha}$, we have $(f')_{\#}(C)\leq H$ and since $H$ is $(\cinfty,c_{\infty})$-closed, we have $(f')_{\#}(c_{\infty})\in H$. But \[(f')_{\#}(c_{\infty})=[\alpha]\left[\left(\prod_{n=1}^{\infty}(f\circ(\lambda_n\cdot \lambda_{n+1}^{-}))\right)\right][\alpha^{-}]=[\alpha\cdot(f\circ\lambda_1)\cdot\alpha^{-}]=[\alpha]f_{\#}(\dinf)[\alpha^{-}].\] Thus $f_{\#}(\dinf)\in H^{\alpha}$.

(2) For every $0\leq t\leq 1$, choose an open neighborhood $U_t$ of $f(t,0)$ such that every loop in the path component of $f(t,0)$ in $U_t$ is null-homotopic in $X$. Find a path-connected open set $W_t$ in $\bbd$ such that $(t,0)\in W_t\subseteq f^{-1}(U_t)$. Recall from Proposition \ref{largeextension} that $\bbd_{n,j}$ is the homeomorphic copy of $\bbd$ beneath the arc $\ell_{n,j}$. There exists $n\in\bbn$ such that for each $j=1,2,\dots,2^{n-1}$, we have $\bbd_{n,j}\subseteq W_{t_j}$ for some $t_j$. Note that $f(\bbd_{n,j})$ and $f(t_j,0)$ lie in the same path component of $U_{t_j}$ so that if $\zeta_{j}$ is a loop traversing the outer curve of $\bbd_{n,j}$, then $f\circ \zeta_{j}$ is null-homotopic in $X$. It follows that $f\circ (\lambda_{n}\cdot\lambda_{\infty}^{-})$ is null-homotopic in $X$. Thus $f_{\#}(\dinf)=[f\circ (\lambda_1\cdot\lambda_{n}^{-})][f\circ (\lambda_{n}\cdot\lambda_{\infty}^{-})]=[f\circ (\lambda_1\cdot\lambda_{n}^{-})]\in f_{\#}(S)\leq H^{\alpha}$.
\end{proof}
\begin{lemma}\label{technicallemma}
Suppose $N\trianglelefteq \pionex$ is a normal $(\bbk,w_{\infty})$-closed subgroup. Let $f:(\bbd,d_0)\to (X,x_0)$ be a map such that $f_{\#}(\bbf)\leq N$. Let $\mathcal{J}\subseteq\mathscr{D}$ be a collection of dyadic unital pairs $(n,j)$ so that the corresponding intervals $I_{n,j}=\left(\frac{j-1}{2^{n-1}},\frac{j}{2^{n-1}}\right)$ satisfy:
\begin{enumerate}
\item if $(n_1,j_1),(n_2,j_2)\in\mathcal{J}$ and $(n_1,j_1)\neq(n_2,j_2)$, then $I_{n_1,j_1}\cap I_{n_2,j_2}=\emptyset$,
\item $U=\bigcup_{(n,j)\in\mathcal{J}}I_{n,j}$ is dense in $\ui$.
\end{enumerate}
Let $s:\ui\to\bbd$ be the path defined as $s(t)=(t,0)$ if $t\in \ui\backslash U$ and $s|_{\overline{I_{n,j}}}\equiv\ell_{n,j}$ if $(n,j)\in \mathcal{J}$. Then $f_{\#}([\lambda_1\cdot s^{-}])\in N$.
\end{lemma}
\begin{proof}
The lemma is clear if $\mathcal{J}$ is finite. Assume $\mathcal{J}$ is infinite. Then $(1,1)\notin\mathcal{J}$. We define a path $\gamma:\ui\to \bbd$, which is homotopic to $\lambda_1$. If $t\in \ui\backslash U$, set $\gamma(t)=(t,0)$. On the intervals $\overline{I_{n,j}}$, $(n,j)\in \mathcal{J}$, we define $\gamma$ to be a path $\gamma_{n,j}$, which is a finite concatenation of standard paths from $\ell_{n,j}(0)$ to $\ell_{n,j}(1)$, using the following inductive procedure:

First, if $a=\frac{j}{2^{n-1}}\leq\frac{j'}{2^{n-1}}=b$ are dyadic rationals, let $\Lambda_n(a,b)$ denote the arc $\prod_{i=j+1}^{j'}\ell_{n,i}$ on the $n$-th level from $(a,0)$ to $(b,0)$. In the case that $a=b$, $\Lambda_n(a,b)$ is the constant path.

To begin the induction, put $J_1=\{(1,1)\}$. Inductively, assume that $J_q$ has been defined as a nonempty, finite set of dyadic unital pairs disjoint from $\mathcal{J}$. Let $(m,p)$ be the smallest (in the dictionary order) element of $J_q$. By Assumption (2), there is a minimal $n>m$ such that there is a $j$ with the property that $(n,j) \in\mathcal{J}$ and $I_{n,j} \subset I_{m,p}$; here, $I_{m,p}$ is defined analogous to $I_{n,j}$. Let $j_1<j_2<\cdots<j_r$ be the complete list of all such $j$'s. Let $k_1<k_2<\cdots <k_u$ be the complete (but possibly empty) list of all $k$'s such that $(n,k_i) \notin \mathcal{J}$ and $\overline{ \bigcup_i I_{n,k_i} \cup \bigcup_i I_{n,j_i}}=\overline{I_{m,p}}$. Define \[\gamma_{n,j_1}=\left(\Lambda_{n}\left(\frac{p-1}{2^{m-1}},\frac{j_1-1}{2^{n-1}}\right)\right)^{-}\cdot \ell_{m,p}\cdot \left(\Lambda_{n}\left(\frac{j_1}{2^{n-1}},\frac{p}{2^{m-1}}\right)\right)^{-}\]
(see Figure \ref{inductionpic2}). For $i>1$, define $\gamma_{n,j_i}=\ell_{n,j_i}$. Define $J_{q+1}$ from $J_q$ by removing $(m,p)$ and adding $(n,k_1), (n,k_2),\dots, (n,k_u)$. Since $\mathcal{J}$ is infinite, $J_{q+1}$ is guaranteed to be nonempty even if no $k_i$ exist. This completes the induction.

By Assumption (1), $\gamma_{n,j}$ has now been defined for every $(n,j) \in \mathcal{J}$. Hence, $\gamma$ has been defined. Note that $\gamma$ is uniformly continuous, because for every $\epsilon>0$, only finitely many $\bbd_{m,p}$ have diameter $> \epsilon$.

Consider the retraction $r_n : \bbd\to E_n$ from the introduction to Section 4. Observe that for $(m,p) \in J_q$ as in the above induction, we have $r_m(\gamma( I_{m,p}\backslash I_{n,j_1} )) \subseteq B$. Since the corresponding statement also applies to the elements $(n,k_i) \in J_{q+1}$, we have that $r_n\circ\gamma$ restricted to $\overline{I_{m,p}}$ is homotopic to $\ell_{m,p}$. We conclude that, for all $n\in\bbn$, $r_n\circ\gamma$ is homotopic to $r_n\circ\lambda_1=\lambda_1$ and thus $\gamma\simeq \lambda_1$.

Since $[\gamma_{n,j}\cdot \ell_{n,j}^{-}]\in\bbf$ for each $(n,j)\in\mathcal{J}$ and $N$ is normal, $\ds f_{\#}\left([\alpha\cdot\gamma_{n,j}\cdot \ell_{n,j}^{-}\cdot\alpha^{-}]\right)\in N$ for every path $\alpha:\ui\to \bbd$ from $d_0$ to $\ell_{n,j}(0)$. Therefore, the paths $s$ and $\gamma$ agree on $\ui\backslash U$ and for each component $(a,b)=I_{n,j}$ of $U$, we have
 \[
f_{\#}\left([\gamma|_{[0,b]}\cdot s|_{[a,b]}^{-}\cdot\gamma|_{[0,a]}^{-}]\right)=
\ds f_{\#}\left(\left[\gamma|_{[0,a]}\cdot\left(\gamma_{n,j}\cdot \ell_{n,j}^{-}\right)\cdot\gamma|_{[0,a]}^{-}\right]\right)\in N.\]Since $X$ is assumed to have transfinite path products relative to $N$, we conclude that $f_{\#}([\lambda_1\cdot s^{-}])=f_{\#}([\gamma\cdot s^{-}])\in N$.
\end{proof}
\begin{figure}[H]
\centering \includegraphics[height=1.5in]{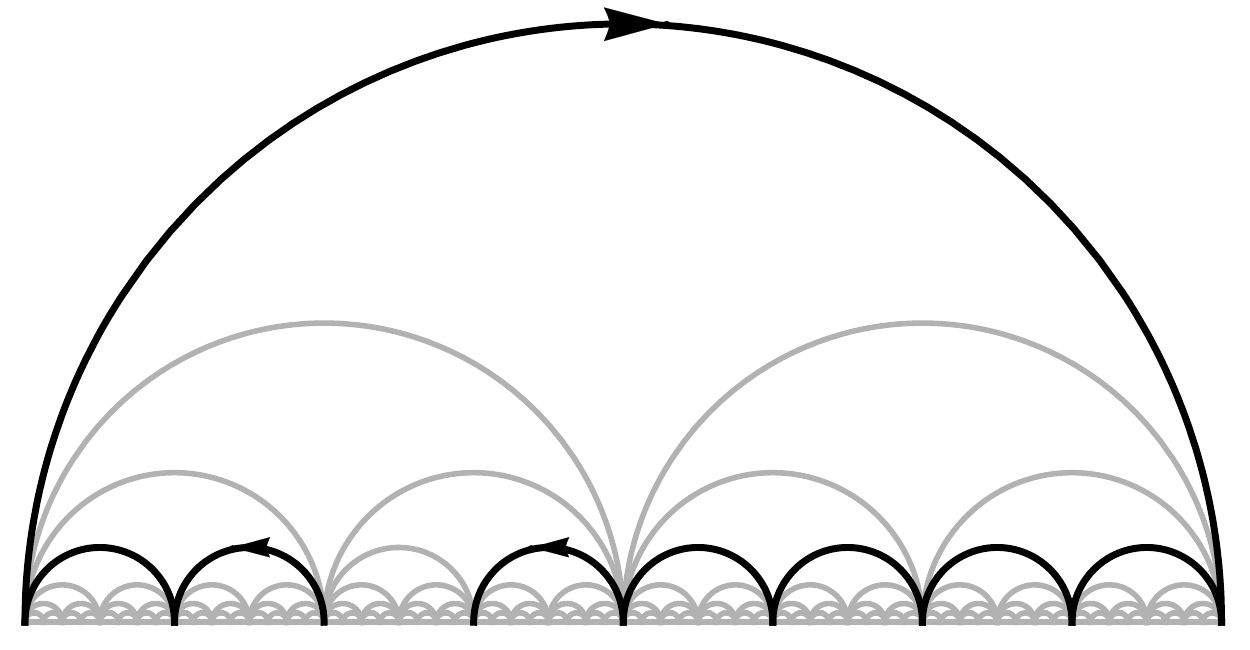}
\caption{\label{inductionpic2} An example of the definition of $\gamma_{n,j_1}$ in $\bbd_{m,p}$ when $n=m+3$ and $j_1=2^{n-m}(p-1)+3$.}
\end{figure}
\begin{theorem}\label{tpdtheorem}
Suppose $\wild(X)$ is totally path disconnected and $N\trianglelefteq \pionex$ is a normal subgroup. Then $N$ is $(\bbf,\dinf)$-closed if and only if $N$ is $(W,w_{\infty})$-closed. In particular, the closure operators $cl_{D,d_{\infty}}$ and $cl_{W,w_{\infty}}$ agree on the normal subgroups of $\pionex$.
\end{theorem}
\begin{proof}
One direction follows from Proposition \ref{transfinitepathproductcomparisonprop}. For the other direction, suppose $X$ has transfinite path products relative to $N$ and $f:(\bbd,d_0)\to (X,x_0)$ is a map such that $f_{\#}(\bbf)\leq N$. If $f(B)=x_0$ or $f(B)\subseteq X\backslash \wild(X)$, then we may apply Lemma \ref{transfiniteprep} (recall that $S\leq D$) to conclude that $f_{\#}(\dinf)\in N$. Thus we may assume that $f|_{B}$ is nonconstant and has image intersecting $\wild(X)$. By Proposition \ref{wildpullbackprop}, $Y=f^{-1}(X\backslash \wild(X))\cap ((0,1)\times \{0\})$ is open in $(0,1)\times \{0\}$. Let $Z$ be the (possibly empty) interior of $((0,1)\times \{0\})\backslash Y$ in $(0,1)\times \{0\}$. Note that $V=Y\cup Z$ is open and dense in $B$.

Since $\wild(X)$ is totally path disconnected and $f(Z)\subseteq \wild(X)$, each connected component of $Z$ must be mapped by $f$ to a single point. Let $\mathcal{J}_{Y}$ be a collection of dyadic unital pairs $(n,j)$ such that the union of the corresponding intervals $I_{n,j}=\left(\frac{j-1}{2^{n-1}},\frac{j}{2^{n-1}}\right)$ are disjoint and dense in $Y$. Similarly, let $\mathcal{J}_{Z}$ be a (possibly empty) collection of dyadic unital pairs $(n,j)$ such that the union of the corresponding intervals $I_{n,j}$ are disjoint and dense in $Z$. Let $\mathcal{J}= \mathcal{J}_{Y}\cup \mathcal{J}_{Z}$ and $U=\bigcup_{(n,j)\in\mathcal{J}}I_{n,j}\subseteq V$. Note that both Conditions (1) and (2) in Lemma \ref{technicallemma} are satisfied so the path $s:\ui\to \bbd$ as defined in the statement of the Lemma satisfies $f_{\#}([\lambda_1\cdot s^{-}])\in N$.

Fix a dyadic unital pair $(n,j)\in \mathcal{J}$ and put $(a,b)=I_{n,j}$. Consider $\bbd_{n,j}\subseteq \bbd$ and the homeomorphism $T_{n,j}:\bbd_{n,j}\to\bbd$ as defined in Proposition \ref{largeextension}. Set $f_{n,j}=f\circ T_{n,j}^{-1}$. Let $\beta_{n,j}:\ui\to\bbd$ be the path which is the restriction of $\lambda_{\infty}$ to $\overline{I_{n,j}}$. Since $N$ is normal and $f_{\#}(\bbf)\leq N$, we have $(f_{n,j})_{\#}(\bbs)\leq (f_{n,j})_{\#}(\bbf)\leq N^{\alpha}$ for every path $\alpha:\ui\to X$ from $x_0$ to $f(a,0)$. Notice that
\begin{enumerate}
\item $f_{n,j}(B)\subseteq X\backslash \wild(X)$ if $(n,j)\in \mathcal{J}_{Y}$,
\item $f_{n,j}(B)$ is a single point if $(n,j)\in \mathcal{J}_{Z}$.
\end{enumerate}
In either case, we may apply Lemma \ref{transfiniteprep} to see that $(f_{n,j})_{\#}(d_{\infty})=[f\circ(\ell_{n,j}\cdot\beta_{n,j}^{-})]\in N^{\alpha}$ for $\alpha=f\circ s|_{[0,a]}$.

By construction, the paths $s$ and $\lambda_{\infty}$ agree on $\ui\backslash U$. Moreover, for every component $(a,b)=I_{n,j}$ of $U$, we have \[f_{\#}([s|_{[0,b]}\cdot\lambda_{\infty}|_{[a,b]}^{-}\cdot s|_{[0,a]}^{-}])=f_{\#}([s|_{[0,a]}\cdot\ell_{n,j}\cdot\beta_{n,j}^{-}\cdot s|_{[0,a]}^{-}])\in N.\]
Since $X$ is assumed to have transfinite products relative to $N$, we conclude that $f_{\#}([s\circ \lambda_{\infty}^{-}])\in N$. Thus $f_{\#}(\dinf)=f_{\#}([\lambda_1\cdot s^{-}])f_{\#}([s\cdot\lambda_{\infty}^{-}])\in N$.
The last statement of the theorem follows from Corollary \ref{agreeonnormalsubgroupscorollary}.
\end{proof}
We conclude this paper by considering spaces $X$ with a discrete wild set $\wild(X)$.
\begin{lemma}\label{wildendsonlylemma}
Suppose $H\leq\pi_1(X,x_0)$ is $(C,c_{\infty})$-closed, $\gamma\in P(X,x_0)$, and $A$ is a nowhere dense closed subset of $\ui$ containing $\{0,1\}$. If $\alpha,\beta\in P(X,\gamma(1))$ are paths such that
\begin{enumerate}
\item $\alpha|_{A}=\beta|_{A}$,
\item $[\alpha|_{[0,b]}\cdot\beta|_{[a,b]}^{-}\cdot\alpha|_{[0,a]}^{-}]\in H^{\gamma}$ for all components $(a,b)$ of $\ui\backslash A$,
\item $\alpha((0,1))\cap \wild(X)=\emptyset$ and $\beta((0,1))\cap \wild(X)=\emptyset$,
\end{enumerate}
then $[\alpha\cdot\beta^{-}]\in H^{\gamma}$.
\end{lemma}
\begin{proof}
By the proof of Proposition \ref{transfinitepathprodchar}, we may reparameterize $\alpha,\beta$ and construct a map $f:\bbb\to X$ such that $f\circ \uinf= \alpha$, $f\circ \linf=\beta$ and $f_{\#}(W)\leq H^{\gamma}$. It suffices to show that $[\alpha\cdot\beta^{-}]=f_{\#}(w_{\infty})\in H^{\gamma}$. Let $x_1=f(d_0)$ and $x_2=f(1,0)$. We consider four cases.

Case I: Suppose $x_1\notin \wild(X)$ and $x_2\notin \wild(X)$. Then $f(\bbb)\subseteq X\backslash \wild(X)$. Recall from the proof of Proposition \ref{transfinitepathproductcomparisonprop} (1) that there is a map $g:(\bbd,d_0)\to(\bbb,d_0)$ such that $g|_{B}=id_{B}$, $g_{\#}(D)\leq W$ and $g_{\#}(d_{\infty})=w_{\infty}$. Therefore $f\circ g:(\bbd,d_0)\to (X,\alpha(1))$ is a map such that $f\circ g(B)\subseteq X\backslash\wild(X)$ and $(f\circ g)_{\#}(S)\leq (f\circ g)_{\#}(D)\leq f_{\#}(W)\leq  H^{\gamma}$. By (2) of Lemma \ref{transfiniteprep}, we have $f_{\#}(w_{\infty})=(f\circ g)_{\#}(\dinf)\in H^{\gamma}$.

Case II: Suppose $x_1\in \wild(X)$ and $x_2\notin \wild(X)$. Define $(t_{n})_{n\in\bbn}$ to be the sequence in the Cantor set $\mcc$ given by $t_{2m-1}=\frac{1}{3^{m-1}}$ and $t_{2m}=\frac{2}{3^{m}}$. Note that $\{(x,y)\in \bbb\mid t_{n+1}\leq x\leq t_n\}$ is homeomorphic to either $\bbb$ when $n$ is odd or $S^1$ when $n$ is even. Since $\alpha((0,1])\cap \wild(X)=\emptyset$ and $\beta((0,1])\cap \wild(X)=\emptyset$, we have $[\alpha|_{[0,t_{n}]}\cdot \beta|_{[t_{n+1},t_n]}^{-}\cdot \alpha|_{[0,t_{n+1}]}^{-}]\in H^{\gamma}$ by assumption when $n$ is even and $[\alpha|_{[t_{n+1},t_{n}]}\cdot \beta|_{[t_{n+1},t_n]}^{-}]\in H^{\gamma\cdot\alpha|_{[0,t_{n+1}]}}$ by Case I when $n$ is odd. Thus $[\alpha|_{[0,t_{n}]}\cdot \beta|_{[t_{n+1},t_n]}^{-}\cdot \alpha|_{[0,t_{n+1}]}^{-}]\in H^{\gamma}$ for all $n\in \bbn$. Define a map $k:\bbhp\to X$ by $k\circ \iota=\gamma$, and $k\circ \ell_{n}=\alpha|_{[0,t_{n}]}\cdot \beta|_{[t_{n+1},t_n]}^{-}\cdot \alpha|_{[0,t_{n+1}]}^{-}$. Since $k_{\#}(C)\leq H$ and $H$ is $(C,c_{\infty})$-closed, we have $k_{\#}(c_{\infty})\in H$. Thus $[k\circ \ell_{\infty}]\in H^{\gamma}$. However,
\begin{eqnarray*}
[k\circ \ell_{\infty}] &=& \left[\prod_{n=1}^{\infty}\left(\alpha|_{[0,t_{n}]}\cdot \beta|_{[t_{n+1},t_n]}^{-}\cdot \alpha|_{[0,t_{n+1}]}^{-}\right)\right]\\
&=& \left[\alpha|_{[0,t_{1}]}\right] \left[\prod_{n=1}^{\infty}\left(\beta|_{[t_{n+1},t_n]}^{-}\cdot \alpha|_{[0,t_{n+1}]}^{-}\cdot \alpha|_{[0,t_{n+1}]}\right)\right]\\
&=&\left[\alpha|_{[0,t_{1}]}\right]\left[\prod_{n=1}^{\infty}\beta|_{[t_{n+1},t_n]}^{-}\right]\\
 &=& \left[\alpha\cdot\beta^{-}\right]
\end{eqnarray*}

Case III: Suppose $x_1\notin \wild(X)$ and $x_2\in \wild(X)$. Define $f':\bbb\to X$ by $f'(x,y)=f(1-x,y)$. Since $(f')_{\#}(W)\leq H^{\gamma\cdot\alpha}$, we may use Case II to conclude that $(f')_{\#}(w_{\infty})=[\alpha^{-}\cdot\beta]\in H^{\gamma\cdot\alpha}$. Conjugating by $[\alpha]$ and inverting gives $[\alpha\cdot\beta^{-}]\in H^{\gamma}$.

Case IV: Suppose $x_1\in \wild(X)$ and $x_2\in \wild(X)$. Define maps $f_1,f_2:\bbb\to X$ so that $f_1\circ \uinf\equiv\alpha|_{[0,1/3]}$, $f_1\circ \linf\equiv\beta|_{[0,1/3]}$, $f_2\circ \uinf\equiv\alpha|_{[2/3,1]}$, and $f_2\circ \linf\equiv\beta|_{[2/3,1]}$. Applying Case II to $f_1$, we see that $[\alpha|_{[0,1/3]}\cdot \beta|_{[0,1/3]}^{-}]\in H^{\gamma}$. Applying Case III to $f_2$, we see that $[\alpha|_{[2/3,1]}\cdot \beta|_{[2/3,1]}^{-}]\in H^{\gamma\cdot\alpha|_{[0,2/3]}}$. Thus $[\alpha\cdot \beta|_{[2/3,1]}^{-}\cdot\alpha|_{[0,2/3]}^{-}]\in H^{\gamma}$. By assumption, we have $[\alpha|_{[0,2/3]}\cdot\beta|_{[1/3,2/3]}^{-}\cdot\alpha|_{[0,1/3]}^{-}]\in H^{\gamma}$. It follows that
\[
[\alpha\cdot\beta^{-}] = [\alpha\cdot\beta|_{[2/3,1]}^{-}\cdot\alpha|_{[0,2/3]}^{-}][\alpha|_{[0,2/3]}\cdot \beta|_{[1/3,2/3]}^{-}\cdot \alpha|_{[0,1/3]}^{-}][\alpha|_{[0,1/3]}\cdot \beta|_{[0,1/3]}^{-}]\in H^{\gamma}.\]
\end{proof}
\begin{lemma}\label{wildlooplemma}
Suppose $N\trianglelefteq\pi_1(X,x_0)$ is a normal $(P,p_{\tau})$-closed subgroup, $\gamma\in P(X,x_0)$, and $A$ is a nowhere dense closed subset of $\ui$ containing $\{0,1\}$. If $\alpha,\beta\in P(X,\gamma(1))$ are paths such that
\begin{enumerate}
\item $\alpha|_{A}=\beta|_{A}$,
\item $[\alpha|_{[0,b]}\cdot\beta|_{[a,b]}^{-}\cdot\alpha|_{[0,a]}^{-}]\in N^{\gamma}$ for every component $(a,b)$ of $\ui\backslash A$,
\item there is a point $x_1\in X$ such that $x_1\in\alpha([a,b])\cup \beta([a,b])$ for all components $(a,b)$ of $\ui\backslash A$,
\end{enumerate}
then $[\alpha\cdot\beta^{-}]\in N^{\gamma}$.
\end{lemma}
\begin{proof}
Let $\mathscr{C}$ denote the set of components of $\ui\backslash A$. First, we note that if $\mathscr{C}$ is finite, then the conclusion is clear since, in this case, $[\alpha\cdot\beta^{-}]$ factors as a product of the elements $[\alpha|_{[0,b]}\cdot\beta|_{[a,b]}^{-}\cdot\alpha|_{[0,a]}^{-}]\in N^{\gamma}$, $(a,b)\in\mathscr{C}$. Therefore, we assume $A$ and $\mathscr{C}$ are infinite. It follows from assumption (3) that if $a$ is a limit point of $A$, then there exists a sequence $t_n\to a$ in $\ui$ such that either $\alpha(t_n)=x_1$ or $\beta(t_n)=x_1$ for all $n$. Hence $\alpha(a)=x_1$ for all limit points $a$ of $A$. Additionally, since $A$ is compact, if $U$ is any open neighborhood of $x_1$, then we must have $\alpha([a,b])\cup\beta([a,b])\subseteq U$ for all but finitely many $(a,b)\in\mathscr{C}$.

For each point $a\in A$, we define a path $\rho_{a}$ from $x_1$ to $\alpha(a)=\beta(a)$. If $a$ is a limit of point of $A$, let $\rho_{a}:\{a\}\to X$ be the degenerate constant path at $x_1$. If $a\in A\cap [0,1)$ is an isolated point, there is a $b\in A$ such that $(a,b)\in\mathscr{C}$. If there exists a smallest $s\in [a,b]$ such that $\alpha(s)=x_1$, define $\rho_{a}\equiv \alpha|_{[a,s]}^{-}$. If no such $s$ exists, then there exists a smallest $s\in [a,b]$ such that $\beta(s)=x_1$; in this case set $\rho_{a}\equiv \beta|_{[a,s]}^{-}$. If $1$ is an isolated point of $A$, take $\rho_{1}$ to be any path from $x_1$ to $\alpha(1)=\beta(1)$. Define loops $L,M:\ui\to X$ so that $L(A)=M(A)=x_1$ and if $(a,b)$ is a component of $\ui\backslash A$, set $L|_{[a,b]}\equiv \rho_a\cdot \alpha|_{[a,b]}\cdot \rho_{b}^{-}$ and $M|_{[a,b]}\equiv \rho_a\cdot \beta|_{[a,b]}\cdot \rho_{b}^{-}$. Note that for any open neighborhood $U$ of $x_1$, all but finitely many $L|_{[a,b]}$ and $M|_{[a,b]}$ have image in $U$. It follows that $L$ and $M$ are continuous.

The loops $L$ and $M$ are constructed from $\alpha$ and $\beta$ respectively by inserting an at most countably infinite number of loops $\rho_{a}\cdot\rho_{a}^{-}$ at the isolated points $a\in A$ with $0<a<1$, prepending $\rho_0$ and appending $\rho_{1}^{-}$. Each such loop contracts in it's own image. Since all but finitely many of these contractions lie within a given neighborhood of $x_1$, there are homotopies $L\simeq \rho_{0}\cdot\alpha\cdot\rho_{1}^{-}$ and $M\simeq \rho_{0}\cdot\beta\cdot\rho_{1}^{-}$. Thus $[\alpha\cdot\beta^{-}]=[\rho_{0}^{-}\cdot L \cdot M ^{-}\cdot\rho_0]$. We now seek to show $[L\cdot M^{-}]\in N^{\gamma\cdot\rho_{0}^{-}}$.

Fix a component $(a,b)$ of $\ui\backslash A$ and recall $[\alpha|_{[0,b]}\cdot\beta|_{[a,b]}^{-}\cdot\alpha|_{[0,a]}^{-}]\in N^{\gamma}$.  Conjugating by $[\alpha|_{[0,a]}^{-}]$ gives $[\alpha|_{[a,b]}\cdot \beta|_{[a,b]}^{-}]=[\alpha|_{[0,a]}^{-}\cdot\alpha|_{[0,b]}\cdot\beta|_{[a,b]}^{-}\cdot\alpha|_{[0,a]}^{-}\cdot\alpha|_{[0,a]}] \in N^{\gamma\cdot\alpha|_{[0,a]}}$. Thus \[[L|_{[a,b]}\cdot M|_{[a,b]}^{-}]=[\rho_a\cdot\alpha|_{[a,b]}\cdot \beta|_{[a,b]}^{-}\cdot\rho_{a}^{-}]  \in   N^{\gamma\cdot\alpha|_{[0,a]}\cdot \rho_{a}^{-}}=N^{\gamma\cdot \rho_{0}^{-}}\]where the equality $N^{\gamma\cdot\alpha|_{[0,a]}\cdot \rho_{a}^{-}}=N^{\gamma\cdot \rho_{0}^{-}}$ follows from the normality of $N$.

Enumerate the infinitely many components of $\ui\backslash A$ as $(a_1,b_1),(a_2,b_2),(a_3,b_3),\dots$. The loop $L$ induces a map $f_{\alpha}:\bbhp\to X$ such that $f_{\alpha}\circ\iota=\gamma\cdot\rho_{0}^{-}$ and $f|_{\alpha}\circ \ell_{n}\equiv L|_{[a_n,b_n]}$. Similarly, $M$ induces a map $f_{\beta}:\bbhp\to X$ such that $f_{\beta}\circ\iota=\gamma\cdot\rho_{0}^{-}$ and $f|_{\beta}\circ \ell_{n}\equiv M|_{[a_n,b_n]}$. Note that there is a loop $\zeta$ in $\bbh$ based at $b_0$ such that $f_{\alpha}\circ \zeta\equiv L$ and $f_{\beta}\circ \zeta\equiv M$. Since $\bbh$ is assumed to have transfinite products relative to $N$ (recall Proposition \ref{transfiniteproductprop}) and $(f_{\alpha})_{\#}(c_n)(f_{\beta})_{\#}(c_n)^{-1}=[\gamma\cdot\rho_{0}^{-}]\left[L|_{[a_n,b_n]}\cdot M|_{[a_n,b_n]}^{-}\right][\rho_{0}\cdot\gamma^{-}]\in N$ for each $n\in\bbn$, we have $[\gamma\cdot\rho_{0}^{-}][ L \cdot M ^{-}][\rho_0\cdot\gamma^{-}]=(f_{\alpha})_{\#}([\iota\cdot\zeta\cdot\iota]) (f_{\beta})_{\#}([\iota\cdot\zeta\cdot\iota])^{-1}\in N$, completing the proof.
\end{proof}
\begin{theorem}\label{discretewildsettheorem}
Suppose $\wild(X)$ is discrete and $N\trianglelefteq \pionex$ is a normal subgroup. Then $N$ is $(D,d_{\infty})$-closed if and only if $N$ is $(P,p_{\tau})$-closed. In particular, the closure operators $cl_{D,d_{\infty}}$ and $cl_{P,p_{\tau}}$ agree on the normal subgroups of $\pionex$.
\end{theorem}
\begin{corollary}\label{heuplcorollary}
Suppose $X$ is a metric space such that $\wild(X)$ is discrete and $N\trianglelefteq \pionex$ is a normal subgroup. Then $p_N:\wt{X}_{N}\to X$ has the unique path lifting property if and only if $X$ has transfinite products relative to $N$. In particular, $p_{K}:\wt{X}_{K}\to X$ has the unique path lifting property if $K=cl_{P,p_{\tau}}(N)$.
\end{corollary}
\begin{corollary}\label{heuplcorollary2}
Suppose $X$ is a metric space such that $\wild(X)$ is discrete and $N\leq \pionex$ contains the commutator subgroup of $\pionex$. Then $p_N:\wt{X}_{N}\to X$ has the unique path lifting property if and only if $X$ is homotopically Hausdorff relative to $N$. In particular, $p_{K}:\wt{X}_{K}\to X$ has the unique path lifting property if $K=cl_{C,c_{\infty}}(N)$.
\end{corollary}
\begin{proof}[Proof of Theorem \ref{discretewildsettheorem}]
The last statement of the theorem follows from Corollary \ref{agreeonnormalsubgroupscorollary} once the equivalence is proven. One direction follows from Proposition \ref{transfinitepathproductcomparisonprop}. Suppose the normal subgroup $N\trianglelefteq \pionex$ is $(P,p_{\tau})$-closed. By Theorem \ref{tpdtheorem}, it suffices to show $N$ is $(W,w_{\infty})$-closed. Let $f:(\mathbb{W},d_0)\to (X,x_0)$ be a map such that $f_{\#}(W)\leq N$. Denote $\alpha=f\circ \uinf$ and $\beta=f\circ \linf$ and observe $\alpha|_{\mcc}=\beta|_{\mcc}$ where $\mcc$ is the Cantor set. We seek to show that $[\alpha\cdot\beta^{-}]\in N$.

Our claim is trivial if $\wild(X)=\emptyset$. Since $(W,w_{\infty})$ is a normal closure pair, we are free to change the basepoint so that $x_0\in \wild(X)$. If $\alpha(1)=\beta(1)\neq x_0$, find a path $\gamma$ from $x_0$ to $\alpha(1)$. Using the self-similarity of $\bbb$, we may replace $f$ with a map $g:\bbb\to \bbh$ such that $g\circ \uinf|_{[0,1/3]}\equiv \alpha$, $g\circ \linf|_{[0,1/3]}\equiv \beta$ and $g\circ \uinf|_{[1/3,1]}=g\circ \linf|_{[1/3,1]}\equiv\gamma$. Clearly $g_{\#}(\bbk)=f_{\#}(\bbk)$ and $g_{\#}(w_{\infty})=f_{\#}(w_{\infty})$. Therefore, without loss of generality, we may assume $\alpha,\beta$ are loops in $X$ based at $x_0$.

Since $\bbb$ is a Peano continuum, $f^{-1}(\wild(X))$ is closed in $\bbb$ by Proposition \ref{wildpullbackprop} and therefore is compact. The continuous image of a compact set in a discrete space is finite. Therefore, we may list the points of $f(\bbb)\cap \wild(X)$ as the finite set $\Omega=\{x_0,x_1,\dots,x_n\}$. Recall that $\mathcal{I}$ denotes the set of components of $\ui\backslash \mcc$. Let
\begin{itemize}
\item[] $A_0=\{t\in \mcc\mid f(t,0)\in \Omega\}$,
\item[] $A_1=\{a\in \mcc\mid (a,b)\in \mathcal{I}\text{ and }(\alpha([a,b])\cup\beta([a,b]))\cap \Omega\neq\emptyset \}$,
\item[] $A_2=\{b\in \mcc\mid (a,b)\in \mathcal{I}\text{ and }(\alpha([a,b])\cup\beta([a,b]))\cap \Omega\neq\emptyset \}$,
\end{itemize}
and $A=A_0\cup A_1\cup A_2$. Certainly, $A$ is nowhere dense in $\ui$; we check that $A$ is closed in $\ui$ by showing $A$ is closed in $\mcc$. Choose a point $t\in \mcc\backslash A$. Notice that $t\notin\alpha^{-1}(\Omega)\cup\beta^{-1}(\Omega)$. If $t=a$ for some $(a,b)\in\mathcal{I}$, then it must be the case that $(\alpha([a,b])\cup\beta([a,b]))\cap \Omega=\emptyset$. Thus $[a,b]\cap A=\emptyset$. Since $a$ does not lie in the closed set $\alpha^{-1}(\Omega)\cup \beta^{-1}(\Omega)$, there is a $c\in \mcc$ such that $c<a$ and $(c,a]\cap (\alpha^{-1}(\Omega)\cup \beta^{-1}(\Omega))=\emptyset$. Thus $a\in (c,b)$ and by the definition of $A$, we have $(c,b)\cap A=\emptyset$. Similarly, if $t=b$ for some $(a,b)\in\mathcal{I}$, we may find a $c\in \mcc$ with $b<c$ such that $(a,c)\cap A=\emptyset$. Finally, if $t$ is not an endpoint of any element of $\mathcal{I}$, then we may find $c,c'\in \mcc$ with $c<t<c'$ such that $(c,c')\cap (\alpha^{-1}(\Omega)\cup\beta^{-1}(\Omega))=\emptyset$. Again, by the definition of $A$, we have $(c,c')\cap A=\emptyset$, finishing the proof that $A$ is closed.

The open set $\ui\backslash A$ is the disjoint union of open intervals $(r,s)$, each of which is either equal to some $(a,b)\in\mathcal{I}$ or to the union of infinitely many $(a,b)\in\mathcal{I}$ and points $t\in \mcc\backslash A$. We further classify the components $(r,s)$ of $\ui\backslash A$ as follows:
\begin{enumerate}
\item $(r,s)$ is Type I if $(r,s)\in \mathcal{I}$.
\item $(r,s)$ is Type II if $(r,s)\notin \mathcal{I}$ and $\Omega\cap\{\alpha(r),\alpha(s)\}\neq \emptyset$.
\item $(r,s)$ is Type III if $(r,s)\notin \mathcal{I}$ and $\Omega\cap\{\alpha(r),\alpha(s)\}=\emptyset$.
\end{enumerate}

If $(r,s)$ is Type I, then $\Omega\cap (\alpha([r,s])\cup\beta([r,s]))\neq\emptyset$ and $[\alpha|_{[0,s]}\cdot\beta|_{[r,s]}^{-}\cdot\alpha|_{[0,r]}^{-}] \in N$ by assumption. If $(r,s)$ is Type II or III, then $\Omega\cap (\alpha((r,s))\cup\beta((r,s)))=\emptyset$. By applying Lemma \ref{wildendsonlylemma} to the paths $\alpha|_{[r,s]}$ and $\beta|_{[r,s]}$ and $\gamma=\alpha|_{[0,r]}$, we see that $[\alpha|_{[0,s]}\cdot\beta|_{[r,s]}^{-}\cdot\alpha|_{[0,r]}^{-}] \in N$.

A component $(r,s)$ of $\ui\backslash A$ is Type III if and only if $\Omega\cap(\alpha([r,s])\cup\beta([r,s]))=\emptyset$. Hence, if $(r,s)$ is Type III, the definition of $A$ guarantees the existence of $q,t\in A$ such that $(q,r)$ and $(s,t)$ are Type I components and thus $\Omega$ intersects both $\alpha([q,r])\cup\beta([q,r])$ and $\alpha([s,t])\cup\beta([s,t])$. In particular, $r$ and $s$ are isolated points of $A$. Let \[A^{\ast}=A\backslash \{s\mid (r,s)\text{ is a Type III component of }\ui\backslash A\}.\] Since $A^{\ast}$ is constructed by removing only isolated points of $A$, $A^{\ast}$ is closed and nowhere dense. In particular, $\ui\backslash A^{\ast}$ is formed by combining each type III component with a unique Type I component. It follows that for every component $(r,s)$ of $\ui\backslash A^{\ast}$, we have $\Omega\cap(\alpha([r,s])\cup\beta([r,s]))\neq \emptyset$ and $[\alpha|_{[0,s]}\cdot\beta|_{[r,s]}^{-}\cdot\alpha|_{[0,r]}^{-}] \in N$.

Let $\mathscr{C}$ denote the set of components of $\ui\backslash A^{\ast}$ with the natural linear ordering inherited from $\ui$. Recall that a subset $\mathscr{J}\subset \mathscr{C}$ is convex if whenever $I_1,I_2\in \mathscr{J}$ and $I_1<I<I_2$, then $I\in \mathscr{J}$. If $I=(r,s)\in \mathscr{C}$, define the \textit{wild image} of $I$ as the nonempty finite set $wim(I)=\Omega\cap(\alpha([r,s])\cup\beta([r,s]))$. Since $\Omega$ is finite, the continuity of $\alpha$ and $\beta$ guarantee that there can only be finitely many $I\in\mathscr{C}$ such that $|wim(I)|>1 $. Also note that any two points from different sets of $f^{-1}(x_0),f^{-1}(x_1),\dots ,f^{-1}(x_n)$ are separated by a minimum distance. Therefore, we may write $\mathscr{C}$ as a disjoint union of finitely many single-point sets $\{I_1\},\{I_2\},\dots,\{I_m\}$ with $|wim(I_j)|>1$ and finitely many convex sets $\mathscr{J}_1,\mathscr{J}_2,\dots,\mathscr{J}_{m'}\subseteq \mathscr{C}$ such that for each $i\in \{1,2,\dots,{m'}\}$, the convex set $\mathscr{J}_i$ has constant wild image of cardinality $1$, i.e. there exists $x_k\in \Omega$ such that $I,I'\in\mathscr{J}_i$ $\Rightarrow$ $wim(I)=\{x_k\}=wim(I')$. Using this decomposition of $\mathscr{C}$, we may find points $0=p_0<p_1<\cdots<p_n=1$ in $A^{\ast}$ such that for each $j\in\{1,2,\dots,n\}$, either:
\begin{enumerate}
\item $(p_{j-1},p_j)=I_j\in\mathscr{C}$ with $|wim(I_j)|>1$ and thus $[\alpha|_{[0,p_j]}\cdot\beta|_{[p_{j-1},p_j]}^{-}\cdot \alpha|_{[0,p_{j-1}]}^{-}]\in N$ by construction of $A^{\ast}$
\item or $\alpha|_{[p_{j-1},p_j]}$ and $\beta|_{[p_{j-1},p_j]}$ satisfy the conditions of Lemma \ref{wildlooplemma} using the nowhere dense set $A^{\ast}\cap [p_{j-1},p_j]$, path $\gamma=\alpha|_{[0,p_{j-1}]}$, and unique wild point $x_k\in\Omega$ such that $\Omega\cap (\alpha([p_{j-1},p_j])\cup \beta([p_{j-1},p_j]))=\{x_k\}$. By applying this Lemma, we see that $[\alpha|_{[0,p_j]}\cdot\beta|_{[p_{j-1},p_j]}^{-}\cdot \alpha|_{[0,p_{j-1}]}^{-}]\in N$.
\end{enumerate}
Finally, since $[\alpha\cdot\beta^{-}]$ is a product of the elements $[\alpha|_{[0,p_j]}\cdot\beta|_{[p_{j-1},p_j]}^{-}\cdot \alpha|_{[0,p_{j-1}]}^{-}]$, $1\leq j\leq n$, we conclude that $[\alpha\cdot\beta^{-}]\in N$.
\end{proof}
\begin{example}
Since $\wild(\bbh)=\{b_0\}$ is discrete, we may apply Theorem \ref{discretewildsettheorem}. Hence, if $N\trianglelefteq\pioneh$ is a normal subgroup, then $p_N:\wt{\bbh}_{N}\to \bbh$ is a generalized regular covering if and only if $\bbh$ has transfinite products relative to $N$.
\end{example}
\noindent\textbf{Acknowledgements.} This work was partially supported by a grant from the Simons Foundation (\#245042 to Hanspeter Fischer).
\end{document}